\documentclass[11pt,reqno]{amsart}
\usepackage[utf8]{inputenc}
\usepackage[T1]{fontenc}
\usepackage[a4paper, total={160mm,250mm}]{geometry}

\usepackage{amsmath}
\usepackage{amssymb}
\usepackage{amsthm}

\usepackage{mathtools}
\usepackage{aliascnt}

\usepackage{xcolor}
\usepackage{graphicx}
\usepackage{tikz}
\usepackage[normalem]{ulem}
\usepackage{dutchcal}
\usepackage{esint} 
\usepackage{enumitem}

\usepackage[hypertexnames=false]{hyperref}
\usepackage{cleveref}

\usepackage{autonum}
\makeatletter
\autonum@generatePatchedReference{eqref}
\makeatother

\numberwithin{equation}{section}

\colorlet{darkblue}{blue!90!black}
\colorlet{darkgreen}{green!50!black}

\newcommand{\eqdef}{\stackrel{\mathclap{\mbox{\tiny def}}}{=}}
\newcommand{\eqd}{\stackrel{\mathclap{\mbox{\tiny d}}}{=}}

\newcommand{\ic}{\mathbf{1}}
\newcommand{\dd}{\mathrm{d}}

\newcommand{\thalf}{\tfrac{1}{2}}

\usepackage{framed}

\def\one{\mathrm{I}}
\def\two{\mathrm{II}}

\newcommand{\law}{\mathrm{law}}

\newcommand{\supp}{\mathrm{supp}}

\newcommand{\vol}{\mathrm{vol}}

\DeclareMathOperator*{\Leb}{Leb}

\theoremstyle{plain}

\newtheorem{theorem}{Theorem}[section]

\newaliascnt{lemma}{theorem}
\newtheorem{lemma}[lemma]{Lemma}
\aliascntresetthe{lemma}

\newaliascnt{proposition}{theorem}
\newtheorem{proposition}[proposition]{Proposition}
\aliascntresetthe{proposition}

\newaliascnt{corollary}{theorem}
\newtheorem{corollary}[corollary]{Corollary}
\aliascntresetthe{corollary}

\theoremstyle{definition}

\newaliascnt{definition}{theorem}
\newtheorem{definition}[definition]{Definition}
\aliascntresetthe{definition}

\newaliascnt{assumption}{theorem}

\aliascntresetthe{assumption}

\theoremstyle{remark}

\newaliascnt{remark}{theorem}
\newtheorem{remark}[remark]{Remark}
\aliascntresetthe{remark}

\crefname{theorem}{Theorem}{Theorems}
\Crefname{theorem}{Theorem}{Theorems}

\crefname{lemma}{Lemma}{Lemmas}
\Crefname{lemma}{Lemma}{Lemmas}

\crefname{proposition}{Proposition}{Propositions}
\Crefname{proposition}{Proposition}{Propositions}

\crefname{corollary}{Corollary}{Corollaries}
\Crefname{corollary}{Corollary}{Corollaries}

\crefname{definition}{Definition}{Definitions}
\Crefname{definition}{Definition}{Definitions}

\crefname{assumption}{Assumption}{Assumptions}
\Crefname{assumption}{Assumption}{Assumptions}

\crefname{remark}{Remark}{Remarks}
\Crefname{remark}{Remark}{Remarks}

\crefname{section}{Section}{Sections}
\Crefname{section}{Section}{Sections}

\crefname{subsection}{Section}{Sections}
\Crefname{subsection}{Section}{Sections}

\crefname{figure}{Figure}{Figures}
\Crefname{figure}{Figure}{Figures}

\newcommand{\fraki}{\mathfrak{i}}
\newcommand{\frakj}{\mathfrak{j}}
\newcommand{\frakk}{\mathfrak{k}}
\newcommand{\frakl}{\mathfrak{l}}

\newcommand{\frakp}{\mathfrak{p}}

\newcommand{\bbC}{\mathbb{C}}

\newcommand{\bbE}{\mathbb{E}}

\newcommand{\bbN}{\mathbb{N}}

\newcommand{\bbP}{\mathbb{P}}

\newcommand{\bbR}{\mathbb{R}}

\newcommand{\bbZ}{\mathbb{Z}}

\newcommand{\bfe}{\mathbf{e}}

\newcommand{\bfn}{\mathbf{n}}

\newcommand{\bfE}{\mathbf{E}}

\newcommand{\bfN}{\mathbf{N}}

\newcommand{\bfP}{\mathbf{P}}

\newcommand{\calB}{\mathcal{B}}

\newcommand{\calF}{\mathcal{F}}

\newcommand{\calH}{\mathcal{H}}

\newcommand{\calJ}{\mathcal{J}}

\newcommand{\calL}{\mathcal{L}}

\newcommand{\calR}{\mathcal{R}}

\newcommand{\calT}{\mathcal{T}}

\newcommand{\VHat}{\hat{V}}

\newcommand{\YHat}{\hat{Y}}

\newcommand{\tauHat}{\hat{\tau}}

\newcommand{\DBar}{\overline{D}}

\newcommand{\WBar}{\overline{W}}

\newcommand{\OmegaBar}{\overline{\Omega}}

\newcommand{\yTil}{\tilde{y}}

\newcommand{\ATil}{\tilde{A}}

\newcommand{\FTil}{\tilde{F}}

\newcommand{\JTil}{\tilde{J}}
\newcommand{\KTil}{\tilde{K}}

\newcommand{\QTil}{\tilde{Q}}

\newcommand{\WTil}{\tilde{W}}

\newcommand{\YTil}{\tilde{Y}}

\newcommand{\etaTil}{\tilde{\eta}}

\def\restr{\mathord{\upharpoonright}}

\newcommand{\ch}{\mathcal h}
\newcommand{\s}{\text{s}}

\date{\today}
\title{The dynamical bead model}

\begin{document}
\author{Giuseppe Cannizzaro}
\address{University of Warwick, UK}
\email{giuseppe.cannizzaro@warwick.ac.uk}

\author{Harry Giles}
\address{TU Wien, Austria}
\email{harry.giles@tuwien.ac.at}

\author{Fabio Toninelli}
\address{TU Wien, Austria}
\email{fabio.toninelli@tuwien.ac.at}

\begin{abstract}
We construct an infinite-dimensional diffusion on the space of interlaced particle configurations on $\bbZ \times \bbR$,
a process we call the \textit{dynamical bead model}. The process (self-)interacts through instantaneous reflection 
and preserves the bead measures of [C. Boutillier, Ann. Probab., 2009]. 
Through the height function correspondence, our process is equivalent to a reversible interface in $(2+1)$-dimensions. 
We provide compelling evidence that at large space-time scales the interface exhibits diffusive behaviour, 
as we compute exactly its diffusivity matrix 
and show it is non-degenerate and independent of time. 
\end{abstract}

\maketitle
\bigskip

\section{Introduction}\label{sec:lozenge-tilings-diff}
In the present paper, we introduce a novel $(2+1)$-dimensional dynamical interface model, that we call the 
\textit{dynamical bead model}. The name is due to the fact that the spatial gradient of the process is reversible 
with respect to a family of infinite-volume Gibbs measures $\bf P_\gamma$  
labelled by a slope parameter $\gamma\in(-1,1)$, which are 
the Gibbs measures of the so-called (static) bead model. The static bead model is
a determinantal point process on $\mathbb Z\times \mathbb R$
constructed in \cite{BoutillierBeads} 
as a universal scaling limit of dimer models on periodic bipartite graphs. 
On the other hand, its dynamical counterpart discussed herein can be viewed as a special case of the dynamical 
Ginzburg-Landau $\nabla\phi$ interface model studied in  \cite{FunakiSpohn,Giacomin2001Equilibrium,Armstrong2022C2}, 
but in which the interaction potential is neither smooth nor strictly convex. 

As in this latter case, we are interested in the large-scale space-time properties of the dynamical bead model. In this perspective, 
our work is a step in the long-term program, set out in \cite{Spohn1993Interface},
of obtaining scaling limits of stochastic reversible interface motion. Before presenting our results, 
let us (informally) introduce the model and its relation to the literature (see \cref{sec:model} for the formal statements).  

\subsubsection*{The dynamical bead model.} The dynamical bead model is  
an infinite, two-dimensional system
of interacting and interlaced one-dimensional Brownian particles. 
Particles live on an 
infinite collection of lines (copies of $\mathbb{R}$), that we label
with an index $\frakl\in \mathbb Z$, and carry a double
discrete label $(\frakl,\frakk)$ that specifies both which line $\frakl$ they
live on (particles cannot jump to other lines) and their ordering $\frakk$
along that line. On each line, there are infinitely many
particles. For convenience, the order index $\frakk$ is chosen to take
values in $\mathbb Z$ for $\frakl$ even and in $\mathbb Z+1/2$ for
$\frakl$ odd.  The interlacement condition imposes that the coordinate
(call it $y(\frakl,\frakk)$) of the particle labelled
$(\frakl,\frakk)$ (that is, the $\frakk$-th particle in line
$\frakl$)
satisfies
\begin{equation}\label{eq:35}
y(\frakl\pm1,\frakk-1/2)<y(\frakl,\frakk)<y(\frakl\pm1,\frakk+1/2).
\end{equation}
See Figure \ref{fig:bead_config_example}. The time-evolution of the particles can then be described as follows: 
each particle is subject to 
an i.i.d.\ one-dimensional Brownian noise and it is reflected off
particles in neighbouring lines, in such a way that the interlacement
is satisfied at all times. In particular, in each line the ordering of
particles is preserved. In formulas, this gives rise to the infinite
system
\begin{equation}\label{eq:65}
\begin{split}
\dd Y_t(\frakl, \frakk) = \dd W_t(\frakl, \frakk) &+ \thalf \dd L_t((\frakl-1, \frakk-\thalf),(\frakl, \frakk)) - \thalf \dd L_t((\frakl, \frakk), (\frakl-1, \frakk+\thalf)) \\
&+ \thalf \dd L_t((\frakl+1, \frakk-\thalf),(\frakl, \frakk)) - \thalf \dd L_t((\frakl, \frakk), (\frakl+1, \frakk+\thalf))
\end{split}
\end{equation}
of coupled Stochastic Differential Equations (SDEs) with normal
reflections.
Here, $(W(\frakl,\frakk))_{\frakl,\frakk}$ is an infinite family of independent standard Brownian motions, 
and $L_t((\frakl\pm1,\frakk- 1/2),(\frakl, \frakk))$ is the local time at zero of  $Y_t(\frakl,\frakk)-Y_t(\frakl\pm 1,\frakk-1/2)$, 
and similarly for $L_t((\frakl, \frakk), (\frakl\pm1, \frakk+1/2))$.
As explained in Definition \ref{def:Height} below, to each configuration of interlaced particles one can associate 
(bijectively, up to a suitable height offset) a height function defined on $\mathbb Z\times \mathbb R$. 
Therefore, one can equivalently view the process as a dynamically evolving $(2+1)$-dimensional interface. 

The system has two microscopic
locally conserved quantities,  corresponding to
the height slopes in the two spatial directions. In terms of particles, these conserved quantities are the particle number and the function $$(\frakl,\frakk)\mapsto [y(\frakl+1,\frakk+1/2)-y(\frakl,\frakk)]-[y(\frakl,\frakk)-y(\frakl+1,\frakk-1/2)].$$
The stationary measures are labelled by the average
 interface gradient. Since by scaling we can fix the average particle density (and choose it to be $1/\pi$), 
these will be parametrised by a single slope parameter
$\gamma\in(-1,1)$ and are given by the 
$\bf P_\gamma$'s alluded to earlier (see Section \ref{sec:thebeadmodel} for their definition).

\subsubsection*{Reflecting diffusions, spin dynamics and lozenge tilings.}  
Finite or infinite systems of reflecting diffusions have a long tradition in probability and mathematical physics. 
Started in the works of Skorokhod \cite{Skorokhod61_StochasticEquations} and Tanaka \cite{Tanaka79_StochasticDifferential}, 
their analysis can count by now on a significant body of literature, see e.g. \cite{Pilipenko14_IntroductionStochastic} for 
a review in the case of finitely many particles, and \cite{IchiKaraShko13_StrongSolutions, Sarantsev17_InfiniteSystems} 
or the more recent \cite{banerjee2025strong} (and references therein) for the case of countably many diffusions 
evolving on a one-dimensional space. 
Another line of research, pioneered by \cite{dyson1962brownian} and 
in which the hard-core reflection is replaced by a singular (long-range) potential, 
originated in random matrix theory, in particular in the study of the eigenvalues of random matrix ensembles;  
see also \cite{assiotis2024isde}, that analyses generalisations of Dyson's Brownian motion, 
and constructs the process with infinitely many particles. 
Instead, results concerning interlaced particles in two dimensions are much scarcer. A notable example
is the system introduced in \cite{Warren2007Dysons} which shares a similar state space to ours, but 
whose dynamics is very different: in \cite{Warren2007Dysons} particles in line $\frakl$ 
influence those in line $\frakl+1$ but not vice-versa, while the interaction is mutual and symmetric in our case. 

From the statistical physics viewpoint, the process we define displays connections 
with several models of interest. On the one hand, it is a natural two-dimensional extension of the one-dimensional 
model of gases of particles with hard-core interactions studied in \cite{harris1965diffusion}. 
On the other (and arguably, more interestingly), it may be viewed as a limit of the zero-temperature Glauber dynamics of $+ / -$
spin interfaces of the three-dimensional Ising model, when the normal vector to the  $+ / -$ interface 
has one almost vanishing component. Indeed, it is known
\cite{Caputo2011Zero} that the zero-temperature $3 d$ Ising dynamics is equivalent to
a dynamics on lozenge tilings of the plane (see Figure~\ref{fig:lozenge_dynamics}), or 
to a dynamics of the dimer model on the hexagonal lattice.  
Our model is then (formally) the limit where the density of horizontal lozenges 
(that correspond to our particles) is vanishingly small,  
while space and time are suitably rescaled in order to restore a  density and a noise variance of order 1. 
At the level of stationary measures, the correspondence between interlaced  particles and lozenge tilings 
can be made rigorous. In fact, it was shown in
\cite{BoutillierBeads} that the Gibbs measures of lozenge tilings of the plane, under suitable space
rescaling, converge in the low-density limit to the bead model measures $\bf P_\gamma$. 

\begin{figure}[htbp]
\centering
% A patch of a genuine lozenge tiling (projection of a plane partition).
% The horizontal and vertical units give equilateral lattice triangles.
% Standalone TikZ fragment, as in the slide figures; requires only tikz.
\begin{tikzpicture}[x=0.55cm,y=0.635085cm,>=stealth,
  tile/.style={draw=black!60,line width=0.35pt,line join=round},
  white tile/.style={tile,fill=white},
  black tile/.style={tile,fill=black}]
  \begin{scope}
    \clip (-5.5,-1) rectangle (5.5,6);
    % The two non-horizontal orientations.
    \foreach \x/\y in {
    -4/4, -4/5, -5/0.5, -5/1.5, -1/5.5, -2/3, -2/4,
    -2/5, -2/6, -3/-0.5, -3/0.5, -3/1.5, -3/2.5, -4/-1,
    1/4.5, 1/5.5, 0/3, 0/4, -1/1.5, 3/5.5, 2/3,
    1/0.5, 1/1.5, 0/-1, 0/0, 5/5.5, 4/1, 4/2,
    4/3, 4/4, 3/-1.5, 3/-0.5, 3/0.5, 3/1.5, 2/-1,
    6/0, 6/1, 6/2, 6/3, 6/4, 5/-1.5, 5/-0.5
    } {
      \path[white tile] (\x,\y) -- ++(-1,-0.5) -- ++(0,1)
        -- ++(1,0.5) -- cycle;
    }
    \foreach \x/\y in {
    -6/3, -6/4, -6/5, -4/4, -4/5, -4/6, -5/0.5,
    -5/1.5, -5/2.5, -6/-1, -1/5.5, -2/3, -2/4, -3/-0.5,
    -3/0.5, -3/1.5, 1/4.5, 1/5.5, 0/3, -1/1.5, -2/-1,
    -2/0, 3/5.5, 2/3, 2/4, 1/0.5, 1/1.5, 0/-1,
    5/5.5, 4/1, 4/2, 4/3, 4/4, 3/-1.5, 3/-0.5,
    5/-1.5
    } {
      \path[white tile] (\x,\y) -- ++(1,-0.5) -- ++(0,1)
        -- ++(-1,0.5) -- cycle;
    }
    % Horizontal lozenges; their centres interlace on neighbouring columns.
    % The central particle (0,2.5) has visible neighbours (-1,1), (-1,5),
    % (1,0), and (1,4), all fully contained in the patch.
    \foreach \x/\y in {
    -5/6.5, -6/3, -4/4, -5/0.5, -1/5.5, -2/3, -3/-0.5,
    1/4.5, 0/3, -1/1.5, 3/5.5, 2/3, 1/0.5, 5/5.5,
    4/1, 6/0
    } {
      \path[black tile] (\x,\y) -- ++(1,-0.5) -- ++(-1,-0.5)
        -- ++(-1,0.5) -- cycle;
    }
  \end{scope}

  % The two tilings of a unit hexagon: one horizontal lozenge moves by
  % one vertical lattice spacing, while its two white neighbours rearrange.
  \begin{scope}[shift={(9,4.4)}]
    \path[white tile] (0,0) -- (-1,0.5) -- (-1,-0.5) -- (0,-1) -- cycle;
    \path[white tile] (0,0) -- (0,-1) -- (1,-0.5) -- (1,0.5) -- cycle;
    \path[black tile] (0,0) -- (1,0.5) -- (0,1) -- (-1,0.5) -- cycle;
  \end{scope}
  \begin{scope}[shift={(9,0.6)}]
    \path[white tile] (0,0) -- (-1,-0.5) -- (-1,0.5) -- (0,1) -- cycle;
    \path[white tile] (0,0) -- (0,1) -- (1,0.5) -- (1,-0.5) -- cycle;
    \path[black tile] (0,0) -- (-1,-0.5) -- (0,-1) -- (1,-0.5) -- cycle;
  \end{scope}
  \draw[->,line width=0.6pt] (8.6,3.15) -- (8.6,1.85)
    node[midway,left,font=\small] {$1$};
  \draw[->,line width=0.6pt] (9.4,1.85) -- (9.4,3.15)
    node[midway,right,font=\small] {$1$};
\end{tikzpicture}
\caption{Horizontal lozenges (black) form interlaced particles. The Glauber dynamics of a 3$d$ Ising interface at zero temperature corresponds to up/down jumps of rate $1$, subject to interlacement.}
\label{fig:lozenge_dynamics}
\end{figure}
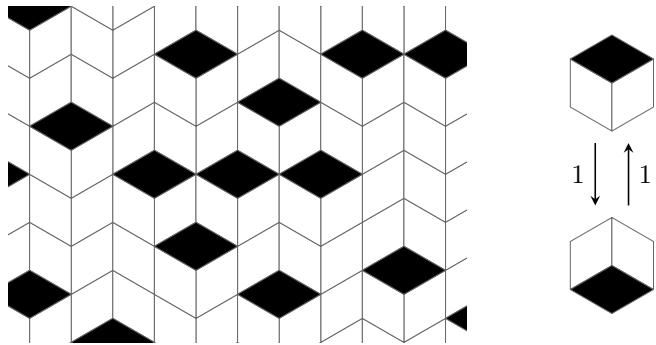

Let us emphasise that the dynamics of lozenge tilings on the plane has in itself attracted significant attention, 
both in the reversible setting (where lozenges can jump, possibly by a distance larger than one,
with symmetric rates) and in the irreversible (asymmetric) one.
Several results appeared concerning classification of stationary measures \cite{TAoP}, 
mixing time bounds in finite domains \cite{Wilson,CMT, AT} and hydrodynamic 
limits \cite{LThydro,LegrasT}. Moreover, as \cref{fig:lozenge_dynamics} clearly suggests (and we discussed earlier), the evolution 
of lozenges can be equivalently described as the motion of a $(2+1)$-dimensional random surface 
which, in case of asymmetric jump rates, is especially relevant as its fluctuations 
are conjectured to lie within the KPZ universality class \cite{KPZ}. 
In this context, spatial dimension $d=2$ is special as Wolf \cite{W91} conjectured the existence of two 
universality classes depending on how the different directions of the 
slope influence the dynamics, the Anisotropic KPZ class (AKPZ), which is logarithmically superdiffusive (see \cite{CET}) 
and should display large-scale Gaussian statistics, and the Isotropic KPZ class (IKPZ), which 
should instead be polynomially superdiffusive (whose exact size and nature of the fluctuations though are not known). 
While several discrete dynamical surface models  
have been shown to belong to the AKPZ class (see e.g. \cite{TonRev, BorodinTonAKPZ}), 
for the IKPZ class, analysable examples are scarcer\footnote{a notable representative is the log-partition function
of directed polymers (see e.g. \cite{CSZ} for the fluctuations in the so-called intermediate disorder regime).}, 
and in either case no results on the fluctuations are currently available. As in dimension $1$ asymmetric
reflecting diffusions lead to polynomially superdiffusive behaviour for the corresponding 
height function \cite{SpohnSasamoto15_PointinteractingBrownian} 
(see also \cite{KaraPalShko16_SystemsBrownian} for a rigorous construction 
of the model), we believe the interface associated to the asymmetric version of the 
dynamical bead model should belong to the IKPZ class and thus deserves further investigation.

\subsubsection*{Our results.} When discussing the infinite system 
\eqref{eq:65} of interlaced Brownian particles, 
the first problem we have to face is the global  well-posedness of the infinite-volume dynamics (Theorem \ref{th:main1}).
This would be standard for either a system of finitely many reflecting particles,
as they can be re-written as a Brownian motion in a finite-dimensional
bounded convex domain with normal reflection at the boundary, or 
for an infinite system of interacting diffusions with sufficiently smooth interaction potential \cite{Fritz}. 
In our case, however, the
interaction is encoded in the reflection, resulting in a highly singular
potential. In order to prove existence and uniqueness of the solution of the
infinite system (at least when started from the invariant measure), we have to rule out finite-time explosion caused by chains of
collisions producing a propagation of information from infinity in finite
time. For this, we crucially use properties of the stationary measure $\bf P_\gamma$ (in
particular, bounds on the probability of anomalously large or small particle
gaps), together with disagreement percolation arguments inspired by \cite{TAoP}. 
Let us mention that systems of reflecting particles on the line (as those alluded to at the beginning of the previous section) 
possess a special feature, namely collisions can be interpreted as a reordering of particles.
In our two-dimensional case, on the other hand, this is not possible 
and the main difficulty turns out to be that of controlling the ``transversal'' collision chains, 
that is, those that propagate from one line to neighbouring ones, an issue that, 
clearly, does not arise in the one-dimensional setting.

Our second main result (\cref{thm:GK}) gives qualitative information on the
large-scale behaviour of the process. 
As is the case for other reversible stochastic interface models \cite[Sec. 9]{Spohn1993Interface}, we expect that
the height profile associated to the particle configuration converges, under
diffusive scaling, to a stochastic heat equation with additive noise of the form
\begin{eqnarray}
  \label{eq:conjectural}
  \partial_t \phi=\mu(\gamma) (\nabla \cdot \Sigma(\gamma)\nabla) \phi+ \sqrt{2\mu(\gamma)} \xi
\end{eqnarray}
with $\mu(\gamma)>0$ (mobility), $\Sigma(\gamma)$ a $2\times 2$ symmetric positive matrix (Hessian of the surface tension as a function of the interface slope), and $\xi$ a space-time white noise.
This has been proven for Ginzburg-Landau $\nabla\phi$ interface models with smooth, strictly convex potential \cite{Giacomin2001Equilibrium,Armstrong2022C2}, which is not the case for us as 
the potential enforcing the interlacement condition is either zero or infinity.
While we
do not prove convergence of the height process to \eqref{eq:conjectural}, Theorem \ref{thm:GK} 
identifies what the coefficients
of the limit equation should be.  
More precisely, we define a suitable diffusion
matrix $D(t)$ (see \eqref{eq:DiffMat}) and we prove not only that it is
 finite, non-degenerate and time-independent (giving strong evidence for 
a normal diffusive behaviour) but we  compute it explicitly, as a
function of the parameter defining the stationary measure. As discussed in Remark \ref{rem:conjectural} below, Theorem \ref{thm:GK} strongly suggests that   the matrix $\Sigma(\gamma)$
 in the conjectural scaling limit \eqref{eq:conjectural} 
coincides with the matrix $a(\gamma)$ defined in \eqref{eq:a.bead}, while $\mu(\gamma)=1/(2\pi)$. In particular, the mobility should be independent of $\gamma$, a phenomenon that occurs also for the Ginzburg-Landau model \cite{Giacomin2001Equilibrium}.

\subsection{Notation}
For integers $m\leq n$ we denote by $m{:}n$ the set $\{m, m+1,\dots,n-1,n\}$. 
For a Polish space $E$, we write $C^\infty_c(E)$ for the space of smooth functions of compact support, 
$C^\infty_0(E) \subset C^\infty_c(E)$ for the subspace of mean zero functions, and 
$C^\infty_p(E)$ for the space of smooth functions of at most polynomial growth.

\section{The model and main results}\label{sec:model}

\subsection{A system of interlaced, reflecting Brownian particles}

In this section, we give  the informal description of the dynamics as an infinite system of reflecting SDEs.  The actual rigorous construction is obtained in \cref{sec:fromLtoinfty}.

\begin{definition}\label{def:omega}
  A subset $\omega\subset F_\infty\eqdef \mathbb Z\times \mathbb R$ is said to be an \textit{interlaced particle configuration} if
  \begin{enumerate}
  \item [(i)] $\omega$ is locally finite;
  \item [(ii)] for every $\frakl\in \mathbb Z, $ the set $(\{\frakl\}\times \mathbb R)\cap \omega$ is unbounded in both vertical directions;
  \item [(iii)] given $(\frakl,y_1),(\frakl,y_2)\in\omega$ with $y_1<y_2$ such that there is no $(\frakl,y)\in\omega$ with $y\in(y_1,y_2)$, there exist $(\frakl-1,y_-),(\frakl+1,y_+)\in\omega$ such that $y_\pm\in(y_1,y_2)$.
  \end{enumerate}
  We let $\Omega$ denote the set of all interlaced particle configurations\footnote{note that if item (ii) holds for $\frakl=0$, then it holds for any other $\frakl$, thanks to item (iii).}, 
and identify $\omega \in \Omega$ with the induced counting measure $\delta_\omega \eqdef \sum_{\s \in \omega} \delta_\s$. 
We  consider the $\sigma$-algebra $\calF \eqdef \sigma(\omega \mapsto \delta_\omega(A) \colon A \in \calB(F_\infty))$.
Also, given $F\subset F_\infty$, we let  $\Omega(F)$ be the set of locally finite point measures on $F$, 
that are obtained as restrictions to $F$ of some point configuration $\omega\in \Omega$.
\end{definition}

We will interpret elements of $\omega\in\Omega$ as particles. To conveniently label them, we need the following definitions.

\begin{definition}\label{def:thegraph}
Let $Q_\infty$ be the graph
$$ Q_\infty \eqdef \big( (2 \bbZ) \times \bbZ \big) \cup \big( (2\bbZ + 1) \times (\bbZ + \thalf) \big) $$ where $\fraki=(\frakl,\frakk)\in Q_\infty$ and $\frakj=(\frakl',\frakk')\in Q_\infty$ are declared neighbors (and we write $\fraki\sim\frakj$) if $|\frakl-\frakl'|=1,|\frakk-\frakk'|=1/2$. For two neighbours $\fraki=(\frakl,\frakk),\frakj=(\frakl',\frakk') \in Q_\infty$, we say that $\fraki \le \frakj$ if $\frakk<\frakk'$. This induces a partial order on $Q_\infty $, where $\frak i\le \frak j$ if $\fraki = \frakj$ or if there exists a sequence $\fraki_0=\frak i,\fraki_1,\dots,\fraki_n=\frakj$ with $\fraki_k\sim\fraki_{k+1}$ and $\fraki_k\le \fraki_{k+1}$.
Given $\fraki\in Q_\infty$, we sometimes write $\frakl(\fraki),\frakk(\fraki)$ to denote its first and second components.
\end{definition}

\begin{definition}
Define the infinite dimensional convex set $D_\infty$ according to
\begin{equation}\label{eq:48}
D_\infty \eqdef \{ y \in \bbR^{Q_\infty } : y(\fraki) < y(\frakj) \text{ for all } \fraki, \frakj \in Q_\infty , \, \fraki \sim \frakj, \, \fraki \le \frakj; \lim_{\frakk\to\pm\infty}y(\frakl,\frakk)=\pm\infty \}, 
\end{equation}
and let $\DBar_\infty$ be obtained by replacing the strict inequalities by weak ones.
We define the corresponding space of point measures by
$$ \OmegaBar \eqdef \Big\{ \sum_{\fraki\in Q_\infty}\delta_{(\frakl(\fraki),y(\fraki))} : y\in\DBar_\infty \Big\}. $$
For $F\subset F_\infty$, we write $\OmegaBar(F)$ for the restrictions to $F$ of elements of $\OmegaBar$.
\end{definition}

\begin{remark}\label{rem:interchangeably}
It is easy to see that interlaced particle configurations $\omega \in \Omega$ are in bijection with elements $y \in D_\infty$. 
Given $y\in D_\infty,\fraki\in Q_\infty$, the value $y(\fraki)$ simply represents the vertical height of the bead with label 
$\fraki \in Q_\infty$ (see \cref{fig:bead_config_example}). To fix uniquely the labelling of points in $\omega\in \Omega$, 
we assign the label $(0,0)$ to the point $(0,y)\in\omega$ that has non-negative and minimal $y$ coordinate. 
In particular, it holds that
$$ \omega = \{ (\frakl, y(\frakl, \frakk)) : (\frakl, \frakk) \in Q_\infty \}. $$
In view of this bijection, we will interchangeably use $\omega$ or $y$ to denote an interlaced particle configuration. 
\end{remark}
\begin{figure}[htbp]
\centering
\includegraphics[height=4cm]{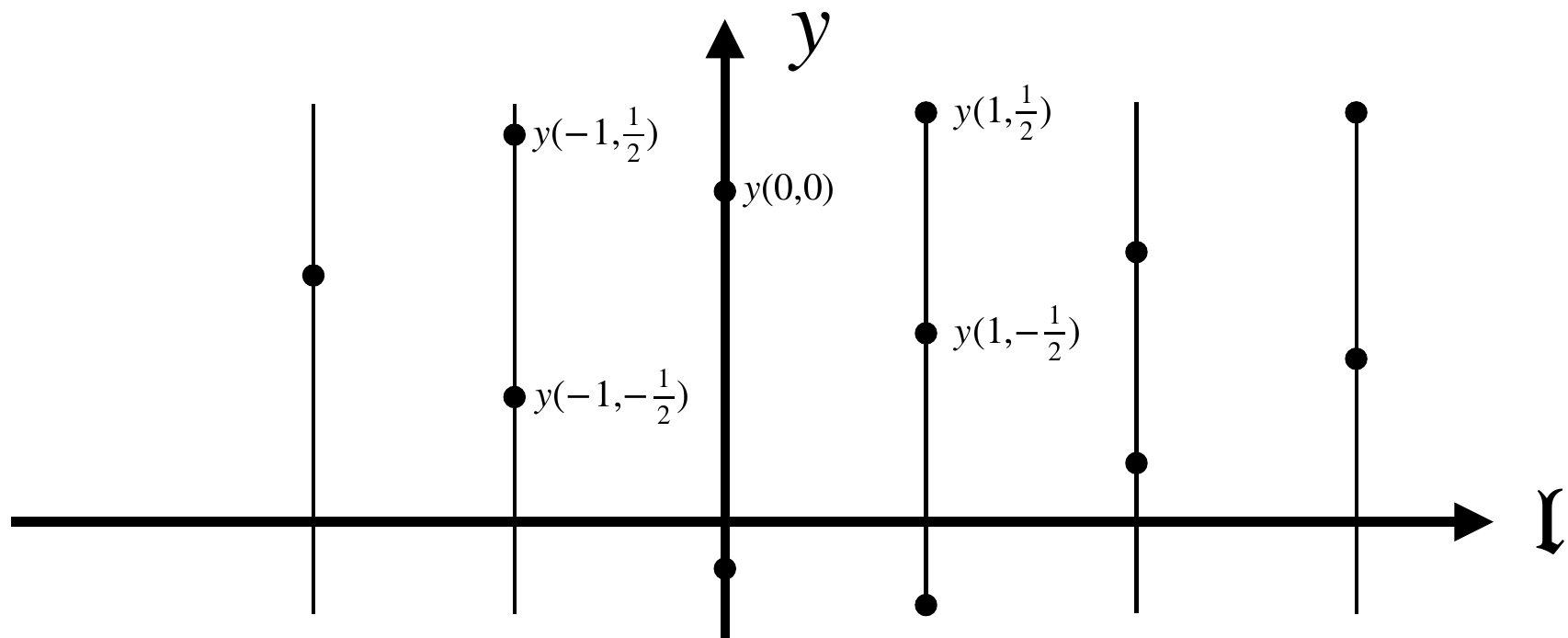}
\caption{An example configuration of beads with their corresponding labels.}
\label{fig:bead_config_example}
\end{figure}

The process of interest is a solution to an infinite-dimensional  diffusion  $Y=(Y_t)_{t\ge0}$ on $\DBar_\infty$.
It corresponds to an infinite system of coupled particle trajectories $(Y_t(\fraki)\colon\fraki\in Q_\infty,t\ge0)$, 
where the particle labelled $\fraki$ is subject to an independent Brownian noise 
and to reflecting interactions with particles labelled $\frakj$ with $\frakj\sim \fraki$, 
that guarantee that $Y_t\in\DBar_\infty$ at all times.

Informally\footnote{the rigorous construction of the process \eqref{eq:65} is discussed in Section \ref{sec:backgr-refl-sdes}
(for finitely many particles) and then extended in \cref{sec:fromLtoinfty} to the infinite system.
}, $Y=(Y_t)_{t\ge0}$ will be given by the infinite system of coupled SDEs \eqref{eq:65},   with $(\frakl, \frakk) \in Q_\infty $,
or, more compactly,
\begin{equation}\label{eq:65bis}
\dd Y_t(\fraki) = \dd W_t(\fraki) + \thalf\sum_{\frakj:\fraki\sim\frakj,\fraki\ge \frakj} \dd L_t(\frakj,\fraki) - \thalf\sum_{\frakj:\fraki\sim\frakj,\fraki\le \frakj} \dd L_t(\fraki,\frakj)
\end{equation}
in which $(W(\fraki) : \fraki \in Q_\infty )$ is an infinite family of independent Brownian motions 
(whose law will be denoted by $\mathbb P$)
and {for $\fraki\sim\frakj,\fraki\le\frakj$, $L_t(\fraki,\frakj) $ is the local time at zero of the non-negative semi-martingale
$Y_t(\frakj)-Y_t(\fraki)$, which (see e.g. \cite[Chap. IV, Section 7]{Protter2004}) 
is given by $L_t(\fraki,\frakj) = \int_0^t \ic\{Y_s(\frakj)-Y_s(\fraki) = 0\} \dd (Y_s(\frakj)-Y_s(\fraki))$.}
The terms $\dd L_t$  encode reflections between particles and guarantee that $Y_t\in\DBar_\infty$ at all times.
\medskip

As mentioned in the introduction, the process above can be translated into a $(2+1)$-dimensional 
random interface model. This is due to the fact that to each particle configuration $\omega\in\Omega$ 
we can associate a height function 
according to the procedure we now explain.

\begin{definition}[Height function and its gradient]\label{def:Height}
Given a configuration $\omega \in \Omega$, we define the {\it height function} $\ch=\ch(\omega)$ associated to $\omega$ 
as the map $\ch \colon \bbZ \times \bbR \rightarrow \thalf \bbZ$ constructed as follows. For $x=(\frakl,z)\in F_\infty$, 
$\ch(x)$ is the unique $\frakk \in \thalf \bbZ$ for which $z\in[y(\frakl,\frakk-1),y(\frakl,\frakk))$. Observe that $\ch$ takes values in $\mathbb Z$ on rows with even index $\frakl$ and takes values in $\mathbb Z+(1/2)$ on rows with odd index $\frakl$.

For $x=(\frakl,z)\in F_\infty$ we write $\nabla \ch(x)=( \nabla_1 \ch(x), \nabla_2 \ch(x))$ for the gradient of $\ch$ whose components are given by 
\begin{gather}
  \label{eq:nablaH}
   \nabla_1 \ch(x)=\ch(\frakl+1,z)-\ch(\frakl,z),\qquad \nabla_2 \ch(x)=\partial_z\ch(\frakl,z).
\end{gather}
\end{definition}

\begin{remark}\label{rem:Delta}
Note that $\partial_z\ch(\frakl,z)$ is a Dirac delta at every particle location on $\{\frakl\}\times \mathbb R$.
\end{remark}

\begin{remark}
If we wish to emphasise the dependence of $\ch$ on the configuration $\omega \in \Omega$, we shall write $\ch(\omega, x)$ in place of $\ch(x)$. 
\end{remark}

\begin{definition}[Pinned height function]
The height function $\ch$ of \cref{def:Height} is ``pinned at zero'' in that $\ch(0) = 0$. 
For $x_0 \in F_\infty$, we define $\ch^{x_0}$ to be the height function pinned at $x_0$, i.e. $\ch^{x_0} := \ch - \ch(x_0)$. 
Of course, the pinning is not seen at the level of the gradient, in that $\nabla \ch = \nabla \ch^{x_0}$ for all $x_0 \in F_\infty$. 
\end{definition}

\subsection{The (static) bead model}
\label{sec:thebeadmodel}
It turns out (see \cref{sec:fromLtoinfty}) that the process $Y$ admits a class of stationary measures $\bfP_\gamma$, parameterised by $\gamma \in (-1, 1)$. The measures $\bfP_\gamma$ on   $(\Omega, \calF)$ were constructed
in \cite{BoutillierBeads} (see also \cite{Johnston23_ContinuousKasteleyn}) as limits of Gibbs measures of 
the dimer model on periodic, bipartite graphs, and they define the so-called {\sl bead model}.  
The main features of $\bfP_\gamma$ we will need in the rest of the paper are 
\begin{enumerate}
\item the distribution $\bfP_\gamma$ is  invariant and  ergodic under translations of $F_\infty$,
\item\label{itm:bead-gibbs} given $\omega_0\in \Omega$ and a bounded subset $\Lambda\subset F_\infty$, the conditional measure $$\bfP_\gamma(\cdot|\omega\restr_{\Lambda^\complement}=\omega_0\restr_{\Lambda^\complement})$$ is the  uniform measure on elements of $\Omega$ that coincide with $\omega_0$ outside of $\Lambda$. In particular, the conditional measure does not depend on the parameter $\gamma$,
\item under $\mathbf{P}_{\gamma}$, $\omega$ is a  determinantal point process on $F_\infty$ whose kernel 
   $K_{\gamma}\colon  F_\infty\ni(\frakl,y)\mapsto K_\gamma(\frakl, y)\in\mathbb R$ is given by (see~\cite[Lemma 1]{BoutillierBeads}) 
   \begin{equation}\label{eq:J}
  K_{\gamma} (\frakl, y) =
  \begin{cases}
    \int_{[- 1, 1]} e^{- i y \varphi} \Big( \gamma + i \varphi \sqrt{1 -
    \gamma^2} \Big)^\frakl \frac{\dd \varphi}{2 \pi} & \text{if} \quad \frakl \geqslant 0\\
    
    - \int_{\bbR \setminus [- 1, 1]} e^{- i y \varphi} \Big( \gamma + i
    \varphi \sqrt{1 - \gamma^2} \Big)^\frakl \frac{\dd \varphi}{2 \pi} & \text{if}
    \quad \frakl < 0.\\
  \end{cases}
\end{equation}
This means that for every $n\in\mathbb N$, the $n$-point correlation function\footnote{Informally, the probability to have a particle in $(\frakl_i,dy_i)$ for all $i\le n$ is given by $\rho_n((\frakl_1,y_1),\dots,(\frakl_n,y_n))dy_1\dots dy_n$. See \cite{Soshnikov00_DeterminantalRandom} for a rigorous definition.} $\rho_n$ at $ (\frakl_i,y_i)\in  F_\infty,\,i\le n$ equals
\begin{equation}\label{rhonJ}
\rho_n((\frakl_1,y_1),\dots,(\frakl_n,y_n))\eqdef  \det \Big((K_{\gamma} (\frakl_i - \frakl_j, y_i - y_j))_{1\le i,j\le n}\Big)\ge0.
\end{equation}

\end{enumerate}

The $n$-point correlation function allows to express all joint moments of the number of points in disjoint Borel  subsets of $F_\infty$ \cite[Eq.\ (1.5)]{Soshnikov00_DeterminantalRandom}. In particular,  given $n\in\mathbb N$ and $B_i=\{\frakl_i\}\times A_i, i=1,\dots,n$  disjoint  subsets of $F_\infty$ (with $A_i$ Lebesgue-measurable), 
we have
\begin{equation}\label{eq:soshni}
  \bfE_\gamma\Big(\prod_{i=1}^n{\# {B_i}}\Big)=\int_{A_1\times\dots\times A_n} \rho_n((\frakl_1,y_1),\dots,(\frakl_n,y_n))
  \dd y_1\dots \dd  y_n,
\end{equation}
where, for each $i$, $\dd y_i$ is the Lebesgue measure on $\bbR$, and
$\# {B_i}$ denotes the number of points in $B_i$.
As a simple consequence,
\begin{equation}\label{eq:soshni2}
  \bfP_\gamma\Big({\# {B_i}\ne 0},\forall i=1,\dots,n\Big)\le\int_{A_1\times\dots\times A_n} \rho_n((\frakl_1,y_1),\dots,(\frakl_n,y_n))
  \dd y_1\dots \dd y_n.
\end{equation}

Other important facts from \cite{BoutillierBeads} include 
  \begin{enumerate}[label=(\alph*)]
  \item
    from \eqref{eq:soshni} and \eqref{eq:J}, it follows that the particle density is $K_\gamma(0,0)=\frac1\pi$, that is, $\bfE_\gamma(\# {\{\frakl\}\times A})=\frac1\pi\Leb(A)$, with $\Leb$ the Lebesgue measure on $\mathbb R$;
  \item for any $\frakl\in\mathbb Z$,   the restriction  $\omega\cap (\{\frakl\}\times \mathbb R)$ is the sine  point process \cite{dyson1962statistical};

  \item the parameter
   $\gamma$ fixes the average of the distance  $y(1,1/2)-y(0,0)$ between a particle and its top-right neighbour;

\item
  under $\bfP_\gamma$, the two-point correlation function 
  \begin{equation}\label{e:TPDecay}
  \rho_2((0,0),(\frakl,y))-\rho_1(0,0)\rho_1(\frakl,y)=\rho_2((0,0),(\frakl,y))-\frac1{\pi^2}\asymp(\frakl^2+y^2)^{-1}\,,
  \end{equation}
as  $\frakl^2+y^2\to\infty$.
\end{enumerate}

At last, it is possible to prove (but we will not need this fact, so we restrain from giving details) that, as is the case for the dimer model \cite{FTnotes}, the scaling limit of the centred  height field $\ch-\mathbb E_\gamma(\ch)$ under $\mathbb P_\gamma$ is a Gaussian Free Field $\psi$ on $\mathbb R^2$ with covariance
\begin{eqnarray}\label{eq:covpsi}
  \mathbb E(\nabla\psi(x)\cdot a(\gamma)\nabla\psi(y))=\delta(x-y), 
\end{eqnarray}
where $a(\gamma)$ is the matrix
\begin{equation}\label{eq:a.bead}
a(\gamma) \eqdef \pi \begin{pmatrix}
1/(1-\gamma^2) & \gamma / \sqrt{1 - \gamma^2} \\
\gamma/\sqrt{1 - \gamma^2} & 1
\end{pmatrix}.
\end{equation}

\subsection{Main results}
Thanks to the definitions given above, we are now ready to state the main results of the paper. 
First, we show that, provided the initial condition of our dynamics is sampled according to $\bf P_\gamma$ 
then the infinite system of SDEs in \eqref{eq:65} admits a unique solution and such solution has $\bf P_\gamma$ 
as stationary measure. The theorem we here state will be made more precise in Corollary \ref{cor:1} below. 

\begin{theorem}\label{th:main1}  Let $\gamma\in(-1,1)$. For $\bfP_\gamma$-almost every initial condition $y_0\in D_\infty$,
  there exists a  strong solution $Y=\{Y_t(\fraki),\fraki\in Q_\infty\}_{t\ge0}$ to \eqref{eq:65}, with initial condition $Y_0=y_0$. 
  We define $Y$ to be the dynamical bead model. Moreover, the measure
   $\bfP_\gamma$ is stationary.
\end{theorem}

\begin{remark}\label{rem:etaY}
In light of \cref{rem:interchangeably}, the solution $\{Y_t(\fraki),\fraki\in Q_\infty\}_{t\ge0}$ of the previous statement 
defines a trajectory on $\OmegaBar$, that we will denote as $\{\eta_t\}_{t\ge0}$. 
\end{remark}

The process $Y$ will be constructed as the almost-sure limit $L\to\infty$ of a system of finitely many 
SDEs solving \eqref{eq:65} but in which particles with index $(\frakl,\frakk)$ outside a finite set $Q_L$ 
(that diverges as $L\to\infty$) are fixed for all times. 
A byproduct of this construction is the following correlation decay property 
(see \cref{rem:quasi-local}):
for any $\gamma\in(-1,1)$ and any bounded local functions $f,g\colon \Omega\to \mathbb R$ and $T>0$, there exists $C=C(\gamma,T)>0$ such that
\begin{align}
  \label{eq:corrdecay-informal}
  \Big|\mathbb E_\gamma\big[(f(Y_t)-f(Y_0))g(\tau_x Y_0)\big]
  \Big|\lesssim e^{-C (\log |x|)^2}
\end{align}
for every $t\in[0,T]$. Here, $\tau_x$ denotes  space translation by $x\in F_\infty$. Note that the decay above is much faster than that of correlations in the stationary measure $\bfP_\gamma$, 
(see~\eqref{e:TPDecay}).
\medskip

Our second main result focuses instead on the large-scale space-time behaviour of the random interface $\ch$ 
associated to the dynamical bead model (see Definition~\ref{def:Height}) 
and determines the behaviour of its (suitably defined) diffusivity matrix. 
To state it, let us introduce the correlation function of the slope of $\ch$ whose average at fixed time 
will be shown in Appendix \ref{computeXi} to be 
\begin{eqnarray}\label{eq:U}
\bfE_\gamma\big[\nabla \ch(x)\big]=\frac1\pi\Big(-\arctan\Big(\frac\gamma{\sqrt{1-\gamma^2}}\Big),1\Big)\,.
\end{eqnarray}
For $x\in F_\infty$, $i,j\in\{1,2\}$, 
let $S_{t, ij}$ be defined according to 
\begin{equation}\label{eq:Stij}
 S_{t,i j}(x) \eqdef 
 \bbE_\gamma\big[( \nabla_i \ch(\eta_t, x) - \nabla_i \ch(\eta_0, x) ) \nabla_j \ch(\eta_0, 0) \big]\,.
\end{equation}
where $ \bbE_\gamma$ is the expectation with respect to $\bbP_\gamma\eqdef\bfP_\gamma\otimes\bbP$ 
and $\bbP$ is the law of the countable family of Brownian motions driving the dynamics of $Y$.
Further, for $a(\gamma)$ as in \eqref{eq:a.bead}, set
\begin{equation}\label{eq:Stilde}
S_t(x) \eqdef \sum_{i,j=1}^2 a(\gamma)_{ij} S_{t,i j}(x)\,.
\end{equation}
Let us point out that $S_t$ is defined as a distribution and not as a pointwise function as so is $\nabla_2\ch$ 
(see Remark~\ref{rem:Delta}). The rigorous distributional definition is given in \cref{def:Diffu}.

\begin{remark}
The rationale behind the choice made in \eqref{eq:Stilde}, in which we mix derivatives with the specific weighting 
$a(\gamma)$, can be better appreciated by considering the actual two-point correlation 
functional $V_t(x) := \sum_{i,j=1}^2 a(\gamma)_{ij} V_{t,i j}(x)$ for
$$V_{t, ij}(x):=\bbE_\gamma[( \nabla_i \ch(\eta_t, x) - \bfE_\gamma[\nabla_i \ch(x)]) \nabla_j \ch(\eta_0, 0) ]$$
(note that $ S_{t,i j}(x) = V_{t, ij}(x) - V_{0, ij}(x) $).
Indeed at $t=0$, $V_0(\cdot)$ decays sufficiently fast so that, in Fourier transform,
  \begin{eqnarray}
    \label{eq:compressibility}
  \hat V_0(p)=1+o(1)\quad \text{for} \quad p\to0\,.
  \end{eqnarray}
In contrast, the individual $V_{0,ij}(x)$ have a worse decay as $x \rightarrow \infty$, which results in $\hat V_{0,ij}(p)$
not being continuous at zero (see Proposition \ref{prop:FT}) making the corresponding expression for $S$ meaningless.  
\end{remark}

\begin{remark}
Continuing from the previous remark, \eqref{eq:compressibility} can also be interpreted as saying  
that the static compressibility at equilibrium \cite{Spohn91_LargeScale} equals $1$. 
If this were not the case, then an additional normalisation would be required in the definition of $S$.
\end{remark}

We are now ready to define the diffusivity matrix 
 $D(t)=(D_{IJ}(t))_{ I, J \in \{1, 2\}}$, which is given by
\begin{equation}\label{eq:DiffMat}
D_{IJ}(t) \eqdef \frac{1}{t} \int_{\bbZ \times \bbR} x_I x_J S_t(x) \dd x, \qquad t>0\,.
\end{equation}
We will prove in \cref{sec:thediffusivity} that $S_t(\cdot)$ decays faster than any power, so that in particular $D_{IJ}(t)$ is well defined.

\begin{theorem}\label{thm:GK}
The diffusivity matrix is given by $D(t) = a(\gamma) / \pi$, with $a(\gamma)$ defined in \eqref{eq:a.bead}. 
In particular, it is constant in time.
\end{theorem}
\begin{remark}
\label{rem:conjectural}
For the (conjectural) limit process \eqref{eq:conjectural}, the matrix $D$ can be easily shown to equal $2\mu(\gamma)\Sigma(\gamma)$. Comparing with Theorem \ref{thm:GK}, this implies $\mu(\gamma)\Sigma(\gamma)=\frac1{2\pi} a(\gamma)$. On the other hand, from \eqref{eq:covpsi} and \cite[Eq. (9.2)]{Spohn1993Interface} it follows that $\Sigma(\gamma)=a(\gamma)$, which implies $\mu(\gamma)=1/(2\pi)$.
\end{remark}

\subsection*{Organization of the paper}
In \cref{sec:finite-volume-dynam} we solve the truncated system by classical means, giving a family of coupled processes parameterised by the truncation parameter $L \in \bbN$. We then prove two properties of the truncated system: monotonicity in \cref{sec:backgr-refl-sdes}, and an identity in law between the Skorokhod solution and the so-called ``rewiring'' solution in \cref{sec:rewiring-solution}.
In \cref{sec:fromLtoinfty} we show that as the truncation is removed, the solutions converge to a natural limit, 
culminating in a proof of \cref{th:main1} (given in the form of \cref{cor:1}). 
This is achieved by controlling the number of collisions, as in \cref{thm:collisions},
and then using this to control correlation decay, as in \cref{thm:prob.decay}. 
Finally, in \cref{sec:goal-green}, we use a Green--Kubo formula to compute the diffusivity. In \cref{sec:we-now-decompose} we decompose into static and dynamic terms, and then in \cref{sec:static-term} and \cref{sec:dynamic-term} respectively we compute the static term, and show that the dynamic term is equal to zero. Together, this yields \cref{thm:GK}, whose proof can be found at the end of \cref{sec:we-now-decompose}.

\section{The finite-volume dynamics}\label{sec:finite-volume-dynam}

The goal of this section is to introduce a finite-volume dynamics which will serve as 
an approximation to the infinite system in \eqref{eq:65}. Its well-posedness comes from the 
classical theory of reflecting diffusions on a (convex) domain, where the domain 
in our case will be chosen so to respect the interlacement. 

We will first consider a rather general finite particle system in~\cref{sec:backgr-refl-sdes} 
for which we provide two alternative but equivalent (in law) constructions and then 
define the {\it truncated} and {\it annulus dynamics} in~\cref{sec:truncated-dynamics}, 
on boxes of size $L$, whose limit as $L\to\infty$ will be the particle system we aim at studying.

\subsection{The finite particle system and monotonicity}\label{sec:backgr-refl-sdes}

Let $Q \subset Q_\infty$ be a finite subset. The truncated dynamics we want to introduce will 
only involve beads whose label lies in $Q$ and thus, we will refer to $\fraki \in Q$ 
as an \textit{active bead} (or particle) and to $\fraki \notin Q$ as an {\it inactive bead}. 
Alongside the internal reflections between active particles described in \eqref{eq:65}, 
the finite particle system has boundary reflections either with respect to inactive ones (so that they obey the interlacement) 
or with artificial boundaries. 
To distinguish them, let $P$ be the set of ordered pairs of neighbouring active particles and
$M$ be a family of parameters determining the spatial boundary conditions, i.e. 
\begin{align}
P&= P(Q) \eqdef \{ (\fraki, \frakj) : \fraki, \frakj \in Q, \, \fraki \sim \frakj, \, \fraki \le \frakj \}\,,\\
M &= M(Q) \eqdef \{M^\pm(\fraki)\in\bbR \cup \{\pm\infty\} : \fraki \in Q\text{ and }M^-(\fraki) \le M^+(\fraki)\}\,.
\end{align}
For $\frakp= (\fraki, \frakj)\in P$ and $\fraki \in Q$, define the half-spaces $H_\frakp,\, H_\fraki^\pm \subset \bbR^Q$
\begin{gather}
H_\frakp \eqdef \{y \in \bbR^Q :  y(\fraki) \le y(\frakj) \}\,, \\
H_\fraki^+ \eqdef \{y \in \bbR^Q : y(\fraki) \le M^+(\fraki)\}, \quad H_\fraki^- \eqdef \{y \in \bbR^Q : y(\fraki) \ge M^-(\fraki)\} 
\end{gather}
and set 
\begin{equation}\label{eq:72}
\calH \eqdef \{ H_\frakp : \frakp \in P \} \cup \{ H^+_\fraki : \fraki \in Q \} \cup \{ H^-_\fraki : \fraki \in Q \}\,.
\end{equation}
whose elements $H \in \calH$ will be said to be \textit{constraints}. Viceversa, 
we denote by $Q_H\subset Q$ the labels involved in the constraint $H\in\calH$, i.e. 
\begin{equation}\label{e:QH}
Q_H\eqdef
\begin{cases}
\{\fraki, \frakj\}\,,& \text{if $H = H_\frakp$ for $\frakp=(\fraki,\frakj) \in P$,}\\
\fraki\,,& \text{if $H = H^\pm_\fraki$ for $\fraki\in Q$\,.}
\end{cases}
\end{equation}
For $H \in \calH$ define $\bfn_H$ to be the inward pointing unit normal, and $a_H \in \bbR \cup \{\pm\infty\}$ to be the unique constant such that
\begin{equation}\label{eq:28}
H = \{ y \in \bbR^Q : y \cdot \bfn_H \ge a_H \}, \quad \partial H = \{ y \in \bbR^Q : y \cdot \bfn_H = a_H \}.
\end{equation}
If $y \in \partial H$ then we say the constraint $H \in \calH$ is \textit{active}. 
At last, let $\DBar_\calH \subset \bbR^{Q}$ be the convex set where the dynamics lives and $\partial \DBar_\calH$ be its boundary, i.e. 
\begin{equation}\label{eq:8}
\DBar_\calH \eqdef \bigcap_{H \in \calH} H, \quad\text{and}\quad \partial \DBar_\calH \eqdef \DBar_\calH \cap \Big(\bigcup_{H \in \calH} \partial H\Big)\,.
\end{equation}
We say that the couple $(Q,M)$ is a {\it setup} provided the corresponding $\text{int}(\DBar_\calH)\neq\emptyset$. 
We further introduce the set $\partial^2 \DBar_\calH \subset \partial \DBar_\calH$ as the subset in which strictly more 
than one constraint is active. If $y \in \partial \DBar_\calH$ we say that $y$ has a \textit{double point} 
and if $y \in \partial^2 \DBar_\calH$ we say that $y$ has a \textit{multiple point}.

For $y \in \partial \DBar_\calH$, write $\bfN(y) \subset \bbR^{Q}$ for the set of inward-pointing unit normal vectors. If $y\in \partial \DBar_\calH\setminus  \partial^2 \DBar_\calH$, $\bfN(y)$ contains a single element (the vector $\bfn_H$ with $H$ such that $y\in\partial H$). On the other hand, if $y \in \partial^2 \DBar_\calH$, $|\bfN(y)| = \infty$.
\medskip

The finite beads dynamics takes the form of a $\bbR^Q$-valued diffusion $Y$ 
which formally solves the following SDE
\begin{equation}\label{eq:38}
\begin{cases}
\dd Y_t(\fraki) = \dd W_t(\fraki) + \sum_{H \in \calH: \fraki \in Q_H} c(H, \fraki) \dd L_t^H, \quad \fraki\in Q\\
Y_0=y_0\in \DBar_\calH,
\end{cases}
\end{equation}
where $(W(\fraki) : \fraki \in Q)$ is a family of independent Brownian motions, the coefficients $c(H, \fraki)$ are defined as 
\begin{equation}\label{eq:74}
c(H, \fraki) =
\begin{cases}
1/\sqrt{2}, \text{ if } H = H_{(\frakj, \fraki)} \text{ for some } \frakj \in Q \\
-1/\sqrt{2}, \text{ if } H = H_{(\fraki, \frakj)} \text{ for some } \frakj \in Q \\
1, \text{ if } H = H^-_\fraki \\
-1, \text{ if } H = H^+_\fraki 
\end{cases}
\end{equation}
and
$L^H \ge 0$ is the local time at zero of the non-negative semi-martingale $Y\cdot\bfn_H-a_H$.

\begin{remark}
In the case $H = H_{(\frakj, \fraki)}$, the local time $L_t^H$ is $L_t(\frakj,\fraki)/\sqrt{2}$, where $L_t(\frakj,\fraki)$ was introduced below \eqref{eq:65}. This is the reason for $|c(H,\fraki)| = 1/\sqrt{2}$ in this case, rather than $1/2$.
\end{remark}

In the following proposition, we state the well-posedness of~\eqref{eq:38} as well as some of its properties. 
As the proof is classical, we postpone it to Appendix~\ref{sec:refl-sde-strong}. 

\begin{proposition}\label{lem:triple}
Let $(Q, M)$ be a setup. Then for all $y_0 \in \DBar_\calH$, there exists a unique strong solution to the SDE \eqref{eq:38}. 
More precisely, there is a continuous function $\Gamma \colon C([0, \infty), \bbR^{Q}) \rightarrow C([0, \infty), \DBar_\calH)$ 
such that, given a family of independent Brownian motions $W=\{W(\fraki)\}_{\fraki\in Q}$, $Y \eqdef \Gamma(y_0 + W)$ 
is the unique continuous semi-martingale for which \eqref{eq:38} holds. Moreover, in case 
$y_0 \in \DBar_\calH \setminus \partial^2 \DBar_\calH$, it holds
\begin{align}
&\bbP( \Leb(t\ge0 : Y_t \in \partial \DBar_\calH) = 0) = 1\,, \label{eq:46}\\ 
&\bbP(\exists t : Y_t \in \partial^2 \DBar_\calH) = 0\,.  \label{eq:47} 
\end{align}
\end{proposition}

The map $\Gamma$ in the previous statement is nothing but the classical Skorokhod reflection map $\Gamma$ 
whose definition and basic properties are 
recalled in Appendix~\ref{sec:skorokhod-solution}. 

\begin{remark}\label{rem:proj}
In what follows, we will say that a $\bbR^{Q_\infty}$-valued process $Y = (Y_t)_{t \ge 0}$ solves \eqref{eq:38} 
if the finite particle projection $\Pi_QY$, for $\Pi_Q$ the canonical projection of $\bbR^{Q_{\infty}}$ onto $\bbR^Q$, 
solves \eqref{eq:38} and $\Pi_{Q_\infty \setminus Q}Y$ is static, 
i.e.\ $\Pi_{Q_\infty \setminus Q}Y_t = \Pi_{Q_\infty \setminus Q} Y_0$ for all $t \ge 0$.
\end{remark}

Before turning to probabilistic properties of our finite particle system (such as its invariant measure), 
we want to show it satisfies a suitable form of monotonicity with respect to the set $Q$ which 
will be important in \cref{sec:priori-estim-bead}. 
Let us begin with the following definition.

\begin{definition}\label{def:agg}
Let $Q' \subset Q \subset Q_\infty$ be two finite index sets, and $M$ and $M'$ two corresponding boundary parameters. We say that $(Q', M') \ge (Q, M)$ if the following two properties hold
\begin{enumerate}
\item for all $\frakj \in Q'$, 
\begin{equation}\label{eq:39}
\{ \fraki \in Q : \fraki \le \frakj \} \subset Q'\,,
\end{equation}
i.e. if $\frakj\in Q'$ then all particles in $Q$ below $\frakj$ are in $Q'$;
\item for all $\fraki \in Q'$, 
\begin{equation}
\label{eq:70}
M_\pm'(\fraki) \ge M_\pm(\fraki)
\end{equation}
\end{enumerate}
\end{definition}

The next theorem discusses the monotonicity property alluded to earlier. In its statement, 
$\Pi_{Q'}$ the canonical projection of $\bbR^{Q}\subset \bbR^{Q_{\infty}}$ onto $\bbR^{Q'}$ for 
$Q'\subset Q\subset Q_\infty$ (see Remark~\ref{rem:proj}) and consider the partial order on $\bbR^Q$ 
according to which $y' \ge y\in \bbR^Q$ if $y'(\fraki) \ge y(\fraki)$ for all $\fraki \in Q$. 

\begin{theorem}\label{thm:monotone}
Let $(Q, M)$ and $(Q', M')$ be two setups such that $(Q', M') \ge (Q, M)$ according to \cref{def:agg}. 
Let $\calH$ and $\calH'$ be the corresponding constraints, as defined in \eqref{eq:72}. 
Let $y_0 \in \DBar_\calH$ and $y_0' \in \DBar_{\calH'}$ be such that $y_0' \ge \Pi_{Q'} y_0$, 
and let $Y$ and $Y'$ be the Skorokhod solutions constructed in \cref{lem:triple} with initial conditions respectively 
given by $y_0$ and $y_0'$ and driven by the same Brownian paths. Then,  
$$ Y'_t \ge \Pi_{Q'} Y_t, \quad \textit{ for all } t \ge 0\,. $$
\end{theorem}

The proof is a slight modification of \cite[Theorem 3.2]{Sarantsev19_ComparisonTechniques}, 
and we give a sketch in Appendix~\ref{a:monotonicity}.

\subsection{Rewiring solution}\label{sec:rewiring-solution}

Apart from the monotonicity property established above, 
a crucial step in the comparison between Skorokhod solutions of~\eqref{eq:38} with different 
setups is the following equivalent (in law) construction of such solutions 
as a ``rewiring'' (random permutation) of independent paths. 

We will focus on setups $(Q, M)$ in which the boundary parameters $M$ are uniform over $\fraki \in Q$, 
namely, for fixed $M_- \in \bbR$, $M_-(\fraki) = M_-$ and $M_+(\fraki) = \infty$ for all $\fraki \in Q$. 
For an initial condition $y_0 \in \DBar_\calH$ and a finite set of independent Brownian paths $W = (W(\fraki) : \fraki \in Q)$, 
our goal is to show that the Skorokhod solution $Y \eqdef \Gamma(y_0 + W)$ equals (in law) the \textit{rewiring} solution, $Z$,  
that we will introduce in \cref{sec:definition-rewiring}. 
The strategy consists of two steps: first, we prove that $Y$ coincides {\it pathwise} with the 
\textit{pairwise Skorokhod solution} $\hat Y$ defined in \cref{sec:decomp-skor-into} (see~\cref{lem:3}); and 
then we prove that $\hat Y$ equals in law $Z$ (see~\cref{thm:yeqz}).

\subsubsection{Pairwise Skorokhod solution}\label{sec:decomp-skor-into}
For a path $y \in C([0, \infty), \bbR)$, $t_0 \ge 0$, and a set of constraints $\calH_0 \subset \calH$, 
let $\tau_1=\tau(y, t_0, \calH_0)$ be the first time after $t_0$ that a constraint in $\calH_0$ is met, i.e. 
\begin{equation}\label{eq:81}
\tau_1=\tau(y, t_0, \calH_0) \eqdef \inf \{ t \ge t_0 : y_t \in \partial H \text{ for some } H \in \calH_0 \}.
\end{equation}
Due to the (almost sure) no multiple point property \eqref{eq:47}, there is a
unique constraint $H \in \calH$ for which
$Y_{\tau_1} \in \partial H$. Call this constraint $H_1$.
For $i \in \bbN$, we now recursively define
$(\tau_{i+1}, H_{i+1})$ as 
\begin{equation}
\tau_{i+1} \eqdef \tau(Y, \tau_i, \calH \setminus \{ H_i \})\,,
\end{equation}
which is the first time after $\tau_i$ that a constraint other than $H_i$ is
met, and $H_{i+1}$ is the unique such
constraint.  The stopping times
$0 \le \tau_1 < \ldots \rightarrow \infty$ are strictly ordered and do
not accumulate, else, by the pigeon hole principle, one could find a contradicting multiple-point.

The pairwise Skorokhod solution $\YHat$ is defined in terms of the \textit{pairwise Skorokhod map}: 
for $H \in \calH$, let $\Gamma_H\colon C( [0, \infty), \bbR^{Q_H} ) \rightarrow C( [0, \infty), \Pi_{Q_H} H )$, 
for $Q_H$ as in~\eqref{e:QH}, be the solution to the Skorokhod problem in the domain $\Pi_{Q_H} H \subset \bbR^{Q_H}$.  
Explicitly, if $H = H_\fraki^-$ for some $\fraki \in Q$, then $\Gamma_H$ is a map on one-dimensional paths 
that gives the Skorokhod reflection against the barrier $M_-$ (see~\cite[Lemma 3.6.14]{KaratzasShreve1991}),
i.e., assuming $w(0) \ge M_-$, 
\begin{equation}\label{eq:92}
\Gamma_H( w )|_t \eqdef w_t + \sup_{0 \le s \le t} (M_- - w_s)^+ .
\end{equation}
If instead $H = H_\frakp$ for some $\frakp \in P$, the map $\Gamma_H$ is the symmetric Skorokhod reflection 
of a two-dimensional path: assuming $w^1_0 \le w^2_0$, 
$$ \Gamma_H( w^1, w^2 )|_t \eqdef \Big( w_t^1 - \thalf \sup_{0 \le s \le t} (w_s^1 - w_s^2)^+, w_t^2 + \thalf \sup_{0 \le s \le t} (w_s^1 - w_s^2)^+ \Big)\,, $$
which is obtained by applying \eqref{eq:92} to the difference $w^2 - w^1$ with $M_- = 0$. 
\medskip 

We now define $\YHat$ inductively on each of the non-degenerate intervals $(\tau_i, \tau_{i+1}]$. For $t \in [0, \tau_1]$, we set $\YHat_t = y_0 + W_t$. For $t \in (\tau_i, \tau_{i+1}]$ we set
\begin{align}
\Pi_{Q_{H_i}}\YHat_t &\eqdef \Gamma_{H_i}( \Pi_{Q_{H_i}}\YHat_{\tau_i} + \Pi_{Q_{H_i}}W_{\tau_i + \cdot} - \Pi_{Q_{H_i}}W_{\tau_i} ) |_{t - \tau_i}\,,\\
\YHat_t(\fraki) &\eqdef \YHat_{\tau_i}(\fraki) + W_t(\fraki) - W_{\tau_i}(\fraki)\,,\qquad \text{for } \fraki \in Q \setminus Q_{H_i}\,.
\end{align}
A priori, it is not even clear that $\YHat$ is $\DBar_\calH$-valued, but this is a consequence of the next lemma.

\begin{lemma}\label{lem:3}
It holds $\bbP$-a.s. that $Y = \YHat$.
\end{lemma}
\begin{proof}
We will proceed by induction on the time interval $[0, \tau_i]$. The case $i = 1$ is trivial, because when the original process $y_0 + W_t$ stays within the domain $\DBar_\calH$, then it exactly equals the Skorokhod solution. 
Assume now that $\YHat = Y$ on $[0,\tau_i]$ and we need to show that $\YHat = Y$ on $[\tau_i, \tau_{i+1}]$. 
It is not apriori known that $\YHat\restr_{[\tau_i, \tau_{i+1}]}$ stays inside the domain $\DBar_\calH$: 
there is no mechanism keeping $\YHat$ to any of the constraints $H \in \calH$ other than $H = H_i$. Define
$$ \tauHat_{i + 1} \eqdef \inf \{ t \in [\tau_i, \tau_{i+1}] : \YHat \notin \DBar_\calH \} $$
with $\tauHat_{i + 1} =\infty$ if no such time exists. We claim that 
\begin{equation}\label{eq:66}
\YHat_t = Y_t, \qquad \forall t \in [\tau_i, \tau_{i+1} \wedge \tauHat_{i+1}]\,.
\end{equation}
If the claim holds, then we are done. Indeed, if $\tauHat_{i + 1} \ne \infty$ 
then $\YHat_{\tauHat_{i+1}} \in \partial H$ for some $H \ne H_i$, 
but since $\YHat_{\tauHat_{i+1}} = Y_{{\tauHat_{i+1}}}$, this would contradict the minimality of $\tau_{i+1}$. 

It remains to prove \eqref{eq:66}, for which it suffices to verify that 
$\YHat\restr_{[\tau_i, \tau_{i+1} \wedge \tauHat_{i+1}]}$ is a solution to the Skorokhod problem for the domain $\DBar_\calH$, 
according to \cref{def:2}, as the latter is unique by~\cref{thm:tanaka}. 
Clearly, $\YHat$ is continuous and, by definition of $\tauHat_{i+1}$, stays inside the domain $\DBar_\calH$ 
on the interval $[\tau_i, \tau_{i+1} \wedge \tauHat_{i+1}]$. 
Define the $\bbR^{Q}$-valued process $(\Lambda_t^i)_{t \in [\tau_i, \tau_{i+1} \wedge \tauHat_{i+1}]}$ according to
$$ \YHat_t = \YHat_{\tau_i} + W_t - W_{\tau_i} + \Lambda_t^i $$
In particular, $\Lambda^i(\fraki) = 0$ for all $\fraki \in Q \setminus Q_H$, and otherwise it is the regulator of the pairwise Skorokhod problem, i.e. 
$$ \Pi_{Q_{H_i}} \Lambda^i_t = \Gamma_{H_i}( \Pi_{Q_{H_i}}\YHat_{\tau_i} + \Pi_{Q_{H_i}}W_{\tau_i + \cdot} - \Pi_{Q_{H_i}}W_{\tau_i} ) |_{t - \tau_i} - (\Pi_{Q_{H_i}}\YHat_{\tau_i} + \Pi_{Q_{H_i}}W_{\tau_i + \cdot} - \Pi_{Q_{H_i}}W_{\tau_i})\,.$$
In particular, it starts at zero and is of finite variation. To conclude, we only need to verify that \eqref{eq:77} and \eqref{eq:78} hold. 
For the former, 
if $\Pi_{Q_{H_i}}\Lambda^i_t$ only increases when $\Pi_{Q_{H_i}}\YHat$ is at the wall $\partial H_i$. Since the particle stays within $\DBar_\calH$ on this interval, $\Lambda^i$ only increases on $\partial \DBar_\calH$. 
Moreover, the candidate $\Lambda^i$ is parallel to $\bfn_{H_i}$, therefore \eqref{eq:78} also holds. \end{proof}

\subsubsection{Definition of rewiring}\label{sec:definition-rewiring}

The rewiring solution is obtained by dividing the driving noise path into segments, which are then 
rewired so to ensure that the constraints of $\calH$ are satisfied. 
The driving noise we will use is not the system of independent Brownian Motions, $(W(\fraki):\fraki\in Q)$, 
used in the construction of $Y$ (or $\YHat$), 
but rather a system of independent Brownian motions $(\WBar(\fraki) : \fraki \in Q)$ 
that have already been reflected against the single barrier at $M_- \in \bbR$ 
and start from $\WBar_0 = y_0$. Since at this point we search for an identity in law, it does not matter how $\WBar$ is constructed\footnote{For example one could construct $\WBar$ from $W$ via one-dimensional Skorokhod reflection, $\WBar(\fraki) \eqdef \Gamma_{M-}(y_0+W(\fraki))$ in which $\Gamma_{M-}$ is the map of \eqref{eq:92}, or via $\WBar(\fraki) \eqdef M_- + |y_0 + W(\fraki) - M_-|$.} and we assume that $\WBar$ is given to us a priori.

We now construct the rewiring solution $(Z_t)_{t \ge 0}$ inductively on intervals $[\tau_i, \tau_{i+1}]$. 
The construction is similar in spirit to the inductive construction of $\YHat$ in \cref{sec:decomp-skor-into}, 
but now instead of a pairwise Skorokhod reflection, we perform a pairwise swapping of paths.
In particular, if one removes the labels, and simply looks at the trace, then $Z$ will coincide with $\WBar$. 

The {\it rewired process} process $Z$ is defined according to $Z = \sigma \WBar$, i.e. 
\begin{equation}\label{e:Rewired}
Z_t(\fraki) = \WBar(\sigma_t \fraki)\,,\qquad t\geq 0
\end{equation}
where $\sigma=(\sigma_t)_{t \ge 0}$ is a permutation-valued process that we now inductively construct. 
In the first iteration, we set $\sigma^1 \equiv 1$ and $Z^1 = \sigma^1 \WBar$. 
Take $\tau_1 = \tau(\sigma^1 \WBar, 0, \calH)$, with $\tau$ as in \eqref{eq:81}, and define
$$ \sigma_t \eqdef \sigma^1_t, \qquad t \in [0, \tau_1] $$
As for $\YHat$, let $H_1$ be the unique constraint that is met at $\tau_1$. We remark that the constraint is unique because trivially the process $\WBar$ has no multiple points. 
Given $(\sigma^i, Z^i, \tau_i, H_i)$, the proceeding $(\sigma^{i+1}, Z^{i+1}, \tau_{i+1}, H_{i+1})$ is obtained as follows. 

We begin by defining $\sigma^{i+1}$. If $H_i = H^-_\fraki$ for some $\fraki \in Q$, then simply set $\sigma^{i+1} \eqdef \sigma^i$. 
Otherwise $H = H_\frakp$ for some $\frakp= (\fraki, \frakj)\in P$ with $\fraki \le \frakj$. 
Writing $(\fraki\, \frakj)$ for the transposition that swaps $\fraki$ and $\frakj$, 
we define $\theta^{i+1}_t \in \{1, (\fraki\, \frakj)\}$ so that 
$$ \theta^{i+1}_t Z^i_t(\fraki) \le \theta^{i+1}_t Z^i_t(\frakj) $$
in which we arbitrarily choose $\theta^{i+1}_t = 1$ when $Z^i_t(\fraki) = Z^i_t(\frakj)$. We set $\sigma^{i+1}_t = \theta^{i+1}_t \sigma^i_t$. 

Let $Z^{i+1} = \sigma^{i+1} \WBar$, $\tau_{i+1} = \tau(Z^{i+1}, \tau_i, \calH \setminus \{H_i\})$ and set
$$ \sigma_t = \sigma^{i+1}_t, \qquad t \in (\tau_i, \tau_{i+1}]. $$
Finally, let $H_{i+1}$ be the unique constraint that is met at $\tau_{i+1}$.

With the construction of $Z$ at hand, we now show that its law coincides with that of $Y$. 

\begin{theorem}\label{thm:yeqz}
It holds $Y \eqd Z$.
\end{theorem}
\begin{proof}
Thanks to \cref{lem:3}, it is sufficient to verify $\YHat \eqd Z$. To do so, we will prove an equality in law for path segments. 
Given a continuous process $(U_t)_{t \ge 0}$ in $C([0, \infty), \bbR^Q)$ along with a family of stopping times 
$ 0 = \tau^U_0 < \tau^U_1 < \ldots \rightarrow \infty $,
define the segment $V^U_{i+1}$ for $i \in \bbN$ to be the $E \eqdef [0, \infty) \times C([0, \infty), \bbR^Q)$-valued random variable 
$$ V^U_{i+1} = (\tau^U_{i+1} - \tau^U_{i}, (U_{(\tau^U_{i}+t) \wedge \tau^U_{i+1}} )_{t \ge 0})\,. $$
We claim that there exists a probability kernel $K \colon \bbR^Q \times \calB(E) \rightarrow [0, 1]$ 
such that for both $U \in \{ \YHat, Z\}$, it holds that for all $A \in \calB(E)$, $i \in \bbN$
$$ \bbP( V^U_{i+1} \in A | \calF^U_{\tau^U_{i}} ) = K( U_{\tau^U_{i}}, A) \qquad \text{a.s.,}$$
where $(\calF^U_t)_{t \ge 0}$ is the natural filtration associated to $U$. 
That separately for $U=\YHat$ and $U=Z$ a kernel $K$ exists is a consequence of the strong Markov property, 
the claim here is that these kernels coincide. From this claim, along with $\YHat_0 \eqd Z_0$, 
the theorem follows from the tower property. 

For $U = \YHat$ we choose $\tau^{\YHat}_i$ to be the $\tau_i$ of \cref{sec:decomp-skor-into}. 
Notice that even though the $\tau_i$'s were initially defined 
in terms of the collisions of $Y$, since $Y = \YHat$ by \cref{lem:3}, $\tau_i$ may be equivalently defined in terms of the collisions of $\YHat$.
Due to the lack of multiple points, $V^{\YHat}_{i+1}$ is a configuration in which there is one unique collision, $H_i$. 
The segment $V^{\YHat}_{i+1}$ evolves in two independent parts: 
the colliding part, $\Pi_{Q_{H_i}}\YHat_{\tau_i}$, and the free part, $\Pi_{Q\setminus Q_{H_i}}\YHat_{\tau_i}$. 
The former is described as follows: if $H = H^-_\fraki$ for some $\fraki \in Q$, 
then $\Pi_{Q_{H_i}}\YHat_{\tau_i}(\frakp_i)$ is distributed as a Brownian Motion 
Skorokhod reflected against a barrier at $M_- \in \bbR$; 
if $H = H_\frakp$ for some $\frakp \in P$ then $\Pi_{Q_{H_i}}\YHat_{\tau_i}$ 
is given by a pair of independent Brownian Motions Skorokhod reflected against each other. 
The latter, on the other hand, evolves simply as a family of independent Brownian Motions, 
independent also from the colliding part. 
The segment ends at $\tau_{i+1}$ when the colliding part meets the free part. 

In the case of $U = Z$, we choose $\tau^Z_i$ to be the $\tau_i$ of \cref{sec:definition-rewiring}. 
The evolution of the segment $V^{Z}_{i+1}$ follows the exact same description as above, 
but with the following difference: we must replace the phrase ``a pair of independent Brownian Motions 
Skorokhod reflected against each other'' with ``the maximum and minimum of two independent Brownian paths''. 
Since these two operations result in the same distribution, the path segments are equal in law. 
In particular, the stopping times $\tau^Z$ and $\tau^{\YHat}$ are equal in law as they describe the hitting time two independent processes, the colliding part and the free part, and as argued, these processes are identical in law. 
\end{proof}

\subsection{Truncated dynamics and their stationary measure}\label{sec:truncated-dynamics}
Let $F \subset F_\infty$ be a bounded region. To avoid pathologies, we assume that $F$ is a finite union of non-degenerate closed intervals, and we refer to such $F \subset F_\infty$ as being \textit{non-degenerate}. The truncated dynamics in $F$ are defined in such a way that the active particles are those
inside $F$, while all other particles are frozen. The evolution of the active particles shall be given by the finite particle dynamics of the previous section for a specific choice of setup $(Q, M)$. The setup $(Q_F, M_F)$ of the truncated dynamics is given by
$$ Q_F = Q_F(y_0)\eqdef \{\fraki \in Q_\infty : (\frakl(\fraki), y_0(\fraki)) \in F \}\,, $$
and, for $\fraki \in Q_F$,
\begin{equation}\label{eq:29}
\begin{split}
M^+_F(\fraki) &\eqdef \min( y_0(\frakj) : \frakj \ge \fraki, \frakj \notin Q_F) \wedge \inf\{y : y \ge y_0(\fraki), \, (\frakl(\fraki), y) \notin F \}\,,\\
M^-_F(\fraki) &\eqdef \max(y_0(\frakj) : \frakj \le \fraki, \frakj \notin Q_F) \vee \sup\{y : y \le y_0(\fraki), \, (\frakl(\fraki), y) \notin F \}\,.
\end{split}
\end{equation}
Notice that $Q_F$ is almost surely finite and the convex set $\DBar_\calH$ is of finite volume.

\begin{definition}\label{def:PointPro}
For $t\ge0$, bounded and non-degenerate $F \subset F_\infty$, $y_0\in D_\infty$, we define the $\OmegaBar$-valued process $\eta^{(F)}$ on $F_\infty$ according to
$$\eta^{(F)}_t=\sum_{\fraki\in Q_\infty}\delta_{(\frakl(\fraki), Y^{(F)}_t(\fraki))}$$
where $\{Y^{(F)}_t(\fraki)\}_{\fraki\in Q_F}$ is the strong solution of \eqref{eq:38}, while $Y^{(F)}_t(\fraki)=y_0(\fraki)$ if $\fraki \not\in Q_F$. 
\end{definition}

Before proceeding, we show that the bead measures are  reversible measures for the truncated system. Recall in particular \ref{rem:proj}.
\begin{proposition}
The bead measures $\bfP_\gamma$ are invariant for the truncated dynamics in $F$.
\end{proposition}

\begin{proof}
We decompose the bead measure into the marginal outside of $F$, and the conditional inside
$$ \bfP_\gamma = \bfP_\gamma\restr_{F^c} \otimes \bfP_\gamma\restr_{F | F^c}. $$
Conditional upon $F^c$, the remaining bead positions are sampled uniformly at random conditional to keeping to interlacement. This is the Markov Property of the bead measure, item~\eqref{itm:bead-gibbs}. The truncated dynamics are that of a reflected brownian motion in the space of points satisfying interlacement, and therefore the uniform measure is reversible for these dynamics. This last statement is classical, and it is stated in more generality (to include the case of an external magnetic field) in \cref{prp:5}.
\end{proof}

We now describe two choices of $F$. The first is the finite box $F_L \subset F_\infty$, $L\in\bbN$, given by
\begin{equation}\label{eq:defFL}
F_L \eqdef \{-L,\dots, L\} \times [-L, L].
\end{equation}
For $F = F_L$, we write
\begin{gather}\label{eq:defQL}
Q_L \eqdef Q_{F_L}, \qquad M^\pm_L(\fraki) \eqdef M^\pm_{F_L}(\fraki)
\end{gather}
Notice that $\DBar_\calH \subset [-L, L]^{Q_L}$. 
The second choice gives the {\it annulus dynamics}, which will be used to prove the decay of certain path functionals, 
namely \cref{lem:4}. Compared to the dynamics in $F_L$, in the annulus dynamics also beads inside the annulus 
$F_{L/4} \setminus F_{L/5}$ are frozen. In other words, upon defining
\begin{equation}\label{e:InOutBox}
\FTil^0_L \eqdef F_{L/5}\,,\qquad \FTil^1_L \eqdef F_L \setminus F_{L/4}\,\qquad\text{and}\qquad \FTil_L \eqdef \FTil^0_L \cup \FTil^1_L\,,
\end{equation}
the active region is given by $F = \FTil_L$. In either case, we write $Y^L \eqdef Y^{(F_L)}$,  $\YTil^L \eqdef Y^{(\FTil_L)}$ 
and, as in \cref{def:PointPro}, 
\begin{equation}\label{equ:TrAn}
\eta^L \eqdef \eta^{(F_L)}\,,\qquad \tilde\eta^L \eqdef \eta^{(\FTil_L)}
\end{equation}
for the resulting dynamics. 

The advantage of the annulus dynamics is that conditional on the initial configuration $y_0$,
the evolution within $\FTil^0_L$ is independent of the evolution within $\FTil^1_L$, 
as the two are separated by the static annulus $F_{L/4} \setminus F_{L/5}$.
Moreover, since information cannot travel too far, 
we expect $\tilde\eta^L$ to be a good approximation of $\eta^L$ if one only observes 
the processes either near the origin, or near the boundary $\partial F_L$.

\section{From finite to infinite volume}\label{sec:fromLtoinfty}

The goal of this section is to determine the infinite volume limit of the dynamics $\eta^L$ in \eqref{equ:TrAn} 
introduced in the previous section 
and define the {\it dynamical bead model}, thus proving Theorem \ref{th:main1}. To do so, we will 
compare the truncated dynamics with different truncation parameters $L,L'$ and show that they coincide with high probability in any finite window provided $L,L'$ are sufficiently large. 
This is the content of the next statement. 

\begin{theorem}\label{thm:prob.decay}
For all $T \ge 0$, $\gamma \in (-1, 1)$, $\theta \in (\tfrac23, \tfrac34)$, there exist constants $c, C, L_0 > 0$
such that for all $x_0 \in F_\infty$ and $L'\ge L \ge L_0$ satisfying $x_0 + F_L \subset F_{L'}$, it holds
\begin{equation}\label{eq:101}
\bbP_\gamma\big( ( \eta^{(x_0 + F_L)}_t\restr_{x_0 + F_{L/2}} )_{t \in [0, T]} \ne ( \eta^{L'}_t\restr_{x_0 + F_{L/2}} )_{t \in [0, T]} \big) \lesssim_T e^{ -(\log L)^2/c} + (L')^4 e^{- c L^2} + L^{\theta} (L')^4 e^{-C L^{3\theta - 2}}.
\end{equation}
where $F_L, F_{L'}$ are defined according to \eqref{eq:defFL} and $\eta^{(x_0 + F_L)}$ and 
$\eta^{L'}=\eta^{(F_{L'})}$ are as in \cref{def:PointPro}. 
\end{theorem}

As an immediate corollary of the above Theorem (whose proof will be detailed in the rest of the section), 
we deduce Theorem \ref{th:main1} which we here restate in 
more rigorous terms (see also Remark \ref{rem:etaY} for the connection between $Y$ and $\eta$). 

\begin{corollary}\label{cor:1}
There exists a unique stationary $\OmegaBar$-valued process $(\eta_t)_{t \ge 0}$ such that for every fixed $x_0 \in F_\infty$, every compact set $F \subset F_\infty$ and every time horizon $T \ge 0$, it holds
\begin{equation}\label{eq:40}
(\eta_t\restr_F)_{t \in [0, T]} = \lim_{L \rightarrow \infty} (\eta_t^{(x_0 + F_L)}\restr_F)_{t \in [0, T]}
\end{equation}
with the limit on the r.h.s. being eventually constant. The process $(\eta_t)_{t \ge 0}$ will be called the dynamical bead model. 
It is stationary 
in the sense that $\mathrm{law}(\eta_t) = \bfP_\gamma$ for all $t \ge 0$, in particular $\eta_t\in\Omega$ 
almost surely for each fixed $t$. It is reversible and spatially stationary, meaning that for every $x \in F_\infty$, 
$(\tau_x\eta_t)_{t \ge 0}$ and $(\eta_t)_{t \ge 0}$ have the same law.
Furthermore, it holds that for a constant $C = C(\gamma, T) > 0$ we have 
\begin{equation}\label{eq:9}
\bbP_\gamma\big( ( \eta_t\restr_{F_{L/2}} )_{t \in [0, T]} \ne ( \eta^{L}_t\restr_{F_{L/2}} )_{t \in [0, T]} \big) \le e^{ -C (\log L)^2 } .
\end{equation}
\end{corollary}
\begin{proof}
Applying \cref{thm:prob.decay} in the case $L' = L+1$, we see by Borel--Cantelli that for any fixed compact set $F \subset F_\infty$, it holds that the following sequence of $\OmegaBar(F)$-valued processes $( \eta^L\restr_{F} )_{L \in \bbN}$ is eventually constant, in which $\eta^L\restr_F = (\eta^L_t\restr_F)_{t \in [0, T]}$. We define $\eta^F := \lim_{L\rightarrow\infty} \eta^L\restr_F$.

Using a standard argument, we lift these local processes to a process that is defined globally in space and time. For the spatial lift, consider a sequence of compact sets $F_n \uparrow F_\infty$. One can easily check that the family of process $(\eta^{F_n} : n \in \bbN)$ is almost surely consistent, in the sense that $\eta^{F_n}\restr_{F_m} = \eta^{F_m}$. This allows us to lift to a global $\OmegaBar$-valued process $(\eta_t)_{t \in [0,T]}$ for which
$$ \eta_t \restr_F = \eta_t^F ,\qquad  \forall  F \subset F_\infty .$$
Considering a similar lift in the time variable $T \ge 0$, we obtain a process $(\eta_t)_{t \ge 0}$ on the full time horizon.

Next we prove that the limit is independent of the location of the $F_L$ box. For $L$ large enough, depending on $x_0$, it holds $x_0 + F_L \subset F_{2L}$. Applying \cref{thm:prob.decay} with $L' = 2L$, we see that $\eta^{(x_0+F_L)}$ and $\eta^{2L}$ eventually agree on every compact set and time interval. Since $\eta^{2L}$ converges locally to $\eta$, the claim is proved. 

From this we obtain the spatial stationarity of the process. The bead measure is invariant under spatial translations, therefore for every $x \in F_\infty$, $(\tau_x\eta^L_t)_{t \ge 0}$ has the same law as $(\eta^{(x+F_L)}_t)_{t \ge 0}$. By \eqref{eq:40}, we obtain 
$$ \law (\eta_t\restr_F)_{t \in [0, T]} = \lim_{L \rightarrow \infty} \law (\eta^{(x+F_L)}_t)_{t \ge 0} = \lim_{L \rightarrow \infty} \law (\tau_x\eta^L_t)_{t \ge 0} = \law (\tau_x\eta_t\restr_F)_{t \in [0, T]} $$

Stationarity and reversibility in time pass to this limit from the truncated dynamics, for which they follow from \cref{prp:5} with $c = 0$.

Equation \eqref{eq:9} is obtained as a bound on the sum of the r.h.s. of \eqref{eq:101}.
\end{proof}

To conclude, let us state a version of \cref{thm:prob.decay} in which we compare $\eta^L$ with the annulus process 
$\etaTil^L$ in \eqref{equ:TrAn}, and make rigorous the heuristic discussion at the end of the previous section. 

\begin{theorem}\label{thm:annulus}
For all $T \ge 0$, $\gamma \in (-1, 1)$, there exists $C=C(\gamma, T) > 0$ so that
\begin{equation}\label{eq:41}
\bbP_\gamma\big( ( \eta^L_t\restr_{G_L} )_{t \in [0, T]} \ne ( \etaTil^L_t\restr_{G_L} )_{t \in [0, T]} \big) \le  C e^{ - C(\log L)^2}.
\end{equation}
for $G_L \eqdef F_{L/6} \cup (F_L \setminus F_{L/3})$. 
\end{theorem}

\subsection{Some equilibrium estimates}\label{sec:equi}

Throughout this section, we shall often work under the event that the collection of labels $Q_L$ in \eqref{eq:defQL}, 
belongs to the deterministic set $P_L$ defined as
\begin{equation}\label{eq:88}
P_L \eqdef \{ (\frakl, \frakk) \in Q_\infty : |\frakl| \le L, |\frakk| \le L \}.
\end{equation}
The reason why we are allowed to do so is contained in the following proposition which further provides moments of 
the cardinality of $Q_L$, denoted by $|Q_L|$. 
\begin{proposition}\label{prp:11}
There is a constant $c > 0$ such that, for every $\gamma\in(-1,1)$ and integers $M\ge L\ge 1$, it holds
\begin{equation}\label{eq:90}
\bfP_\gamma( Q_L \not\subset P_M) \le c e^{- M^2/c}\,.
\end{equation}
Furthermore, $|Q_L|$ admits finite moments of all orders, i.e. there exists a constant $c=c(p,\gamma)$ such that 
\begin{equation}\label{eq:MomQL}
\bfE_\gamma[|Q_L|^p]\leq c L^{2p}\,.
\end{equation}
\end{proposition}
\begin{proof}
We first prove~\eqref{eq:90}. For $\frakl \in -L{:}L$, define $\frakk_{+}(\frakl) > 0$ to be the smallest such that $(\frakl, \frakk_+(\frakl)) \in Q_\infty \setminus P_M$, and similarly let $\frakk_-(\frakl)<0$ be the largest such. Recalling \eqref{eq:defFL}, we have
$$ \{Q_L \not\subset P_M\} \subseteq \bigcup_{\frakl \in -L{:}L} \bigcup_{\sigma \in \pm} \{ |y(\frakl,\frakk_\sigma(\frakl))| \le L \}. $$
By a union bound, it is enough to prove that
\begin{equation}\label{eq:LDsine}
  \sup_{\frakl \in -L{:}L} \sup_{\sigma \in \pm}\bfP_\gamma(|y(\frakl,\frakk_\sigma(\frakl))| \le L)\le c e^{-M^2/c}.
\end{equation}

We consider the case $0\le \frakl\le L$ and $\sigma = +$, the remaining cases are treated similarly. Then by interlacement, we have that $y(\frakl,\frakk_+)\ge y(0,\frakk_+-\frakl/2)\ge y(0,\lfloor M/2\rfloor+1)$. On the other hand, on the event $y(0,\lfloor M/2\rfloor+1)\le L$ there are at least $M/2$ points in $\{0\}\times[0,L]$, and hence in $\{0\}\times[0,M]$.
  Recall from \cref{sec:thebeadmodel} that, under $\bfP_\gamma$, the one-dimensional point process $\omega\cap(\{0\}\times\mathbb R)$ is the sine point process, normalized to have particle density $1/\pi$, so that the event $|\omega\cap(\{0\}\times[0,M])|\ge M/2$ is a large deviation event. In fact, \cite[Th.\ 1]{holcomb2015large} (specialized to the case $\beta=2$ and rescaled to our particle density) implies that this event has probability at most $c e^{-M^2/c}$.
  This implies \eqref{eq:LDsine}, hence~\eqref{eq:90} after absorbing the factor $4L+2$ from the union bound into the constant.
\medskip

Let us now turn to~\eqref{eq:MomQL}. Notice that, on $\{Q_L\subset P_M\}$, we have $|Q_L|\le (2L+1)(2M+1)\le 9LM$. Thus, for $x\ge 9L^2$, choosing $M=\lfloor x/(9L)\rfloor$ yields
$$ \bfP_\gamma(|Q_L|>x) \le \bfP_\gamma(Q_L \not\subset P_M) \le c e^{-M^2/c} \le c e^{-x^2/(cL^2)}, $$
after increasing $c$. By the tail integral formula, for every $p>0$, there is a constant $c=c(p,\gamma)>0$ 
(which might change from step to step) such that 
$$ \bfE_\gamma[|Q_L|^p] = p\int_0^\infty x^{p-1}\bfP_\gamma(|Q_L|>x)\dd x \le (9L^2)^p + cp\int_{9L^2}^\infty x^{p-1}e^{-x^2/(cL^2)}\dd x \le c L^{2p}. $$
\end{proof}

Another statistic that we will need concerns the maximal distance between neighbouring particles. 
Given $y\in D_\infty$ and $L\in\mathbb N$, this is defined as 
\begin{equation}\label{eq:2}
\Delta_L(y) \eqdef \sup \{|y(\fraki)-y(\frakj)|, \fraki\sim \frakj, \fraki\in Q_L,\frakj\in Q_L \}.
\end{equation}

\begin{proposition}\label{prp:12}
There exists a universal constant $c\in(0,\infty)$ such that for every $\gamma\in(-1,1)$ and $L>1$ it holds
\begin{equation}\label{eq:Dyson}
\bfP_\gamma(\Delta_L(y)>\log L)\le c e^{ -(\log L)^2/c}.
\end{equation}
  \end{proposition}
  \begin{proof}
    Note that, by the interlacing condition \eqref{eq:35}, $\Delta_L(y)>\log L$ implies that there exists $\frakl\in-L{:}L$ and an interval $(a,b)$ with $b-a=\log L$ such that $\{\frakl\}\times (a,b)$ contains no points of $\omega$, and is contained in $F_L$. In turn, this gives that there is a set of the form $\{\frakl\}\times{[\frac j2\log L,\frac{j+1}2\log L]}$, with $j\in\mathbb Z$, that is contained in $F_L$ and that contains no particles.
    Then, by  translation invariance of $\bfP_\gamma$ and a union bound, it is sufficient to prove that
    \begin{align}
      \bfP_\gamma((\{0\}\times[0,\thalf\log L])\cap \omega=\emptyset)\le c e^{ -(\log L)^2/c}.
    \end{align}
The point process $\omega$ restricted to a single line has the distribution of the sine point process and
    for the latter, it is known that the probability that the interval $[0,N]$ is empty behaves for large $N$ like $e^{-c N^2 +O(\log N)}$ for a strictly positive constant $c$ \cite{ehrhardt2006dyson,krasovsky2004gap} from which the statement follows.   \end{proof}

  To introduce the next equilibrium statement, we need a preliminary definition.
  \begin{definition}\label{def:chain}
Given $M\in\mathbb N$, we say that a sequence of distinct  sites 
$\fraki_1=(\frakl_1,\frakk_1), \ldots , \fraki_M=(\frakl_M,\frakk_M) \in Q_\infty$ forms a \textit{chain} if $\fraki_i \sim \fraki_{i+1}$ 
for all $i \in 1{:}M-1$ and will be denoted by $\fraki_{1:M}$. 
For fixed $M\in\mathbb N$, chain $\fraki_{1:M}$ and $y\in D_\infty$, 
we define the maximal distance between neighbouring particles as 
\begin{equation}\label{eq:1}
S(\fraki_{1:M},y) \eqdef \max_{i \in 1{:}M-1} |y(\fraki_i) - y(\fraki_{i+1})|\,.
\end{equation}
  \end{definition}

\begin{lemma}\label{lem:1}
For every $\gamma\in(-1,1)$ there exists $C=C(\gamma)\in(0,\infty) $ such that for every bounded, non-degenerate $F \subset F_\infty$, $M\in\mathbb N,\varepsilon>0$, we have
\begin{align}
   \bfP_\gamma(\exists \text{ chain } \fraki_{1:M} : \{\fraki_1,\ldots,\fraki_M\} \cap Q_F \ne \emptyset,   S(\fraki_{1:M},y) \le \varepsilon ) \le |F| \varepsilon^{M-1}[ M C]^M\,.
\end{align}
\end{lemma}

\begin{proof}

  It is easy to see from \eqref{eq:J} that $|K_\gamma(\frakl,y)|\le C_0(\gamma)$ for some constant $C_0(\gamma)$ depending only on $\gamma\in(-1,1)$. As a consequence (recalling also \eqref{rhonJ}), we have
  \begin{equation}
    \label{boundrhon}
    \rho_n((\frakl_1,y_1),\dots,(\frakl_n,y_n))\le C_0(\gamma)^n n!
    .
  \end{equation}
  We prove the statement of the lemma for  $C=C(\gamma)=4 C_0(\gamma)$. 

For $F_\frakl := \{z\in\bbR : (\frakl,z)\in F\}$, a union bound over $r\in1{:}M$ and $\frakl\in\bbZ$ reduces the claim to
\begin{equation}\label{eq:bytra}
   \bfP_\gamma(\exists \text{ chain } \fraki_{1:M} : \fraki_r \in Q_F,  S(\fraki_{1:M},y) \le \varepsilon, \frakl(\fraki_r)=\frakl ) \le \Leb(F_\frakl)4^{M-1} \varepsilon^{M-1} C_0(\gamma)^M M!.
\end{equation}
Recall that $\fraki_j=(\frakl_j,\frakk_j)$ with
$|\frakl_j-\frakl_{j-1}|=1,|\frakk_j-\frakk_{j-1}|=1/2$, and that the
$\fraki_j, j=1,\dots,M$ are distinct. Call $\mathbb L_M$ the set of
sequences of integers $\bar\frakl=(\frakl_1,\dots,\frakl_M)$ with
$\frakl_r=\frakl,|\frakl_j-\frakl_{j-1}|=1$. Since $\mathbb L_M$ has $2^{M-1}$
elements, via a union bound it is enough to prove that, for every
$\bar\frakl\in \mathbb L_M$,
\begin{equation}\label{eq:bytra2}
   \bfP_\gamma(\exists \text{ chain } \fraki_{1:M} : \fraki_r \in Q_F,   S(\fraki_{1:M},y) \le \varepsilon, (\frakl(\fraki_1),\ldots,\frakl(\fraki_M)) = \bar\frakl) \le \Leb(F_\frakl)(2\varepsilon)^{M-1} C_0(\gamma)^M M!.
\end{equation}

Let $B\subset \mathbb R^M$ be defined as
  \begin{equation}
    \label{eq:B}
    B=\{y=(y_1,\dots,y_M): y_r\in F_\frakl, |y_j-y_{j-1}|\le \varepsilon,2\le j\le M\}.\end{equation}
  Note that on the event at the l.h.s. of \eqref{eq:bytra2} the particle configuration $\omega$ sampled according to $\bfP_\gamma$ contains $M$ distinct particles located at $(\bar\frakl_j,y_j),j\in1{:}M$ for some $(y_1,\dots,y_M)\in B$. As a  consequence of \eqref{eq:soshni2} and \eqref{boundrhon}, the probability of the latter is upper bounded by
  \begin{align}
    \int_{B}\rho_M((\bar\frakl_1,y_1),\dots,(\bar\frakl_M,y_M))dy_1\dots dy_M\le \Leb(B)C_0(\gamma)^MM!=\Leb(F_\frakl)(2\varepsilon)^{M-1}C_0(\gamma)^MM!.
  \end{align}
  This implies \eqref{eq:bytra2} and thus \eqref{eq:bytra}; summing over $r,\frakl$ and using $\sum_\frakl\Leb(F_\frakl)=|F|$ and $M M!\le M^M$ proves the statement.
\end{proof}

\subsection{A priori estimates on bead increments}\label{sec:priori-estim-bead}
As a key step towards \cref{thm:prob.decay}, we control the increments of a single particle.

\begin{lemma}\label{lem:2} Let $F \subset F_\infty$ be bounded and non-degenerate, $y_0\in D_\infty$ and let  $Y = Y^{(F)}_{y_0}$ be the solution to the truncated problem, \eqref{eq:38}.
For $\fraki \in Q_\infty$, $\delta, \varepsilon > 0$, define
\begin{equation}
A(\delta, \varepsilon, \fraki) \eqdef \Big\{ \sup_{t\in [0,\delta]} |Y_{t}(\fraki) - Y_{0}(\fraki)| \ge \varepsilon\Big \} . \label{eq:75}
\end{equation}
Then,   it holds
$$ \bbP(A(\delta, \varepsilon, \fraki)) \le c |Q_F| \exp (-\varepsilon^{2} / 2\delta) $$
for a universal finite constant $c>0$.
\end{lemma}

The strategy for the proof of \cref{lem:2} is to couple the system $Y$ with a system $\YTil$ that monotonically dominates it 
and for which we can take advantage of the identity in law given by \cref{thm:yeqz}. 
In short, in order to control, say, the upward increments of $Y(\fraki)$, we define a process $\YTil$ from $Y$ as illustrated in Figure~\ref{fig:QL},
i.e.  by
\begin{enumerate}
\item removing all of the beads that ``push the $\fraki$-particle down'', for $\fraki \in Q_\infty$, leaving $\fraki$ as the ``top bead'';
\item bunching the beads below $\fraki \in Q_\infty$ vertically upwards, so that they all share one boundary condition from below.
\end{enumerate}
In the perspective of the rewiring solution of \cref{sec:rewiring-solution}, the reason for (1) is clear: 
the position of top bead is given by the supremum of the driving paths $\WBar$. 
The second point is required as the construction of \cref{sec:rewiring-solution} forces us to consider a unified boundary condition.

\begin{proof}
We have $A(\delta, \varepsilon, \fraki) = A_+(\delta, \varepsilon, \fraki) \cup A_-(\delta, \varepsilon, \fraki)$ in which
\begin{align}
A_+(\delta, \varepsilon, \fraki) \eqdef \Big\{ \sup_{t \in [0, \delta]} (Y_{t}(\fraki) - Y_{0}(\fraki)) \ge \varepsilon \Big\}, \quad A_-(\delta, \varepsilon, \fraki) \eqdef \Big\{ \inf_{t \in [0, \delta]} (Y_{t}(\fraki) - Y_{0}(\fraki)) \le - \varepsilon \Big\}\,.
\end{align}
By symmetry, it is enough to prove the bound for $A_+(\delta, \varepsilon, \fraki)$.
Furthermore, if $\fraki \notin Q_F$, then the event $A_+$ does not occur,
since the particle with label $\fraki$ does not evolve. Therefore it is enough to show
$$ \bbP(A_+(\delta, \varepsilon, \fraki)) \lesssim |Q_F| \exp (-\varepsilon^{2} / 2\delta), \quad \text{uniformly in} \quad
\fraki \in Q_F\,. $$

We now define $\YTil$ 
as a $\bbR^{\QTil_F}$-valued reflecting diffusion, in which the active labels $\QTil_F \subset Q_F$ are given by
\begin{equation}\label{eq:69}
\QTil_F \eqdef \{ \frakj \in Q_F : \frakj \le \fraki \} .
\end{equation}
There are many ways to monotonically disfigure $Y$ so to produce an appropriate $\YTil$ 
that fits the description given before the proof. For concreteness, we make a specific choice.
Define the barrier $\tilde M_F \in \bbR$ according to
$$ \tilde M_F \eqdef \max \{ M^-_F(\frakj) : \frakj \in \QTil_F \}.$$
We have $\tilde M_F < y_0(\fraki)$, since $y_0\in D_\infty$ and no bead lies on $\partial F$. The initial condition $\yTil_0 \in \bbR^{\QTil_F}$ is obtained from $\Pi_{\QTil_F} y_0$ by bunching the beads upwards, according to the following linear contraction
$$ \yTil_0 = y_0(\fraki) - \alpha (y_0(\fraki) - \Pi_{\QTil_F} y_0) $$
for $\alpha \in (0, 1)$ small enough so that $\yTil_0(\frakj) > \tilde M_F$ for all $\frakj \in \QTil_F$.
See \cref{fig:QL} for a graphic of how $\yTil_0$ is obtained from $y_0$ in the case $Y = Y^L$.

\begin{figure}[h]
    \centering
    \begin{minipage}{0.48\textwidth}
        \centering
        \includegraphics[width=\linewidth]{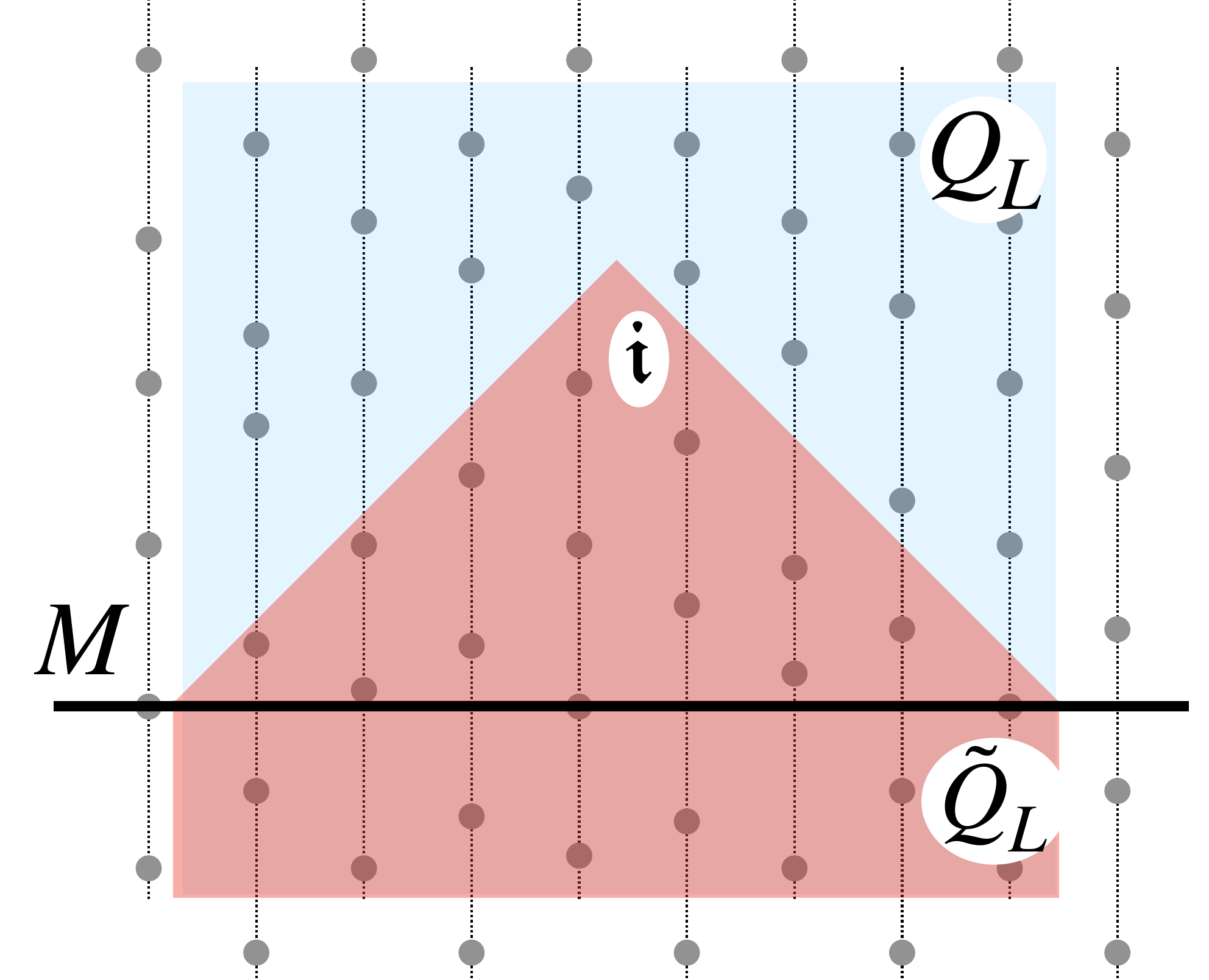}
    \end{minipage}
    \hfill
    \begin{minipage}{0.48\textwidth}
        \centering
        \includegraphics[width=\linewidth]{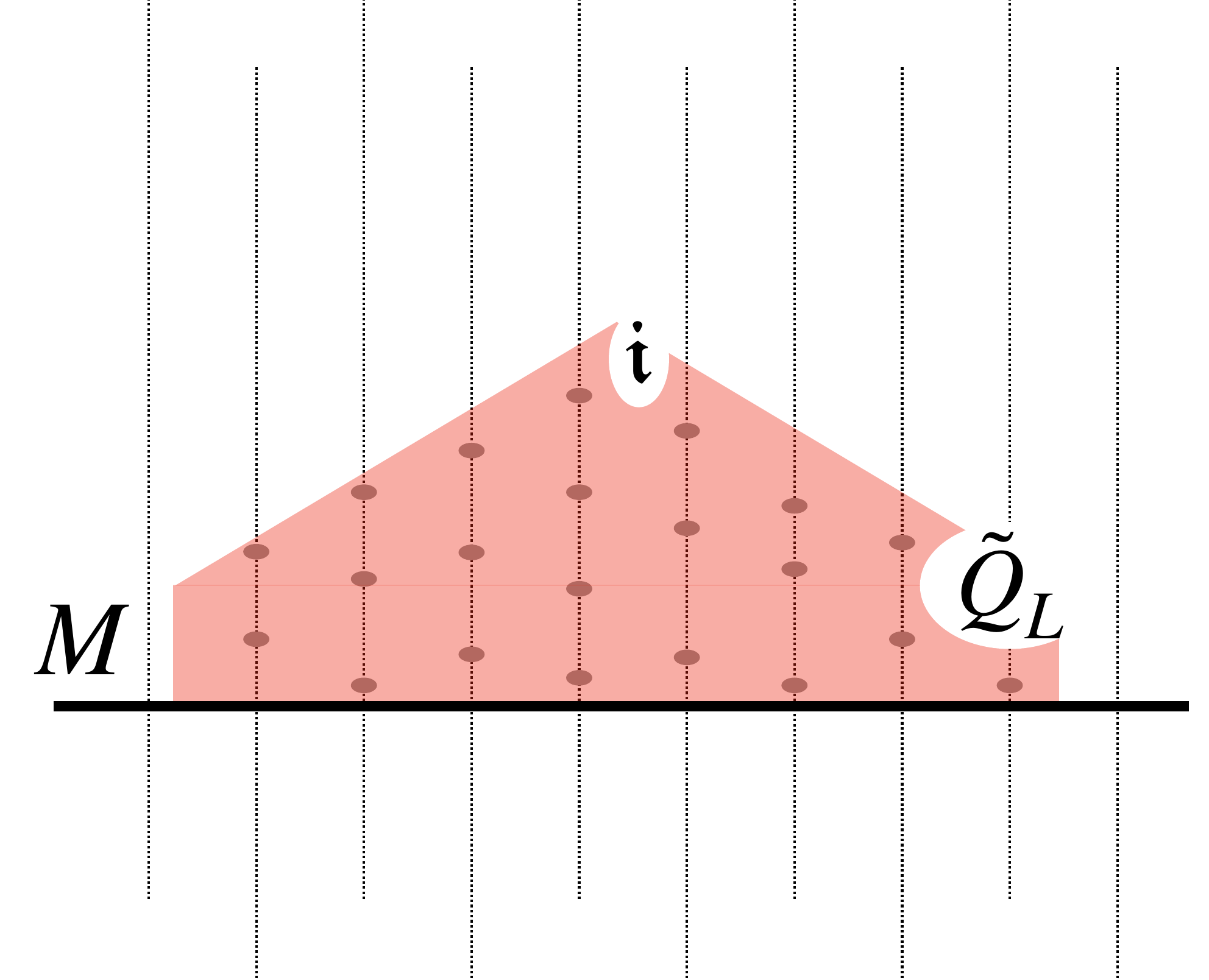}
    \end{minipage}
    \caption{For $Y = Y^L$, the configuration of beads $y_0$ (left) gives rise to the configuration $\yTil_0$ (right). The position of the bead with label $\fraki \in Q_L$ is kept fixed.}\label{fig:QL}
\end{figure}

The process $\YTil$ is the solution of the single boundary problem introduced in \cref{sec:rewiring-solution}, 
with setup $(Q,M)=(\QTil_F, \tilde M_F)$ and using the same noise $W$ as in the definition of $Y$.

Since $\yTil_0 \ge \Pi_{\QTil_F}y_0$, the monotonicity property stated in \cref{thm:monotone} ensures that
$\YTil \ge \Pi_{\QTil_F}Y$ over the whole path. In particular, using also $Y_0(\fraki) = \YTil_0(\fraki)$, we have
\begin{align}\label{eq:76}
A_+(\delta, \varepsilon, \fraki) \subseteq \ATil_+(\delta, \varepsilon, \fraki), \eqdef \Big\{ \sup_{t \in [0, \delta]} (\YTil_{t}(\fraki) - \YTil_{0}(\fraki)) \ge \varepsilon \Big\}\,.
\end{align}

Let $Z$ be the rewiring solution defined in \cref{sec:definition-rewiring} with the initial state $Z_0 = \yTil_0$, and with $Q = \QTil_F$.
By the identity in law, \cref{thm:yeqz}, we have 
\begin{align}
\bbP\big(\ATil_+(\delta, \varepsilon, \fraki)\big)&=\bbP\Big( \sup_{t \in [0, \delta]} (Z_{t}(\fraki) - Z_{0}(\fraki)) \ge \varepsilon \Big)\\
&\leq \bbP\Big( \sup_{\frakj \in \QTil_F} \sup_{t \in [0, \delta]} (\WBar_{t}(\frakj) - \WBar_{0}(\frakj)) \ge \varepsilon \Big) \lesssim |Q_F| \exp (-\varepsilon^{2} / 2\delta)
\end{align}
where the second step is a consequence of the very construction of $Z$ which ensures that $Z_t(\fraki) = \WBar_t(\sigma_t \fraki)$ 
and the fact that we have the initial ordering $\WBar_{0}(\fraki) \ge \WBar_{0}(\frakj)$, while the last 
follows from a union bound over $\frakj \in \QTil_F$, noting that $|\QTil_F| \le |Q_F|$, and then using a simple bound on the tail behaviour of
(the supremum of) a single reflected Brownian Motion, $\WBar(\frakj)$. 
\end{proof}

\subsubsection{Control of collision chains}\label{sec:contr-coll-chains}
In this section we use \cref{lem:2} to control the number of collisions seen by the system over a small time interval $[0, \delta]$. 
To lighten the notation, we write $Y = Y^{(F)}$ for the solution of the truncated problem \eqref{eq:38} in a bounded, non-degenerate region $F \subset F_\infty$, with $y_0 \sim \bfP_\gamma$.

Recall the notion of chain in \cref{def:chain}. We say that a chain $\fraki_{1:M}$ along with a sequence of times $0 \le t_1 \le \ldots \le t_{M-1} \le T$ is a \textit{collision chain} (for $Y$ and over the time interval $[0, T]$) if
$$  Y_{t_i}(\fraki_i) = Y_{t_i}(\fraki_{i+1}) , \, \forall i \in 1{:}M-1. $$
Note that necessarily, at least one among $\fraki_i,\fraki_{i+1}$ belongs to $Q_F$,
since particles with index outside $Q_F$ do not move, and $y_0\in D_\infty$.

\begin{theorem}\label{thm:collisions} For $M \in \bbN$ and $\delta > 0$, define the event 
\begin{equation}\label{eq:103}
B(\delta, M) \eqdef \{ \exists \text{ a collision chain of size } M \text{ over the time interval } [0, \delta] \}.
\end{equation}
For every $\gamma\in(-1,1)$, there exists $C=C(\gamma) > 0$ such that for all $\delta > 0$, $M \in \bbN$ satisfying $\delta M^2 (M-1) > 2$,
\begin{equation}\label{eq:80}
\bbP_\gamma(B(\delta, M)) \le C |F|(1+|F|) e^{- 1/(C\delta M^2)}.
\end{equation}
\end{theorem}

\begin{proof} In the proof, $C(\gamma)$ will denote a positive constant (not necessarily the same at each occurrence) that depends only on $\gamma\in(-1,1)$. We begin by controlling the increments of the beads. Fix a parameter $\varepsilon > 0$ that shall be optimised later, and set
$$ A' \eqdef \bigcup_{\fraki \in Q_F} A(\delta, \varepsilon/2, \fraki) $$
the events at the r.h.s. being as in \eqref{eq:75}. We consider a decomposition
\begin{align}
\bbP_\gamma(B(\delta, M)) &\le \bbP_\gamma(A') + \bbP_\gamma(B(\delta, M)\cap (A')^\complement) \\
&\lesssim_\gamma |F|(1+|F|) \exp (-\varepsilon^{2} / 8\delta) + \bbP_\gamma(B(\delta, M)\cap (A')^\complement) \label{eq:106}
\end{align}
where in the last line we employed a union bound conditional on $y_0$, \cref{lem:2} and \eqref{eq:soshni} with \eqref{boundrhon}, which gives $\bfE_\gamma[|Q_F|^2] \lesssim_\gamma |F|(1+|F|)$.

Let $S(\fraki_{1:M},y)$ be as in \eqref{eq:1}, then we claim that
\begin{equation}\label{eq:5}
B(\delta, M)\cap (A')^\complement\subset\big\{ \exists \text{ chain } \fraki_{1:M}, \, \{\fraki_1,\ldots,\fraki_M\} \cap Q_F \ne \emptyset, S(\fraki_{1:M}, y_0) \le \varepsilon\big\}\,.
\end{equation}
Indeed, if for all chains $\fraki_{1:M}$ satisfying $\{\fraki_1,\ldots,\fraki_M\} \cap Q_F \ne \emptyset$, it holds that $S(\fraki_{1:M}, y_0) > \varepsilon$,
and there has to be a collision for the chain $\fraki_{1:M}$ over the interval $[0, \delta]$, 
then at least one of the beads must travel a distance of $\varepsilon / 2$ in that time. In other words, $A'$ must hold. This implies \eqref{eq:5}. 

Invoking \eqref{eq:5} and \cref{lem:1}, we obtain
\begin{equation}
\bbP_\gamma( B(\delta, M)\cap (A')^\complement) \le |F| \varepsilon^{M-1}[ M C(\gamma)]^M.
\end{equation}
which, once plugged into \eqref{eq:106}, gives
$$ \bbP_\gamma(B(\delta, M)) \le C(\gamma) |F|(1+|F|) \Big( \exp (-\varepsilon^{2} / 8\delta) + [M C(\gamma)]^M \varepsilon^{M-1} \Big)\,. $$
At this point, we optimise $\varepsilon$. Since we assumed $\delta M^2 (M-1) > 2>1$, we may choose, with $C = C(\gamma)$,
$$ \varepsilon = \Big( \frac{1}{CM} \Big)^{M/(M-1)} \Big(1- \frac{1}{\delta M^2 (M-1)} \Big) $$
Using the bound $(1-t)^n \le e^{-nt}$, we deduce $[M C]^M \varepsilon^{M-1} \le \exp(-1/\delta M^2)$. From the definition of $\varepsilon$ and $\delta M^2 (M-1) > 2$  it is not hard to see that $\varepsilon^2\gtrsim_\gamma 1/M^2$. As a consequence, 
$\exp (-\varepsilon^{2} / 8\delta) \le \exp(-1/(C\delta M^2))$ for a modified constant $C = C(\gamma) > 0$,
and thus the statement follows at once. 
\end{proof}

\subsection{Percolation argument}
The goal of this Section is to prove \cref{thm:prob.decay,thm:annulus}. As they will be proved simultaneously, 
let us first introduce a notation that encompasses them both. 
The statement compares two processes $\eta$ and $\eta'$ which have active regions $F$ and $F'$ respectively, 
with $F \subset F'$, or in terms of labels, $Q \subset Q'$. 
In detail, writing $\one$ for \cref{thm:prob.decay} and $\two$ for \cref{thm:annulus}, we set
\begin{align}
(F,F',G) &\eqdef
\begin{cases}
(x_0+F_L,F_{L'},x_0+F_{L/2}) &\one\,, \\
(\FTil_L,F_L,G_L) &\two\,,
\end{cases}\\
(\eta,\eta',Q,Q',Y,Y') &\eqdef (\eta^{(F)},\eta^{(F')},Q_F,Q_{F'},Y^{(F)},Y^{(F')}).
\end{align}
in which $\FTil_L$ is the region defined in \eqref{e:InOutBox}. In addition, in the case of $\two$, $L' = L$. Thus, both statements are proven provided we show that 
in either case, we have 
$$ \bbP_\gamma( ( \eta_t\restr_G )_{t \in [0, T]} \ne ( \eta'_t\restr_G )_{t \in [0, T]} ) \lesssim_T e^{ -(\log L)^2/c} + (L')^4 e^{- c L^2} + L^{\theta} (L')^4 e^{-C(\gamma, T) L^{3\theta - 2}}\,. $$

The strategy is as follows: any disagreement between $\eta$ and $\eta'$ within $G$ can be traced via an $\eta'$-collision chain
to the region $F' \setminus F$ which is active for $\eta'$ but not for $\eta$. Since this collision chain crosses a gap of order $L$, 
we argue that its size is of order $L$ (up to logarithmic corrections). 
Finally, \cref{thm:collisions} implies that collision chains of this size are very unlikely.

In order to form the collision chain, we introduce the infection process $(V_t)_{t \ge 0}$, 
i.e. a monotonically increasing $\{0, 1\}^{Q_\infty}$-valued process, 
in which the initially infected particles are those that lie in $F'\setminus F$. Particles become infected either if they collide with an infected particle, or if they exit the box $F$. The evolution of the infection state is determined by the dynamics of $\eta'$ alone.

In detail, define a sequence of stopping times $\tau_0 < \tau_1 < \ldots $ along with the infection process $V$, which is constant on each interval $[\tau_{n-1}, \tau_n)$, as follows. Set $(V^0_t)_{t \ge 0}$ to be the constant process $V^0 \eqdef \ic\{ \fraki \in Q'\setminus Q \}$ and set $\tau_0 = 0$. For $n \in \bbN$ inductively define 
\begin{align}
\tau_n^{1} &\eqdef \inf \{ t \ge \tau_{n-1}  : \exists \text{ ordered } \fraki, \frakj \in Q', \, \fraki \sim \frakj, \, V_{t}^{n-1}(\fraki) = 0, \, V_{t}^{n-1}(\frakj) = 1, Y'_t(\fraki) = Y'_t(\frakj)\} \\
\tau_n^{2} &\eqdef \inf \{ t \ge \tau_{n-1} : (\frakl(\fraki), Y'_t(\fraki)) \notin F \text{ for some } \fraki \in Q', \, V_{t}^{n-1}(\fraki) = 0  \} \\
\tau_n &\eqdef \min( \tau_n^{1}, \tau_n^{2} ), \qquad\text{and}\qquad V_t \eqdef V^{n-1}_t, \quad t \in [\tau_{n-1}, \tau_n).
\end{align}
\begin{equation}\label{eq:93}
\fraki_n = 
\begin{cases}
\text{the unique } \fraki \in Q' : V^{n-1}_{\tau_{n}}(\fraki) = 0, \, \exists \frakj \in Q', \, \frakj \sim \fraki, \, V^{n-1}_{\tau_{n}}(\frakj) = 1, \, Y'_{\tau_n}(\frakj) = Y'_{\tau_n}(\fraki) \quad &\text{ if } \tau_n = \tau_n^1 \\
\text{the unique } \fraki \in Q' : V^{n-1}_{\tau_{n}}(\fraki) = 0, \, (\frakl(\fraki), Y'_{\tau_n+}(\fraki)) \in \partial F \quad &\text{ if } \tau_n = \tau_n^2
\end{cases}
\end{equation}
and define
$$ V^n_t(\fraki) := V^{n-1}_t(\fraki) \, \vee \, \ic\{ \fraki = \fraki_n, t \ge \tau_n \}.$$
We begin with the following intuitive fact, that if a particle $\fraki \in Q$ is uninfected, then its location is the same for both $\eta$ and $\eta'$. The proposition follows from the uniqueness of the Skorokhod map, and we omit the proof. 
\begin{proposition}\label{prp:8}
If $\fraki \in Q_\infty$, $t \ge 0$ are such that $V_t(\fraki) = 0$, then it holds $Y_t(\fraki) = Y'_t(\fraki)$.
\end{proposition}

Now we are ready to prove \cref{thm:prob.decay,thm:annulus}.

\begin{proof}[Proof of \cref{thm:prob.decay,thm:annulus}]
Suppose there exists $t \in [0, T]$ for which $\eta_t\restr_{G} \ne \eta'_t\restr_{G}$, then there exists $\fraki_0 \in Q_\infty$ 
for which $Y_t(\fraki_0) \ne Y'_{t}(\fraki_0)$ and $Y^{\sharp}_{t}(\fraki_0) \in G$ for $Y^{\sharp}=Y$ or $Y'$. 
In particular, by \cref{prp:8}, it holds $V_{t}(\fraki_0) = 1$.

We now construct a collision chain by following this infected particle $\fraki_0$ backwards in time to a source on the boundary. 
Recall that a collision chain, as defined in \cref{sec:contr-coll-chains} is a chain of distinct neighbouring labels 
$\fraki_0 \sim \ldots \sim \fraki_n$, along with a sequence of collision times, 
that in this proof we shall index in reverse: $t_0 \ge \ldots \ge t_{n-1} \ge 0$. 
The chain is defined inductively as follows. 
Given a bead $\fraki_{i}$ that is infected at some time $t \in [0, T]$, let $t_i$ be the first infection time
$$ t_{i} \eqdef \inf \{ t \in[0,T] : V_t(\fraki_{i}) = 1\}. $$
The infection at time $t_i$ is triggered by one of the following three reasons: 
$t_i = 0$ and $\fraki_{i} \notin Q$; 
$t_i > 0$ and $(\frakl(\fraki_i), Y_{t_{i}}(\fraki_i)) \in \partial F$; or there exists neighbouring $\frakj \sim \fraki$ with $V_{t_{i}-}(\frakj) = 1$ and $Y_{t_{i}}(\frakj) = Y_{t_{i}}(\fraki_i)$. 
In case one and two, terminate the chain, excluding the time $t_i$, for it is not a collision time. 
In case three, we choose $\fraki_{i+1} \eqdef \frakj$ to be the unique such label, and iterate.
The procedure terminates eventually, since there are no repeats in the chain. 

Thus, we have constructed a collision chain $\fraki_{1{:}N}$ with collision times $t_{1{:}N-1}$, 
where the random variable $N \in \bbN$ denotes the length of this specific chain.
We claim that the constructed collision chain is of order $L$, up to logarithmic corrections. 
Specifically, we claim that there exists $L_0\in\bbN$ such that for every $L > L_0$ 
\begin{equation}\label{eq:98}
\bbP_\gamma( \{N < L / (\log L)^2\} \cap A) \lesssim 
e^{ -(\log L)^2/c} + (L')^4 e^{- c L^2}
\end{equation}
where $A$ is the event 
\begin{equation}\label{eq:58}
A = \{|Q| \le (2L+1)^2\}\cap \{Q' \subset P_{L'}\}\,.
\end{equation}
Assuming \eqref{eq:98} (whose proof is given at the end of the section), we deduce 
\begin{align}\label{eq:104}
\bbP_\gamma\big( ( \eta_t\restr_{G} )_{t \in [0, T]} &\ne ( \eta'_t\restr_{G} )_{t \in [0, T]} \big)  \\
&\le \bbP_\gamma\big( A^\complement \big) + \bbP_\gamma\big( \{N < L / (\log L)^2\}\cap A \big) + \bbP_\gamma\big( \{N \ge L / (\log L)^2\} \cap A \big).
\end{align}
Therefore, using \eqref{eq:98} and \cref{prp:11}, 
the theorem follows upon proving
\begin{equation}\label{eq:102}
\bbP_\gamma\big( \{N \ge L / (\log L)^2\} \cap A \big) \le C(\gamma, T) L^{\theta} (L')^4 e^{-C(\gamma, T) L^{3\theta - 2}}\,.
\end{equation}
For this, we fix a parameter $\delta > 0$ that shall be later optimised, 
and consider the mesh $\calT = \{0, \delta, 2\delta, \ldots \} \subset [0, T]$. 
For there to exist a collision chain of size $N \ge L / (\log L)^2$ on the interval $[0, T]$, 
by the pigeon-hole principle, there must exist $t \in \calT$ such that the interval $[t, t+\delta]$ 
contains a collision chain of size $M \ge \delta L / T (\log L)^2$. 
In particular, it must contain a collision chain of size exactly $M \eqdef \lfloor\delta L / T (\log L)^2\rfloor$ (in the following, for lighness, we drop the integer parts).

In other words, defining $B(\delta, M, t)$ as in as in \eqref{eq:103} but for the shifted time interval $[t, t+\delta]$ rather than $[0, \delta]$, and with respect to the process $Y'$, we have
$$ \{N \ge L / (\log L)^2\}\cap \, A \subseteq \bigcup_{t \in \calT} B(\delta, M, t)\cap A\,.$$
By union bound, $|\calT| \le \delta^{-1}T$, and time-stationarity in law of $Y'$, 
together with the fact that the set of active beads $Q'$ does not evolve with time, we have
\begin{align}
\bbP_\gamma( \{N \ge L / (\log L)^2\}\cap A ) &\lesssim_T \delta^{-1} \bbP_\gamma( B( \delta, M, 0 )\cap A ) \le C(\gamma) \delta^{-1} (L')^4 e^{-1/(C(\gamma)\delta M^2)}
\end{align}
where in the last line, we invoked \cref{thm:collisions}. In order to satisfy the condition of the latter statement, 
i.e. $\delta M^2(M-1) > 2$, it is sufficient to choose $\delta = L^{-\theta}$ for $\theta < 3/4$. 
For the bound to be meaningful, we need $\delta M^2 \ll 1$, which allows for the choice 
$\delta = L^{-\theta}$ for some $\theta > 2/3$. Thus, we deduce \eqref{eq:102} for any $\theta \in (\tfrac23, \tfrac34)$.
\end{proof}

\begin{proof}[Proof of \eqref{eq:98}] Define the sets $G^+ \supset G$ and $F^- \subset F$ according to
\begin{equation}\label{eq:63}
G^+ \eqdef
\begin{cases}
x_0 + F_{L / 2 + L/100}, & \one \\
F_{L/6+L/100} \cup (F_L \setminus F_{L/3-L/100}), & \two
\end{cases},
\qquad F^- \eqdef
\begin{cases}
x_0 + F_{L - L/100}, & \one \\
F_{L/5-L/100} \cup (F_L \setminus F_{L/4+L/100}), & \two
\end{cases}
\end{equation}
the key property, as used in the second half of this proof, being that $G^+$ and $(F^-)^c$ are separated by a gap of $L/100$, which is still of order $L$.

By construction, the bead $\fraki_0$ lies in $G$ at time $t_0 \in [0, T]$, therefore it is likely that at time $t = 0$ the bead lies in the slightly larger set $G^+$. Likewise, the bead $\fraki_N$ lies outside $F$ at time $t_N \in [0, T]$, therefore it is likely that it lies outside $F^-$ at time $t = 0$. In detail, we claim there exists $c = c(T)>0$ such that
\begin{equation}\label{eq:94}
\bbP_\gamma(\{(\frakl, y_0(\fraki_0)) \notin G^+\}\cap  A) + \bbP_\gamma(\{(\frakl, y_0(\fraki_N)) \in F^- \}\cap A) \lesssim (L')^4 e^{- c L^2} .
\end{equation}
The two summands can be estimated similarly so we will only focus on the first. 
Recall that $(\frakl(\fraki_0), Y^\sharp_t(\fraki_0)) \in G$ for at least one of 
$Y^\sharp=Y$ or $Y'$, and for some $t \in [0, T]$. If in addition $(\frakl(\fraki_0), Y^\sharp_0(\fraki_0)) \notin G^+$, then 
$\sup_{t \in [0, T]} |Y^\sharp_t(\fraki_0) - Y^\sharp_0(\fraki_0)| \ge L/100$. 
Hence, conditioning on $y_0$ before taking the union bound,
\begin{align}
\bbP_\gamma((\frakl, y_0(\fraki_0)) \notin G^+, \, A ) &\leq \bbE_\gamma\bigg[\ic_A \sum_{\sharp \in \{\emptyset, '\}} \sum_{\fraki \in Q^\sharp} \bbP\Big( \sup_{t \in [0, T]} |Y^\sharp_t(\fraki) - Y^\sharp_0(\fraki)| \ge \tfrac{L}{100} \mid y_0 \Big)\bigg]\\
&\lesssim (L')^4 e^{- L^2 / 20000 T }
\end{align}
where we denoted $Y=Y^\emptyset$ and applied \cref{lem:2}, noting that on $A$, $|Q^\sharp| \lesssim (L')^2$ for both processes.
Employing \eqref{eq:94}, we get
\begin{align}\label{eq:100}
&\bbP_\gamma( \{N < L / (\log L)^2\} \cap A )  \\
& \quad \quad \le \bbP_\gamma( \{N < L / (\log L)^2\}\cap \{(\frakl, y_0(\fraki_0)) \in G^+\}\cap\{(\frakl, y_0(\fraki_N)) \in F \setminus F^-\} ) + (L')^4 e^{- c L^2}\,.
\end{align}
Taking $x_0 = 0$ in case $\two$, for $L$ large enough the event at the r.h.s. of~\eqref{eq:100} is contained in $\{\Delta_L(\tau_{-x_0}\eta_0)>\log L\}$.
Indeed, considering the gap that separates $G^+$ from $F \setminus F^-$, either 
$$ | \frakl(\fraki_0) - \frakl(\fraki_N) | > L/100\quad \text{ or }\quad |y_0(\fraki_0) - y_0(\fraki_N)| > L/100. $$
However, we also have by the triangle inequality that 
$$ | \frakl(\fraki_0) - \frakl(\fraki_N) | \le N\quad\text{ and }\quad |y_0(\fraki_0) - y_0(\fraki_N)| \le N \Delta_L(\tau_{-x_0}\eta_0). $$
For this to be consistent with $N < L / (\log L)^2$ for large $L$, we must have $\Delta_L(\tau_{-x_0}\eta_0) > \log L$.
Therefore, 
\begin{equation}
\bbP_\gamma( \{N < L / (\log L)^2\}\cap \{(\frakl, y_0(\fraki_0)) \in G^+\}\cap\{(\frakl, y_0(\fraki_N)) \in F \setminus F^-\} ) \leq \bbP_\gamma( \Delta_L(\tau_{-x_0}\eta_0) > \log L)
\end{equation}
and the r.h.s. is small thanks to translation invariance of $\bfP_\gamma$ and \cref{prp:12}, which, once plugged into \eqref{eq:100}, concludes the proof.
\end{proof}

\section{Green--Kubo formula}\label{sec:goal-green}

The goal of this section is to prove one of the main results of the paper, 
namely Theorem~\ref{thm:GK}, which provides a Green-Kubo formula 
for the diffusion matrix in~\eqref{eq:DiffMat}.  Notice though that, as written, the expression 
defining the two-point function $S_t$, for $t\geq 0$, is formal 
since the height function $\ch$ is a step function in the vertical direction, 
so that Dirac deltas appear inside the expectation in \eqref{eq:Stij} whenever $i$ or $j = 2$.
Thus, we will begin by providing a rigorous definition of the latter, for which a crucial role will be played by the {\it current}. 

\subsection{The current, its properties and the diffusivity}
\label{sec:thediffusivity}

As is usual for  interacting particle systems, 
the {\it current} at a site $x$ over a time-interval $I$ is the (local) functional 
that counts (with a sign) the number of beads that cross $x$ over $I$. 
Let us give its rigorous definition. 

\begin{definition}\label{def:current}
Let $(\eta_r)_{r\ge0}$ be the dynamical bead model of \cref{cor:1}.
For $x\in F_\infty$ and $0\leq s<t$, the current at $x$ over the interval $[s, t]$ is the functional 
$J=J((\eta_r)_{r \in [s, t]}, x)$ defined according to 
\begin{equation}\label{eq:21}
J((\eta_r)_{r \in [s, t]}, x) \eqdef \sum_{\fraki \in Q_\infty} \ic\{ \frakl(\fraki) = \frakl(x) \} \big(  \ic\{ Y_t(\fraki) \ge y(x) \} - \ic\{ Y_s(\fraki) \ge y(x) \} \big)\,.
\end{equation}
Thus an upward crossing contributes $+1$ to the current and lowers the height by $1$. For the statically pinned height of \cref{def:Height}, this gives
$$ \ch^{x_0}(\eta_t, x) - \ch^{x_0}(\eta_s, x) = -J((\eta_r)_{r \in [s, t]}, x) + J((\eta_r)_{r \in [s, t]}, x_0). $$
\end{definition}

Let us collect in the following statement various properties of the current $J$. 

\begin{proposition}\label{prp:momentsJ}
Let $J$ be the current given according to~\cref{def:current}. For any  $x\in F_\infty$ and $0\leq s<t$, 
$J$ is a local function that depends only on the path segment $(\eta_r)_{r \in [s, t]}$
in an (arbitrarily small) neighbourhood of $x$ and the sum \eqref{eq:21} a.s.\ has only finitely many non-zero terms. 
Moreover, for $T \ge 0$, $\gamma \in (-1, 1)$, it satisfies the following estimates 
\begin{enumerate}
\item  for any $p \ge 1$, $\theta \in (0, \thalf)$ there exists a constant $C = C(\gamma, T, p, \theta) > 0$ 
such that for all $0 \le s \le t \le T$ it holds
\begin{equation}\label{eq:3}
\bbE_\gamma[ |J((\eta_r)_{r \in [s, t]}, x)|^p ] \le C (t-s)^{1/2 - \theta}\,,
\end{equation}
\item there exists a constant $C=C(\gamma,T)>0$ for all $t \in [0, T]$, $x_0, x \in F_\infty$ with $|x_0| \le 1$, it holds
\begin{equation}\label{eq:DecayCross} 
|\bbE_\gamma[J((\eta_s)_{s \in [0, t]}, 0) \ch^0(\tau_x\eta_0, x_0) ]| \le C e^{- C \log(|x|)^2 }\,.
\end{equation}
\end{enumerate}
\end{proposition}

\begin{remark}\label{rem:quasi-local}
Using the same tools and procedure as in the proof of \cref{prp:momentsJ}, one can show 
that for sufficiently local functions $f, g \in L^{2+\varepsilon}(\bfP_\gamma)$, for some $\varepsilon > 0$, 
the bound in \eqref{eq:corrdecay-informal} holds. 
A sufficient condition would be, for example, that the function $f$ (and $g$) becomes local when conditioning on the bead gap not being too large, i.e. if there exists a constant $R > 0$ and a family of local functions $f_L \colon\Omega(F_{R \log L}) \rightarrow \bbR$ such that
$$ f(\omega) \ic\{ \Delta_{R \log L}(\omega) < \log L \} = f_{L}( \omega\restr_{F_{R \log L}} ). $$
\end{remark}

\begin{proof}
The locality immediately follows by definition. To see that the sum has finitely many terms, 
fix $x \in F_\infty$, then the current $J$ in \eqref{eq:21} depends only on the behaviour of $(\eta_r)_{r \in [s, t]}$ 
in the neighbourhood $x + F_1$. From Corollary \ref{cor:1}, $(\eta_r\restr_{x+F_1})_{r \in [0, t]} = (\eta^L_r\restr_{x+F_1})_{r \in [0, t]}$ 
eventually in $L$, therefore there is a large enough $L$ such that $J((\eta_r)_{r \in [s, t]}, x) = J((\eta^L_r)_{r \in [s, t]}, x)$. 
The latter is a finite sum because $\eta^L$ has only a finite number of active particles almost surely.
\medskip

\noindent{\it The estimate \eqref{eq:3}.} By time and spatial stationarity, it is enough to consider $x = 0$ and $s = 0$. 
Let us begin by approximating $\eta$ by the truncated dynamics $\eta^L$, for some $L$ to be fixed later.
Namely, thanks to \eqref{eq:9}, we have
\begin{align}\label{eq:7}
\bbP_\gamma( |J((\eta_r)_{r \in [0, t]}, 0)| > \lambda ) \le e^{- C (\log L)^2} + \bbP_\gamma( |J((\eta^L_r)_{r \in [0, t]}, 0)| > \lambda ) .
\end{align}
For the rest of this proof, the constant $C = C(\gamma, T, p)$ may change from one line to the next.

We estimate the tail probability $\bbP_\gamma( |J((\eta_r)_{r \in [0, t]}, 0)| > \lambda )$. 
We consider separately large $\lambda > \Lambda_t$ and small $\lambda \le \Lambda_t$ for $\Lambda_t$ 
to be determined as a function of $t$. 

For large $\lambda > \Lambda_t$, we consider a truncation at $L = \lambda/2$. Note that if 
$|J((\eta^L_r)_{r \in [0, t]}, 0)| > \lambda$ then $\eta_0( \ic_{\{0\} \times [-L, L]} ) \ge \lambda $, i.e. 
the current can only exceed $\lambda$ if there are more than $\lambda$ active beads in the central column. 
Since $\lambda = 2L$ and the density of beads is $\pi^{-1}$, arguing as in \cref{prp:11}, we get
\begin{align}
\bbP_\gamma( |J((\eta^L_r)_{r \in [0, t]}, 0)| > \lambda ) \le C e^{-C L^2}\,.
\end{align}
Taking the worst growth in $L$, and substituting $L = \lambda / 2$, we obtain 
\begin{equation}\label{eq:14}
\bbP_\gamma( |J((\eta_r)_{r \in [0, t]}, 0)| > \lambda ) \le e^{- C (\log \lambda )^2}, \qquad \text{for } \lambda > \Lambda_t\,.
\end{equation}
Instead, for $1 \le \lambda \le \Lambda_t$, we truncate at a value of $L = L(t)$. 
We introduce $\varepsilon > 0$, and consider the event 
$ B_\varepsilon = \{ \eta_0\restr_{[-\varepsilon, \varepsilon]} = \emptyset \}$ 
which, by integrating the one-point density over $[-\varepsilon, \varepsilon]$ and applying Markov's inequality, satisfies
$\bbP( B_\varepsilon^c ) \le C \varepsilon$. 
Then, 
\begin{align}
\bbP_\gamma( |J((\eta_r)_{r \in [0, t]}, 0)| > \lambda ) \le e^{- C (\log L)^2} + C \varepsilon + \bbP_\gamma( \{|J((\eta^L_r)_{r \in [0, t]}, 0)| > \lambda\}\cap B_\varepsilon) \,.
\end{align}
Now, since $\lambda \ge 1$, for $A(t, \varepsilon, \fraki)$ as in \eqref{eq:75}, we have
$$ \{|J((\eta^L_r)_{r \in [0, t]}, 0)| > \lambda\}\cap B_\varepsilon \subseteq \bigcup_{\fraki \in Q_L} A(t, \varepsilon, \fraki)\,.$$
In words, for the current to be non-zero, if the $\varepsilon$-neighbourhood of the origin has no beads, then at least one of the particles needs to move a distance of $\varepsilon$. Using a union bound and \cref{lem:2}, and conditioning further on $Q_L \subset P_L$, which has lower order thanks to \cref{prp:11}, we obtain 
$$ \bbP_\gamma( |J((\eta_r)_{r \in [0, t]}, 0)| > \lambda ) \le e^{- C (\log L)^2} + C \varepsilon + C L^4 \exp (-\varepsilon^{2} / 2t). $$
At this point, choose $\varepsilon = t^{1/2 - \theta}$, so that 
\begin{align}
\bbP_\gamma( |J((\eta_r)_{r \in [0, t]}, 0)| > \lambda ) &\le e^{- C (\log L)^2} + C t^{1/2 - \theta} + C L^4 \exp (- t^{- 2 \theta} / 2 ) \\
&\le e^{- C (\log L)^2} + C L^4 t^{1/2 - \theta}
\end{align}
and, setting $L = t^{-\theta/4}$, we conclude
\begin{equation}\label{eq:12}
\bbP_\gamma( |J((\eta_r)_{r \in [0, t]}, 0)| > \lambda ) \le C t^{1/2 - 2\theta}, \qquad \text{for } 1 \le \lambda \le \Lambda_t\,.
\end{equation}
Combining \eqref{eq:14} and \eqref{eq:12}, we get for $p \in \bbN$, 
\begin{align}
\bbE_\gamma[ |J((\eta_r)_{r \in [0, t]}, 0)|^p ] &\le \sum_{\lambda = 1}^{\Lambda_t} \lambda^{p-1} \bbP_\gamma( |J((\eta_r)_{r \in [0, t]}, 0)| > \lambda )  + \sum_{\lambda = \Lambda_t}^{\infty} \lambda^{p-1} \bbP_\gamma( |J((\eta_r)_{r \in [0, t]}, 0)| > \lambda ) \\
&\le C t^{1/2 - 2\theta} \sum_{\lambda = 1}^{\Lambda_t} \lambda^{p-1}  + \sum_{\lambda = \Lambda_t}^{\infty} \lambda^{p-1} e^{- C (\log \lambda )^2}  \\
&\le C \Lambda_t^p t^{1/2 - 2\theta} + C e^{- C (\log \Lambda_t)^2}.
\end{align}
The bound follows upon setting $\Lambda_t = t^{-\theta/p}$ and replacing $\theta$ by $3\theta$.
\medskip

\noindent{\it The estimate \eqref{eq:DecayCross}.} First, we localise the height function. Setting $N = |x| / 4$, define
$$ g(\eta) = \ch^0(\eta, x_0) \ic\{ \Delta_N(\eta) \le \log N \} $$
so that $g(\eta)$ is a function of $\eta\restr_{F_N}$. 
Further, consider the event 
\begin{equation}\label{eq:11}
A \eqdef \{( \eta_t\restr_{F_{L/2}} )_{t \in [0, T]} = ( \eta^{L}_t\restr_{F_{L/2}} )_{t \in [0, T]} \}.
\end{equation}
Upon defining $B = \{ \Delta_N(\tau_x\eta) \le \log N \}$, \eqref{eq:9} and Proposition \ref{prp:12} imply that 
$$\bbP\big( (A \cap B)^\complement\big)^{1/2} \le Ce^{- C \log(|x|)^2} .$$
Writing $J_0 \eqdef J((\eta_s)_{s \in [0, t]}, 0)$ and $J^L_0 \eqdef J((\eta^L_s)_{s \in [0, t]}, 0)$, by Cauchy-Schwarz we get 
\begin{align}
\big| \bbE_\gamma[ J_0 \ch^0(\tau_x\eta_0, x_0) ] \big| &\le \big| \bbE_\gamma[J_0 \ch^0(\tau_x\eta_0, x_0) \ic_{A \cap B} ] \big| + Ce^{- C \log(|x|)^2}\bigl( \bbE_\gamma[ J_0^4 ] \bfE_\gamma[ \ch^0(\eta, x_0)^4 ] \bigr)^{1/4} \\
&= \big| \bbE_\gamma[J^L_0 g(\tau_x\eta_0) \ic_{A \cap B} ] \big| + Ce^{- C \log(|x|)^2}\bigl( \bbE_\gamma[ J_0^4 ] \bfE_\gamma[ \ch^0(\eta, x_0)^4 ] \bigr)^{1/4} \\
&\le \big| \bbE_\gamma[J^L_0 g(\tau_x\eta_0) ] \big| + 2 Ce^{- C \log(|x|)^2}\bigl( \bbE_\gamma[ J_0^4 ] \bfE_\gamma[ \ch^0(\eta, x_0)^4 ] \bigr)^{1/4} .
\end{align}
The moments on the r.h.s.\ can all be bounded uniformly over $x_0 \in F_\infty$ with $|x_0| \le 1$, 
therefore the claim follows upon proving
$ \bbE_\gamma[J^L_0 g(\tau_x\eta_0) ] = 0 $. 
To see this, we write 
$$ J^L_0 = \ch^{x_1}(\eta^L_0, 0) - \ch^{x_1}(\eta^L_t, 0) $$
which holds provided that the height is pinned at a point $x_1 \in F_\infty \setminus F_L$ outside of the active region for $\eta^L$. Moreover, by construction $g(\tau_x\eta_0)$ depends on a static region of the dynamics, therefore $g(\tau_x\eta_0) = g(\tau_x\eta_t)$. Together this gives
\begin{align}
\bbE_\gamma[J^L_0 g(\tau_x\eta_0) ] &= \bbE_\gamma[\ch^{x_1}(\eta^L_0, 0) g(\tau_x\eta^L_0) ] - \bbE_\gamma[\ch^{x_1}(\eta^L_t, 0) g(\tau_x\eta^L_0) ] \\
&= \bbE_\gamma[\ch^{x_1}(\eta^L_0, 0) g(\tau_x\eta^L_0) ] - \bbE_\gamma[\ch^{x_1}(\eta^L_t, 0) g(\tau_x\eta^L_t) ]
\end{align}
which is zero because both terms are expectations against the stationary measure.
\end{proof}

Now that the current has been defined, let us introduce a suitable class of 
distributions on $F_\infty$. For $x_0 \in F_\infty$, define the distribution $U^{x_0}_t \in (C^\infty_c(F_\infty))'$ according to
\begin{align}\label{eq:22}
U^{x_0}_t[f] \eqdef \int f(x) u^{x_0}_t(-x)\dd x, \quad\text{and}\quad u^{x_0}_t(x) \eqdef \bbE_\gamma[J((\eta_s)_{s \in [0, t]}, 0) \ch^{x_0}(\eta_0, x) ] . 
\end{align}
Because both $J$ and $\ch^{x_0}$ have all moments (thanks to the previous Proposition), 
$U^{x_0}_t$ is well defined, with $|U^{x_0}_t[f]| \lesssim \lVert f \rVert_\infty$. Moreover, if $f \in C^\infty_0(F_\infty)$, 
then $U^{x_0}_t[f]$ is independent of $x_0 \in F_\infty$. In the next statement, we study the decay properties of $u$. 

\begin{proposition}\label{prp:3}
For $T \ge 0$, $\gamma \in (0,1)$, there exists a constant $C = C(\gamma, T) > 0$ such that for all $t \in [0, T]$,
$x, x_0 \in F_\infty$,
\begin{equation}\label{eq:13}
|u^{x_0}_t(x)| \le C e^{- C \log(|x| \wedge |x_0|)^2 }\,.
\end{equation}
In particular, there exists a pointwise limit
\begin{align}\label{eq:82}
u_t(x) \eqdef \lim_{x_0 \rightarrow \infty} u^{x_0}_t(x),
\end{align}
which satisfies 
\begin{equation}\label{eq:97}
|u_t(x) - u^{x_0}_t(x)| \le C e^{- C \log(|x_0|)^2}\qquad\text{and}\qquad |u_t(x)| \le C e^{- C \log(|x|)^2}\,.
\end{equation}
We define the distribution $U_t \in (C^\infty_p(F_\infty))'$ according to \eqref{eq:22} but with $u_t$ replacing $u^{x_0}_t$. 
\end{proposition}
\begin{proof}
Let $(x_n)_{n = 0}^N$ be a sequence of sites within $F_\infty$ starting from $x_0$ and 
such that $x_N = x$ and $|x_n - x_{n-1}| \le 1$ for all $n \in 1{:}N$. 
We choose the path that stays outside of the ball $F_{|x| \wedge |x_0|}$ and whose points are pairwise distinct.

By expanding as a telescopic sum and then translating the height function, we have
$$ \ch^{x_0}(\eta_0, x) = \sum_{n = 1}^{N} \ch^{x_{n-1}}(\eta_0, x_n) = \sum_{n = 1}^{N} \ch^0(\tau_{-x_{n-1}}\eta_0, x_n - x_{n-1}) $$
therefore
\begin{align}
|\bbE_\gamma[J((\eta_s)_{s \in [0, t]}, 0) \ch^{x_0}(\eta_0, x) ]| &\le \sum_{n = 1}^{N} |\bbE_\gamma[J((\eta_s)_{s \in [0, t]}, 0) \ch^0(\tau_{-x_{n-1}}\eta_0, x_n - x_{n-1}) ]| \\
&\le \sum_{n = 1}^{N} C e^{-C \log(|x_n|)^2} \le \sum_{|x'| \ge |x| \wedge |x_0|} C e^{-C \log(|x'|)^2} 
\end{align}
where in the last line we used \eqref{eq:DecayCross} and the bound \eqref{eq:13} follows from a change of $C$.

For the existence of the pointwise limit, one sees from the additivity that $u^{x_0}_t(x) - u^{x_1}_t(x) = u^{x_0}_t(x_1)$, therefore considering \eqref{eq:13} we have
\begin{equation}\label{eq:105}
|u^{x_0}_t(x) - u^{x_1}_t(x)| = |u^{x_0}_t(x_1)| \le C e^{- C \log(|x_1| \wedge |x_0|)^2 }
\end{equation}
This shows the sequence $(u_t^{x_0}(x))_{x_0 \in F_\infty}$ is Cauchy for fixed $x \in F_\infty$. 
Taking a limit of \eqref{eq:105} as $x_0 \rightarrow \infty$ gives the second bound in \eqref{eq:97}, while the limit 
as $x_1 \rightarrow \infty$ gives \eqref{eq:97}.

At last, note that the decay in \eqref{eq:82} allows to test $U$  
against functions $f \in C^\infty_p(F_\infty)$ that are, along with all of their derivatives, of at most polynomial growth, 
and thus $U\in (C_p^\infty(F_\infty))'$. 
\end{proof}

We are now ready to rigorously define the two-point function $S$ and the diffusivity $D$. 
For test functions $f$, use $\nabla_i$ as in \eqref{eq:nablaH}, and let
$$ \nabla_1^- f(\frakl,y) \eqdef f(\frakl,y) - f(\frakl-1,y), \qquad \nabla_2^- f \eqdef \nabla_2 f. $$
Thus summation or integration by parts takes the form $\int f \, \nabla_i g \, \dd x = -\int (\nabla_i^- f) g \, \dd x$.

\begin{definition}\label{def:Diffu}
Given $f \in C^\infty_p(F_\infty)$, we define
\begin{equation}\label{eq:30}
S_t[f]\eqdef \sum_{i,j=1}^2 a_{ij}(\gamma) S_{t, ij}[f] \,,\qquad\text{where}\qquad S_{t, ij}[f] \eqdef U_t[ \nabla_j\nabla_i^- f ]
\end{equation}
where $U_t\in (C^\infty_p(F_\infty))'$ is defined according to~\cref{prp:3}. 
The diffusivity matrix $D = (D_{IJ})_{I,J \in \{1, 2\}}$ is then given according to 
\begin{equation}
\label{eq:83}
D_{IJ}(t) \eqdef t^{-1} S_t[ f_{IJ} ] \quad \text{ in which } \quad f_{IJ}(x) \eqdef x_I x_J.
\end{equation}
\end{definition}

The next proposition tells us that $S_{t, ij}$ defined in \eqref{eq:30} is the correct distributional interpretation of \eqref{eq:Stij}.

\begin{proposition}\label{prp:2}
Let $\rho \in C^\infty_c(\bbR)$ be a unit mollifier, i.e. $\int \rho \, \dd y = 1$. 
Write $S^\varepsilon$ for the $S$ of \eqref{eq:Stij} in which the height function $\ch$ is replaced by 
$\ch_\varepsilon = \rho_\varepsilon * \ch$, where the convolution acts only in the continuous direction and 
in which $\rho_\varepsilon$ is given by $\rho_\varepsilon \eqdef \varepsilon^{-1}\rho(\cdot/\varepsilon)$ 
(so that, as $\varepsilon\to 0$, $\rho_{\varepsilon}$ converges to a delta function at the origin). 
Then, 
\begin{equation}\label{eq:96}
\lim_{\varepsilon\to0}S^\varepsilon_{t, ij}[f] = S_{t, ij}[ f ], \qquad \forall f \in C^\infty_c(F_\infty), 
\end{equation}
the r.h.s. being independent of $\rho$. 
\end{proposition}
\begin{proof}
We will use the pathwise identities
\begin{align}
\nabla_i \ch(\eta_t, x) - \nabla_i \ch(\eta_0, x) &= -\nabla_i J( (\eta_s)_{s \in [0, t]}, x ) \label{eq:20} \\
J((\eta_s)_{s \in [0, t]}, x) &= J((\tau_{-x}\eta_s)_{s \in [0, t]}, 0) \label{eq:19} 
\end{align}
We obtain \eqref{eq:96} from \eqref{eq:Stij} as follows: writing $J_\varepsilon(x) = \rho_\varepsilon * J(x)$,
\begin{align}
S^\varepsilon_{t, ij}[f] &= \int f(x) \bbE_\gamma[( \nabla_i \ch_\varepsilon(\eta_t, x) - \nabla_i \ch_\varepsilon(\eta_0, x) ) \nabla_j \ch_\varepsilon(\eta_0, 0) ] \dd x \\
&= - \int f(x) \bbE_\gamma[ \nabla_i J_\varepsilon((\eta_s)_{s \in [0, t]}, x) \nabla_j \ch_\varepsilon(\eta_0, 0) ] \dd x \\
&= \int \nabla_i^- f(x) \bbE_\gamma[ J_\varepsilon((\eta_s)_{s \in [0, t]}, x) \nabla_j \ch_\varepsilon(\eta_0, 0) ] \dd x \\
&= \int \nabla_i^- f(x) \bbE_\gamma[ J_\varepsilon((\tau_{-x}\eta_s)_{s \in [0, t]}, 0) \nabla_j \ch_\varepsilon(\eta_0, 0) ] \dd x \\
&= \int \nabla_i^- f(x) \bbE_\gamma[ J_\varepsilon((\eta_s)_{s \in [0, t]}, 0) \nabla_j \ch_\varepsilon(\tau_x\eta_0, 0) ] \dd x \\
&= \int \nabla_i^- f(x) \bbE_\gamma[ J_\varepsilon((\eta_s)_{s \in [0, t]}, 0) \nabla_j \ch_\varepsilon(\eta_0, -x) ] \dd x \\
&= \int \nabla_j\nabla_i^- f(x) \bbE_\gamma[ J_\varepsilon((\eta_s)_{s \in [0, t]}, 0) \ch^{x_0}_\varepsilon(\eta_0, -x) ] \dd x \\
& \stackrel{\varepsilon\to0}\longrightarrow \int \nabla_j\nabla_i^- f(x) \bbE_\gamma[ J((\eta_s)_{s \in [0, t]}, 0) \ch^{x_0}(\eta_0, -x) ] \dd x = U^{x_0}_t[ \nabla_j\nabla_i^- f]\,.
\end{align}
Here we used \eqref{eq:20}, summation or integration by parts, \eqref{eq:19}, and spatial stationarity. In the last integration by parts, the reflection $x \mapsto -x$ contributes a second minus sign and gives the forward difference $\nabla_j$ acting on $\nabla_i^- f$.

Finally, since $\nabla_j\nabla_i^- f$ has zero total mass, $U^{x_0}_t[ \nabla_j\nabla_i^- f]$ is independent of $x_0 \in F_\infty$, and
sending $x_0 \rightarrow \infty$ yields \eqref{eq:96}.
\end{proof}

\subsection{Decomposition into static/dynamic}\label{sec:we-now-decompose}
We now decompose the diffusivity of \eqref{eq:83} into a static and dynamic contribution. 
To do so, we first show in the next statement that such a decomposition is available for 
the distribution $U$ in~\cref{prp:3}. 

\begin{proposition}\label{prp:1}
For all $\kappa > 0$, $t \in \kappa\bbN$ and $f\in C^\infty_p(F_\infty)$, it holds
\begin{equation}\label{eq:31}
U_t[f] = \frac{t}{\kappa} U_\kappa[f] + V_{t, \kappa}[f]\,,
\end{equation}
where $U$ is defined according to \cref{prp:3} and $V_{t,\kappa}\in (C^\infty_p(F_\infty))'$ is given by 
\begin{align}\label{eq:44}
V_{t, \kappa}[ f ] \eqdef \sum_{s \in \kappa \bbN \cap [0, t)} \int f(x) \bbE_\gamma[J( (\eta_{r})_{r \in [0, \kappa]}, 0) J((\eta_r)_{r \in [\kappa, s+\kappa]}, -x) ] \dd x\,.
\end{align}
\end{proposition}

We see that as $\kappa \rightarrow 0$, the distribution $U_\kappa$ depends on the process in a vanishing time interval, and as such, we refer to $\frac{t}{\kappa} U_\kappa$ as the static term. On the other hand, $V_{t, \kappa}$ remains dependent on the process over the full interval $s \in [0, t]$, therefore we refer to it as the dynamic term. 

Before proving \cref{prp:1}, we prove the following Lemma. In spirit, it is similar to the bound \eqref{eq:DecayCross}, 
but we allow both functionals to depend on the path.
Ultimately, this is the reason why the annulus dynamics was introduced. 

\begin{lemma}\label{lem:4}
For $T \ge 0$ and $L \in \bbN$, let $K$ be a measurable function on $C([0, T], \OmegaBar(F_L \setminus F_{L/3}))$. Then, for any time interval $0 \le s \le t \le T$, it holds
$$ |\bbE_\gamma[J((\eta^L_r)_{r \in [s,t]}, 0 ) K( (\eta^L_r\restr_{F_L \setminus F_{L/3}})_{r \in [0, T]} ) ]| \lesssim e^{-C \log(L)^2} (\bbE_\gamma[ |K|^4 ])^{1/4} \,. $$
\end{lemma}

\begin{proof}
Let $\etaTil^L$ be the annulus dynamics from \cref{sec:truncated-dynamics}, and write
\begin{align}
\JTil \eqdef J((\etaTil^L_r)_{r \in [s,t]}, 0), \quad \KTil \eqdef K((\etaTil^L_r\restr_{F_L \setminus F_{L/3}})_{r \in [0,T]}).
\end{align}
We condition on the configuration in the static region,
$$ Z \eqdef \eta_0\restr_{(F_{L/4} \setminus F_{L/5}) \cup F_L^\complement}. $$
Conditional on $Z$, the dynamics of $\etaTil^L$ in the two disjoint active regions $F_{L/5}$ and $F_L \setminus F_{L/4}$ are independent. Therefore, $\JTil$ and $\KTil$ are conditionally independent, and
$$ \bbE_\gamma[\JTil \KTil | Z ] = \bbE_\gamma[\JTil | Z ] \bbE_\gamma[\KTil | Z ]. $$
By stationarity of the inner dynamics conditional on $Z$, $\bbE_\gamma[\JTil | Z ] = 0$, and thus $\bbE_\gamma[\JTil \KTil] = 0$.

Now let
$$ E \eqdef \{(\eta^L_r\restr_{G_L})_{r \in [0,T]} = (\etaTil^L_r\restr_{G_L})_{r \in [0,T]}\}. $$
On $E$, the locality of the current and the measurability assumption on $K$ imply
$$ J((\eta^L_r)_{r \in [s,t]}, 0) K((\eta^L_r\restr_{F_L \setminus F_{L/3}})_{r \in [0,T]}) = \JTil \KTil. $$
Hence, by H\"older's inequality, \cref{thm:annulus}, and the uniform fourth moment bound for the current,
\begin{align}
&\left|\bbE_\gamma\Big[J((\eta^L_r)_{r \in [s,t]}, 0) K((\eta^L_r\restr_{F_L \setminus F_{L/3}})_{r \in [0,T]})\Big]\right| \\
&\qquad \le \left|\bbE_\gamma\Big[\ic_{E^\complement}J((\eta^L_r)_{r \in [s,t]}, 0) K((\eta^L_r\restr_{F_L \setminus F_{L/3}})_{r \in [0,T]})\Big]\right| + |\bbE_\gamma[\ic_{E^\complement}\JTil \KTil]| \\
&\qquad \lesssim e^{-C \log(L)^2} (\bbE_\gamma[|K|^4])^{1/4},
\end{align}
which proves the claim.
\end{proof}

In the following proof we make use of the dynamically pinned height function. 
For an $\OmegaBar$-valued process $(\eta_t)_{t \ge 0}$, define the dynamic height function 
$\ch_t(x) = \ch(\eta_0,x) - J( (\eta_{r})_{r \in [0, t]}, x)$. In addition, for $x_0 \in F_\infty$, 
its pinned to $x_0$ at time zero version is $\ch^{x_0}_t(x) = \ch^{x_0}(\eta_0,x) - J( (\eta_{r})_{r \in [0, t]}, x)$.

\begin{remark}
The difference between the dynamic height function $\ch_t$ and the height function $\ch$ that was defined in \cref{def:Height} is given by the current at the pinned point: for all $x \in F_\infty$
$$ \ch^{x_0}_s(x) - \ch^{x_0}(\eta_s, x) = -J((\eta_r)_{r \in [0, s]}, x_0). $$
In particular, it is independent of $x$ and the height gradients of the two coincide.
\end{remark}

\begin{proof}[Proof of \cref{prp:1}]
It is enough to prove the statement for $f \in C^\infty_c(F_\infty)$. 
In particular, \eqref{eq:31} implies that $V_{t, \kappa}$ shares the same super-polynomial decay as $U$.  

By \eqref{eq:21}, we can write
$$ J((\eta_s)_{s \in [0, t]}, 0) = \sum_{s \in \kappa \bbN \cap [0, t)} J((\eta_{s + r})_{r \in [0, \kappa]}, 0)\,. $$
Therefore, for a test function $f \in C^\infty_c(F_\infty)$ and for $x_0 \in F_\infty$, we get
\begin{align}
U^{x_0}_t[f] &= \sum_{s \in \kappa \bbN \cap [0, t)} \int f(x) \bbE_\gamma[J((\eta_{s + r})_{r \in [0, \kappa]}, 0) \ch^{x_0}(\eta_0, -x) ] \dd x \\
&= \sum_{s \in \kappa \bbN \cap [0, t)} \int f(x) \bbE_\gamma[J((\eta_{s + r})_{r \in [0, \kappa]}, 0) \ch^{x_0}_0(-x) ] \dd x \\
&= \sum_{s \in \kappa \bbN \cap [0, t)} \int f(x) \bbE_\gamma[J((\eta_{s + r})_{r \in [0, \kappa]}, 0) (\ch^{x_0}_0(-x) - \ch^{x_0}_s(-x)) ] \dd x \label{eq:24} \\
&\qquad + \sum_{s \in \kappa \bbN \cap [0, t)} \int f(x) \bbE_\gamma[J((\eta_{s + r})_{r \in [0, \kappa]}, 0) \ch^{x_0}_s(-x) ] \dd x \,.\label{eq:25}
\end{align}
First we re-express \eqref{eq:24}. We have
\begin{align}
\bbE_\gamma[J((\eta_{s + r})_{r \in [0, \kappa]}, 0) (\ch^{x_0}_0(-x) - \ch^{x_0}_s(-x)) ] &= \bbE_\gamma[J( (\eta_{s+r})_{r \in [0, \kappa]}, 0) J((\eta_r)_{r \in [0, s]}, -x) ] \\
&= \bbE_\gamma[J( (\eta_{r})_{r \in [0, \kappa]}, 0) J((\eta_r)_{r \in [\kappa, s+\kappa]}, -x) ] \label{eq:4}
\end{align}
where we used reversibility from \cref{cor:1}. Time reversal changes the sign of each current, so their product keeps its sign.
For \eqref{eq:25}, we have
\begin{align}
\bbE_\gamma[J((\eta_{s + r})_{r \in [0, \kappa]}, 0) \ch^{x_0}_s(-x) ] =& \bbE_\gamma[J((\eta_{s + r})_{r \in [0, \kappa]}, 0) \ch^{x_0}(\eta_s, -x) ] \\
& + \bbE_\gamma[J((\eta_{s + r})_{r \in [0, \kappa]}, 0) (\ch^{x_0}_s(-x) - \ch^{x_0}(\eta_s, -x)) ]\,.
\end{align}
For the first summand, time-stationarity gives 
$$ \bbE_\gamma[J((\eta_{s + r})_{r \in [0, \kappa]}, 0) \ch^{x_0}(\eta_s, -x) ] = \bbE_\gamma[J((\eta_{r})_{r \in [0, \kappa]}, 0) \ch^{x_0}(\eta_0, -x) ] $$
and it further holds that the following difference in the second is independent of $x$
$$ \ch^{x_0}_s(-x) - \ch^{x_0}(\eta_s, -x) = -J((\eta_r)_{r \in [0, s]}, x_0). $$

Combining the above, we deduce 
$$ U^{x_0}_t[f] = \frac{t}{\kappa} U^{x_0}_\kappa[f] + V_{t, \kappa}[f] - \Big(\int f \dd x\Big) \sum_{s \in \kappa \bbN \cap [0, t)}\bbE_\gamma[J((\eta_{s+r})_{r \in [0, \kappa]}, 0) J((\eta_r)_{r \in [0, s]}, x_0) ]. $$
The result follows provided we show that 
\begin{equation}\label{eq:26}
|\bbE_\gamma[J((\eta_{s+r})_{r \in [0, \kappa]}, 0) J((\eta_r)_{r \in [0, s]}, x_0) ]| \lesssim e^{- C (\log |x_0|)^2}
\end{equation}
holds for some $C > 0$, and then sending $x_0 \rightarrow \infty$. To see the above, one has to first 
replace the dynamics $\eta$ with the truncated dynamics $\eta^L$ at scale $L = 2 |x_0|$, 
which gives the same decay as that at the r.h.s.\ of \eqref{eq:26} because $J$ has all moments by \cref{prp:momentsJ} 
and by \cref{cor:1}. Then, invoke \cref{lem:4}.
\end{proof}

As a consequence of \cref{prp:1}, we deduce the claimed decomposition of the diffusivity. 

\begin{proposition}\label{prp:DiffuDecom}
Let $D$ be the diffusivity in~\eqref{eq:83} and $a=a(\gamma)$ the matrix in \eqref{eq:a.bead}. For every $t > 0$, the following identity holds
\begin{equation}\label{eq:10}
D_{IJ}(t) = (a_{IJ} + a_{JI}) \lim_{R\rightarrow\infty} (\kappa^{-1} U^{x_0}_\kappa[ \chi_R ] + t^{-1} V_{t, \kappa} [ \chi_R ])
\end{equation}
in which $U^{x_0}$ and $V$ are respectively given as in~\eqref{eq:22} and~\eqref{eq:44}, 
$\kappa = \kappa(R)=t/\lceil t R^{100}\rceil$, $x_0 = (2R, 2R)$ and, for $R \in \bbN$, $\chi_R$ is the cutoff function
\begin{equation}\label{eq:moll}
\chi_R(\frakl,y) \eqdef \ic\{ |\frakl| \le R \} \big(\rho * \ic_{[-R-\thalf,R+\thalf]}\big)(y)
\end{equation}
for an even, non-negative $\rho \in C^\infty_c(\bbR)$ satisfying $\supp (\rho) \subset (-\thalf, \thalf)$ and $\int \rho(y) \, \dd y = 1$. Thus $\chi_R(-x) = \chi_R(x)$ and $\ic_{F_R} \le \chi_R \le \ic_{F_{R+1}}$. The choice of $\kappa$ ensures $t/\kappa \in \bbN$ and 
$\kappa =O( R^{-100})$.
\end{proposition}
\begin{proof}
For $f_{IJ}(x) = x_I x_J$, we have $\nabla_j\nabla_i^- f_{IJ} = \delta_{iI}\delta_{jJ} + \delta_{iJ}\delta_{jI}$. Hence \eqref{eq:30} and \eqref{eq:83} give
$$ D_{IJ}(t) = (a_{IJ} + a_{JI}) t^{-1} U_t[1]. $$
By the decay in \eqref{eq:82}, $U_t[1] = \lim_{R\rightarrow\infty} U_t[\chi_R]$.
Thanks to \cref{prp:1} we have
$$ U_t[ \chi_R ] = \frac{t}{\kappa} U_\kappa[ \chi_R ] + V_{t, \kappa}[ \chi_R ] $$
and since, by \eqref{eq:97}, we know that 
$$ | U_\kappa[ \chi_R ] -  U^{x_0}_\kappa[ \chi_R ] | \le C e^{- C \log(|x_0|)^2} \int |\chi_R| \dd x  $$
and the r.h.s. goes to zero as $R \rightarrow \infty$ even after multiplication by $t\kappa^{-1}$, the statement follows. 
\end{proof}

Thanks to~\cref{prp:DiffuDecom},~\cref{thm:GK} follows directly from the next statement which 
separately analyses the two terms at the r.h.s. of~\eqref{eq:10}. 

\begin{proposition}\label{prp:GKsep}
For $t > 0$, recall the definitions of $\chi_R$ in~\eqref{eq:moll},
$U^{x_0}_t$ in~\eqref{eq:22} and $V_{t,\kappa}$ in~\eqref{eq:44}. Then, upon choosing $\kappa=\kappa(R)= t/\lceil t R^{100}\rceil$
and $x_0=(2R,2R)$,
the following limits hold 
\begin{align}
\lim_{R\rightarrow\infty} \kappa^{-1} U^{x_0}_\kappa[ \chi_R ] &= \frac{1}{2\pi}\,,\label{eq:static}\\
\lim_{R \rightarrow \infty} V_{t, \kappa} [ \chi_R ] &= 0\,. \label{eq:dynamic}
\end{align}
\end{proposition}

The proof of the previous proposition will occupy the remainder of the section. 
Before that, let us briefly conclude the proof 
of the Green-Kubo formula stated in~\cref{thm:GK}. 

\begin{proof}[Proof of Theorem~\ref{thm:GK}. ]
Substituting \eqref{eq:static} and \eqref{eq:dynamic} into \eqref{eq:10} gives, for every $t > 0$,
$$ D_{IJ}(t) = \frac{a_{IJ} + a_{JI}}{2\pi} = \frac{a_{IJ}}{\pi}, $$
where the last equality uses the symmetry of $a$.
\end{proof}

\subsection{The static term}\label{sec:static-term}
This subsection is devoted to the proof of~\eqref{eq:static}. 
For this, we introduce a perturbation of the finite system \eqref{eq:38} in which each particle receives an additional drift $-c$, for $c \in \bbR$.
One can think of the parameter $c \in \bbR$ as representing an external driving field.

In detail, fix $c \in \bbR$, $T \ge 0$ and recall from \cref{sec:finite-volume-dynam} the notation introduced for the SDE \eqref{eq:38}. 
For $L \in \bbN$ and $y_0 \in D_\infty$ we search for a $\DBar_\calH$-valued continuous semi-martingale $(Y^L_t)_{t \in [0, T]}$
that is a weak solution of the following perturbed system
\begin{equation}\label{eq:16}
\begin{cases}
\dd Y^L_t(\fraki) = \dd \WTil_t(\fraki) - c \, \dd t + \sum_{H \in \calH: \fraki \in Q_H} c(H, \fraki) \dd L_t^H,\quad \fraki\in Q_L\\
Y^L_0=y_0^L,
\end{cases}
\end{equation}
for a family of independent Brownian motions $(\WTil(\fraki) : \fraki \in Q_L)$. 

We solve \eqref{eq:16} through a Girsanov transform. Introduce the tilted probability measure
$$ \dd \bbP^c \eqdef e^{ -c \sum_{\fraki \in Q_L} W_T(\fraki) - \thalf c^2 T |Q_L| } \dd \bbP. $$
Under $\bbP^c$, the semi-martingales $(\WTil(\fraki) : \fraki \in Q_L)$,
$$ \WTil(\fraki) \eqdef (W_t(\fraki) + c t)_{t \in [0, T]}, \qquad \fraki \in Q_L \,,$$
are distributed as an independent family of Brownian Motions.
In particular, defining, as in \cref{sec:finite-volume-dynam}, $Y^L \eqdef \Gamma(y_0^L + W^L)$, we conclude that $Y^L$ is a weak solution to \eqref{eq:16}. 

In addition, we define a tilted measure $\bfP_\gamma^c$ of $\bfP_\gamma$ as follows. For $R \in \bbN$, $x_0 \in F_\infty$, set
\begin{equation}\label{eq:vol} 
\dd \bfP_\gamma^c \eqdef e^{2 c \, \vol_R \ch^{x_0}} \dd \bfP_\gamma / Z^c \quad \text{ in which } \quad \vol_R \, \ch^{x_0} \eqdef \int \chi_R(x) \, \ch^{x_0}(x) \dd x
\end{equation}
and $\chi_R \in C^\infty_c(F_\infty)$ is any cutoff satisfying $\ic_{F_R} \le \chi_R \le \ic_{F_{R+1}}$.
\begin{proposition}\label{prp:5}
If $R \ge L+1$ and $x_0 \notin F_{L+1}$, then for all $c \ge 0$, it holds that $Y^L$ is time-stationary in law under
$$\bbP_\gamma^c \eqdef \bfP_\gamma^c \otimes \bbP^c .$$
\end{proposition}
\begin{proof}
By the Gibbs property (item~\eqref{itm:bead-gibbs} in \cref{sec:thebeadmodel}), under $\bfP_\gamma$, upon conditioning on the frozen configuration outside $F_L$, the active bead positions are uniformly distributed on $\DBar_\calH$. Since the height is pinned outside of $F_L$, $x_0 \notin F_{L+1}$, it holds that
$$ \vol_R \ch^{x_0}(y) = C - \sum_{\fraki \in Q_L} y(\fraki), $$
where $C$ depends only on the conditioned region and the cutoff.

It is classical that normally reflected Brownian motion with constant drift $v$ has reversible density proportional to $e^{2 v \cdot y}$ on a bounded domain; see \cite[Theorem~2.3(iii)]{BurdzyChenPal11_Archimedes}. Taking $v(\fraki) = -c$ for every $\fraki \in Q_L$ gives \eqref{eq:vol}. The conditioned region remains fixed under the dynamics, so averaging over it proves stationarity (and reversibility).
\end{proof}

\begin{proof}[Proof of~\eqref{eq:static}]
In this proof we reparameterise $R = L+1$, so $x_0 = (2L+2,2L+2) \notin F_{L+1}$, and consider a limit as $L \rightarrow \infty$.
In other words, for the above choice $\kappa =O( (L+1)^{-100})$, we will show that
\begin{equation}\label{eq:15}
\lim_{L\rightarrow\infty} \kappa^{-1} U^{x_0}_{\kappa}[ \chi_{L+1} ] = \frac{1}{2\pi}.
\end{equation}
Recall the definition of $U_t^{x_0}$ in~\eqref{eq:22}, which together with~\eqref{eq:vol} allows us to write 
$$ U_{t}^{x_0}[ \chi_{L+1} ] = \bbE_\gamma[J((\eta_s)_{s \in [0, t]}, 0) \, \vol_{L+1} \ch^{x_0}(\eta_0) ]\,. $$
We wish to replace $\eta$ with the truncated dynamics $\eta^L$. Defining
\begin{align}
U^L_\kappa &\eqdef \bbE_\gamma[J((\eta^L_s)_{s \in [0, t]}, 0) \, \vol_{L+1} \ch^{x_0}(\eta_0) ]
\end{align}
we claim that in the regime of interest, $\kappa =O( (L+1)^{-100})$, it holds
$$ \kappa^{-1} U^{x_0}_{\kappa}[ \chi_{L+1} ] = \kappa^{-1} U^L_{\kappa}[ \chi_{L+1} ] + o(1). $$
Indeed, let $A^\complement_L$ be the event whose probability is estimated in \eqref{eq:101}. 
On $A_L$ it holds that $J((\eta_s)_{s \in [0, t]}, 0) = J((\eta^L_s)_{s \in [0, t]}, 0)$, therefore
\begin{align}
|U_{t}[ \chi_{L+1} ] - U^L_t| &= |\bbE[ \ic_{A_L^\complement} J((\eta_s)_{s \in [0, t]}, 0) \, \vol_{L+1} \ch^{x_0}(\eta_0) ] | \\
&\le \bbP\big(A_L^\complement\big)^{1/2} \, \bbE_\gamma[ J((\eta_s)_{s \in [0, t]}, 0)^4]^{1/4} \, \bbE_\gamma[(\vol_{L+1}\ch^{x_0}(\eta_0))^4]^{1/4}
\end{align}
where we have used Cauchy--Schwarz.

Now we turn our attention to $U^L_t$. By stationarity, the current is mean zero, i.e. 
$$\bbE_\gamma^c [  J((\eta^L_s)_{s \in [0, t]}, 0) ] = 0.$$
Take derivatives in $c$ at both sides of this equation and evaluate them at $c=0$. This gives
$$ -\bbE_\gamma \Big[ \sum_{\fraki \in Q_L} W_t(\fraki) J((\eta^L_s)_{s \in [0, t]}, 0) \Big] + 2 \bbE_\gamma[ J((\eta^L_s)_{s \in [0, t]}, 0) \, \vol_R \ch^{x_0}(\eta^L_0)  ] = 0. $$
The derivative of the normalizing constant contributes zero, since it multiplies the mean current, which vanishes by stationarity. Thus
$$ U^L_t = \thalf \bbE_\gamma\Big[ \sum_{\fraki \in Q_L} W_t(\fraki) J((\eta^L_s)_{s \in [0, t]}, 0) \Big]\,.$$
Therefore, it remains to show
\begin{align}\label{eq:18} 
\kappa^{-1} \bbE_\gamma\Big[ \sum_{\fraki \in Q_L} W_\kappa(\fraki) J_\kappa^L \Big] & \rightarrow \frac1\pi 
\end{align}
where we shortened the notation by writing $J_\kappa^L \eqdef J((\eta^L_s)_{s \in [0, \kappa]}, 0)$. 

By \eqref{eq:21}, we have 
$$ \bbE_\gamma\Big[ \sum_{\fraki \in Q_L} W_\kappa(\fraki) J_\kappa^L \Big] = \bbE_\gamma\Big[ \sum_{\fraki, \frakj \in Q_L} W_\kappa(\fraki) \ic\{ \frakl(\frakj) = 0 \} \big(  \ic\{ Y^L_\kappa(\frakj) \ge 0 \} - \ic\{ Y^L_0(\frakj) \ge 0 \}  \big) \Big] .$$
We wish to replace $J^L_\kappa$ for $\JTil^L_\kappa$, in which
$$ \JTil^L_\kappa \eqdef \sum_{\frakj \in Q_\infty} \ic\{ \frakl(\frakj) = 0 \} \big(  \ic\{ Y_0 + W_\kappa(\frakj) \ge 0 \} - \ic\{ Y_0(\frakj) \ge 0 \}  \big) . $$
Provided such replacement comes at a negligible cost, all that remains is an explicit computation, 
which is done at the end of this proof. Let us then first control the error. 
In short: if there are two particles near the origin, this replacement is invalid. 
But this event is quadratically unlikely, so it can beat the linear blow-up of $\kappa^{-1}$. 
To that end, introduce the box $F_\varepsilon \subset F_L$, 
$$F_\varepsilon := -1{:}1 \times [-2\varepsilon, 2\varepsilon].$$
We set $\varepsilon = \kappa^{1/2-\theta}$ for some $\theta \in (0, \thalf)$. 
Let $B_\varepsilon$ be the event $B_\varepsilon \eqdef \{ \eta_0( F_\varepsilon ) \le 1 \}$ 
and consider the error term, by H\"older, with $p^{-1}+q^{-1}=1$,
\begin{equation}\label{eq:34}
\begin{split}
\Big|\bbE_\gamma\Big[ &\sum_{\fraki \in Q_L} W_\kappa(\fraki) (J_\kappa^L - \JTil_\kappa^L) \Big]\Big| \\& \lesssim_{p, q}  \bbE_\gamma\Big[ \Big(\sum_{\fraki \in Q_L} W_\kappa(\fraki)\Big)^q\Big ]^{1/q} \Big( \bbE_\gamma[ |J_\kappa^L - \JTil_\kappa^L|^p \ic_{B_\varepsilon} ]^{1/p} + 2 \bfE_\gamma[ |Q_L|^p \ic_{B_\varepsilon^\complement}]^{1/p} \Big)
\end{split}
\end{equation}
in which the last term comes from the simple bound $|J| \le |Q_L|$ applied to both $J = J^L_\kappa,\JTil^L_\kappa$. 

To estimate the first factor in \eqref{eq:34}, we use the independence of $Q_L$ and the Brownian Motions,
$$ \bbE_\gamma\Big[ \Big(\sum_{\fraki \in Q_L}W_\kappa(\fraki)\Big)^q\Big] \le \kappa^{q/2} \bfE_\gamma[ |Q_L|^{q/2}] \lesssim \kappa^{q/2} L^{q} $$
where we applied~\eqref{eq:MomQL}. 
We turn to the second factor in \eqref{eq:34}, and consider the first summand therein. Recall $A(\varepsilon, \kappa, \fraki)$ from \eqref{eq:75}, and define $\ATil(\varepsilon, \kappa, \fraki)$ for the same event, but with $Y_0 + W^L$ replacing $Y^L$. We claim that
\begin{equation}\label{eq:37}
\{ J_\kappa^L \ne \JTil_\kappa^L \} \cap B_\varepsilon \subseteq \bigcup_{\fraki \in Q_L} (A(\varepsilon, \kappa, \fraki) \cup \ATil(\varepsilon, \kappa, \fraki) ).
\end{equation}
In words, the r.h.s.\ of \eqref{eq:37} requires some particle or Brownian motion to move a distance of at least $\varepsilon$. If $\eta_0( \{0\} \times [-\varepsilon, \varepsilon] ) = 0$, this is immediate, because otherwise no particle can cross the origin, and $J_\kappa^L = \JTil^L_\kappa = 0$. Otherwise, if $\eta_0( \{0\} \times [-\varepsilon, \varepsilon] ) \ne 0$ then on $B_\varepsilon$, $\eta_0( \{0\} \times [-\varepsilon, \varepsilon] ) = 1$. In this case $J_\kappa^L \ne \JTil_\kappa^L$ can only occur if either the particle $\fraki$ in the box $\{0\} \times [-\varepsilon, \varepsilon]$ has a collision, causing the paths $(Y_r(\fraki))_{r \in [0, \kappa]})$ and $(Y_0(\fraki) + W_r(\fraki))_{r \in [0, \kappa]}$ to decouple, or if a particle from outside the box $\{0\} \times [-2\varepsilon, 2\varepsilon]$ crosses the origin. Either case requires movement of at least $\varepsilon$.

Using \eqref{eq:37}, the trivial bound $|J| \le |Q_L|$ and~\eqref{eq:MomQL}, we have
\begin{align}
\bbE_\gamma[ |J_\kappa^L - \JTil_\kappa^L|^p \ic_{B_\varepsilon} ] &\le 2 \bbE_\gamma[ |Q_L|^p \sum_{\fraki \in Q_L} (\ic_{A(\varepsilon, \kappa, \fraki)} + \ic_{\ATil(\varepsilon, \kappa, \fraki)} ) ] \\
&\le 2 \bbE_\gamma[ |Q_L|^{p+1} \ic_{\{Q_L \not\subset P_L\}} ] + L^{2p} \sum_{\fraki \in P_L} (\bbP_\gamma(A(\varepsilon, \kappa, \fraki)) + \bbP(\ATil(\varepsilon, \kappa, \fraki)) ) \\
&\lesssim L^{2p+2} e^{-C L^2} + L^{2p+4} e^{-\varepsilon^2/(2 \kappa)} \lesssim L^{2p+2} e^{-C L^2} + L^{2p+4} e^{- C \kappa^{-2\theta}}\,.
\end{align}
We turn to the final term in \eqref{eq:34}. We claim 
$$ \bfE_\gamma[ |Q_L|^p \ic_{B_\varepsilon^\complement}] \lesssim L^{2p} \varepsilon^2 .$$
Indeed, by the definition of the two-point correlation function and \eqref{boundrhon},
$$ \bfP_\gamma\big(B_\varepsilon^\complement\big) \le \thalf \bfE_\gamma[\eta_0(F_\varepsilon)(\eta_0(F_\varepsilon)-1)] = \thalf \int_{F_\varepsilon^2}\rho_2(x_1,x_2)\dd x_1\dd x_2 \lesssim \varepsilon^2. $$
On the event $\{Q_L\subset P_L\}$, we have $|Q_L|\le 9L^2$, while by Cauchy--Schwarz, \eqref{eq:90} and \eqref{eq:MomQL},
$$ \bfE_\gamma[|Q_L|^p\ic_{\{Q_L\not\subset P_L\}}] \le \bfE_\gamma[|Q_L|^{2p}]^{1/2}\bfP_\gamma(Q_L\not\subset P_L)^{1/2}\lesssim L^{2p}e^{-CL^2}. $$
The claim follows since $e^{-CL^2}\le \varepsilon^2$ for all sufficiently large $L$, recalling that $\varepsilon=\kappa^{1/2-\theta}$ and $\kappa=O((L+1)^{-100})$.

Plugging all this into \eqref{eq:34}, we obtain, for some new constants $C$, 
\begin{equation}
\Big|\bbE_\gamma\Big[ \sum_{\fraki \in Q_L} W_\kappa(\fraki) (J_\kappa^L - \JTil_\kappa^L) \Big]|\le  L^{1/2 + 2 + 4/p} \kappa^{1/2} \Big( e^{-C L^2} + e^{- C \kappa^{-2\theta}} + \kappa^{(1 - 2\theta)/p} \Big)
\end{equation}
Choosing, e.g. $(p, q) = (5/4, 5)$ and $\theta = 1/8$, gives 
$$\Big |\bbE_\gamma\Big[ \sum_{\fraki \in Q_L} W_\kappa(\fraki) (J_\kappa^L - \JTil_\kappa^L) \Big]\Big| \lesssim L^{10} \kappa^{11/10}, $$
which is enough to beat the $\kappa^{-1}$ of \eqref{eq:18}.

Therefore, it remains to show \eqref{eq:18} but with $\JTil^L_\kappa$ in place of $J^L_\kappa$. 
We have, by independence of the different Brownian paths, 
\begin{align}
\bbE_\gamma\Big[ \sum_{\fraki \in Q_L} W_\kappa(\fraki) \JTil_\kappa^L \Big] = \bbE_\gamma\Big[ \sum_{\fraki \in Q_L} \ic\{\frakl(\fraki) = 0\} W_\kappa(\fraki)  \big( \ic\{y_0(\fraki) + W_\kappa(\fraki) \ge 0 \} - \ic\{y_0(\fraki) \ge 0 \} \big)\Big ] .
\end{align}
Taking the expectation only in $\bbP$, for a single Brownian Motion integration by parts gives 
$$ \bbE[ W_\kappa(\fraki) \big( \ic\{y_0(\fraki) + W_\kappa(\fraki) \ge 0 \} - \ic\{y_0(\fraki) \ge 0 \} \big) ] = \kappa p_\kappa(y_0(\fraki)) $$
in which $p_t(y)$ is the transition kernel of a Brownian Motion. Therefore, 
$$ \bbE_\gamma\Big[ \sum_{\fraki \in Q_L} W_\kappa(\fraki) \JTil_\kappa^L \Big] = \kappa \bfE_\gamma\Big[ \sum_{\fraki \in Q_L} \ic\{\frakl(\fraki) = 0\} p_\kappa(y_0(\fraki)) \Big] = \kappa \bfE_\gamma[ \omega( p_\kappa\restr_{-L, L} ) ] $$
Since $\bfE_\gamma[ \omega( p_\kappa\restr_{-L, L} ) ] \rightarrow \pi^{-1}$, we are done. 
\end{proof}

\subsection{The dynamic term}\label{sec:dynamic-term}
The goal of this section is to prove~\eqref{eq:dynamic}. 

\begin{proof}[Proof of~\eqref{eq:dynamic}] Take $R = L/2$. Recall the definition of $ V_{t, \kappa}[\chi_{L/2}]$ in~\eqref{eq:44} and write
\begin{equation}\label{eq:52}
V_{t, \kappa}[\chi_{L/2}] = \sum_{s \in \kappa \bbN \cap [0, t)} V^0_s, \qquad V^0_s \eqdef \int \chi_{L/2}(x) \bbE_\gamma[J( (\eta_{r})_{r \in [0, \kappa]}, 0) J((\eta_r)_{r \in [\kappa, s+\kappa]}, x) ] \dd x
\end{equation}
where we used the symmetry of $\chi_{L/2}$ to replace $-x$ by $x$ in \eqref{eq:44}.

We proceed by making a series of replacements of $V$ with $V^i_s$, 
where for $i \in \{0,1,2\}$ it will hold, for some $C = C(T, \gamma)>0$,
\begin{equation}\label{eq:45}
\sup_{s \in \kappa \bbN \cap [0, t)} |V^i_s - V_s^{i+1}| \lesssim_{T, \gamma}  e^{-C (\log L)^2} + L^{-200}
\end{equation}
and finally $V^3_s$ will be controlled directly by \cref{lem:4}. This is sufficient to prove the theorem, since the decay of the r.h.s.\ of \eqref{eq:45} beats the $\kappa^{-1} =O( L^{100})$ generated by the sum in \eqref{eq:52}. 
Only the case $i = 2$ will be controlled by $L^{-200}$,  the rest enjoying the better decay of $e^{-C (\log L)^2}$.

We proceed by defining all of the replacements, $V^i_s$ for $i \in 1{:}3$, leaving the proof of \eqref{eq:45} to the end. 
First, we replace the process $\eta$ with the truncated process $\eta^L$, defining
$$ V^1_s \eqdef \int \chi_{L/2}(x) \bbE_\gamma[J( (\eta^L_r)_{r \in [0, \kappa]}, 0) J((\eta^L_r)_{r \in [\kappa, s+\kappa]}, x) ] \dd x. $$
At this point, we perform the integral in space, writing
\begin{align}
\int &\chi_{L/2}(x) J((\eta^L_r)_{r \in [\kappa, s+\kappa]}, x) \dd x \\
&= \sum_{\fraki \in Q_L} \int \chi_{L/2}(x) \ic_{\{ \frakl(\fraki) = \frakl(x) \}} \big(  \ic_{\{ Y_{s + \kappa}(\fraki) \ge y(x) \}} - \ic_{\{ Y_\kappa(\fraki) \ge y(x) \}} \big) \dd x=:\sum_{\fraki \in Q_L} J_\fraki\,.
\end{align}
An explicit computation using only that $\ic_{F_R} \le \chi_R \le \ic_{F_{R+1}}$ shows that, 
for $\fraki \in Q_\infty$, we have 
\begin{equation}\label{eq:6}
J_\fraki = \begin{cases}
0, \quad &|\frakl(i)| > \thalf L \\
0, \quad &|\frakl(i)| \le \thalf L, \, \{Y_{s + \kappa}(\fraki), Y_\kappa(\fraki)\} \cap [-\thalf L - 1, \thalf L + 1] = \emptyset \\
Y_{s + \kappa}(\fraki) - Y_\kappa(\fraki), \quad &|\frakl(i)| \le \thalf L, \{Y_{s + \kappa}(\fraki), Y_\kappa(\fraki)\} \subset [-\thalf L, \thalf L] \\
\int_{Y_\kappa(\fraki)}^{Y_{s + \kappa}(\fraki)} \chi_{L/2}(\frakl(\fraki), y) \dd y, \quad &\text{otherwise. }
\end{cases}
\end{equation}
Observing \eqref{eq:6}, we expect that for most $\fraki \in Q_{L/2}$ it holds that $J_\fraki = Y_{s + \kappa}(\fraki) - Y_\kappa(\fraki)$, 
whereas for most $\fraki \notin Q_{L/2}$, it holds $J_\fraki = 0$, the error coming from boundary contributions. 
To that end, we define $Q_{L/2}' \subset Q_{L/2}$ according to 
$$ Q_{L/2}' \eqdef \{\fraki \in Q_{L/2} : J_\fraki \ne Y_{s+\kappa}(\fraki) - Y_\kappa(\fraki) \} $$
and write
\begin{align}\label{eq:36}
\sum_{\fraki \in Q_L} J_\fraki = \sum_{\fraki \in Q_{L/2}} (Y_{s+\kappa}(\fraki) - Y_{\kappa}(\fraki)) + \sum_{\fraki \notin Q_{L/2}} J_\fraki + \sum_{\fraki \in Q_{L/2}'} J_\fraki - \sum_{\fraki \in Q_{L/2}'} (Y_{s+\kappa}(\fraki) - Y_{\kappa}(\fraki))\,.
\end{align}
Therefore, abbreviating $J_0 \eqdef J( (\eta^L_r)_{r \in [0, \kappa]}, 0)$, we have
\begin{align}
V^1_s &= \bbE_\gamma\Big[J_0 \sum_{\fraki \in Q_L} J_\fraki \Big] \\
&= \bbE_\gamma\Big[J_0 \sum_{\fraki \in Q_{L/2}} (Y_{s+\kappa}(\fraki) - Y_{\kappa}(\fraki)) \Big] + \bbE_\gamma\Big[ J_0 \Big( \sum_{\fraki \notin Q_{L/2}} J_\fraki + \sum_{\fraki \in Q_{L/2}'} J_\fraki - \sum_{\fraki \in Q_{L/2}'} (Y_{s+\kappa}(\fraki) - Y_{\kappa}(\fraki))\Big) \Big]. 
\end{align}
At this point, we write the SDE \eqref{eq:38} as $\dd Y(\fraki) = \dd W(\fraki) + \dd L(\fraki) $, in which 
$$ L(\fraki) \eqdef \sum_{H \in \calH: \fraki \in Q_H} c(H, \fraki) \dd L_t^H $$
which allows us to conclude
$$ \bbE_\gamma\Big[J_0 \sum_{\fraki \in Q_{L/2}} (Y_{s+\kappa}(\fraki) - Y_{\kappa}(\fraki)) \Big] = \bbE_\gamma\Big[J_0 \sum_{\fraki \in Q_{L/2}} (L_{s+\kappa}(\fraki) - L_{\kappa}(\fraki)) \Big] $$
because the martingale term in the diffusion  does not contribute. Therefore, upon defining the functional $K$ according to 
\begin{equation}\label{eq:23}
K( (\eta^L_r)_{r \in [\kappa, s+\kappa]} ) \eqdef \sum_{\fraki \in Q_{L/2}} (L_{s+\kappa}(\fraki) - L_{\kappa}(\fraki)) + \sum_{\fraki \notin Q_{L/2}} J_\fraki + \sum_{\fraki \in Q_{L/2}'} J_\fraki - \sum_{\fraki \in Q_{L/2}'} (Y_{s+\kappa}(\fraki) - Y_{\kappa}(\fraki)),
\end{equation}
we have 
\begin{equation}\label{eq:42}
V^1_s = \bbE_\gamma\Big[J_0 ((\eta^L_r)_{r \in [0,\kappa]} ) K((\eta^L_r)_{r \in [\kappa, s+\kappa]} ) \Big]\,.
\end{equation}

At this point we introduce the event
$$ E_1 \eqdef \{ \Delta_L(\eta_0)>\log L \} \cap \{ Q_L \subset P_L \} \cap \{ \forall \fraki \in Q_L : \sup_{t \in [0, T]} |Y_t(\fraki) - Y_0(\fraki)| \le \log L \}$$
and define
$$ V^2_s \eqdef \bbE_\gamma\Big[\ic_{E_1} J_0 ((\eta^L_r)_{r \in [0,\kappa]} ) K((\eta^L_r)_{r \in [\kappa, s+\kappa]} ) \Big]\,. $$
We claim that there is a function $K^\sharp$ on $C([0, s), \OmegaBar(F_L \setminus F_{L/3}))$ such that
\begin{equation}\label{eq:33}
\ic_{E_1} K((\eta^L_r)_{r \in [\kappa, s+\kappa]} ) = K^\sharp( (\eta^L_r\restr_{F_L \setminus F_{L/3}})_{r \in [\kappa, s+\kappa]} ).
\end{equation}
In other words, after restricting to $E_1$, the function $K$ depends only on what happens outside of $F_{L/3}$. We argue that such a function exists independently for each of the four terms on the r.h.s.\ of \eqref{eq:23}. 

The most interesting case being the first term. Upon aggregating the local times $L_t^H$ coming from each $L(\fraki)$, we see that in the bulk, most of the coefficients sum to zero. We are left only with a remainder
$$ \sum_{\fraki \in Q_{L/2}} L(\fraki) = \sum_{H \in \calH'} c'(H) L^H $$
for some unimportant constants $|c'(H)| \le 1$, and for the set of constraints
$$ \calH' := \{ H_\frakp : \frakp \in P_L, \, |\frakp \cap Q_{L/2}| = 1 \} \cup \{ H^+_\fraki : \fraki \in Q_L \} \cup \{ H^-_\fraki : \fraki \in Q_L \}\,. $$
We argue that under $E_1$, the local time increments $L^H_{s+\kappa} - L^H_{\kappa}$ for $H \in \calH'$ 
only depend on $(\eta^L_r\restr_{F_L \setminus F_{L/3}})_{r \in [\kappa, s+\kappa]}$. 
If $H = H^\pm_\fraki$ for some $\fraki \in Q_L$, then this is immediate. 
Otherwise, for $H = H_\frakp$, for there to be any dependence on the region $F_{L/3}$, 
a particle would need to move from within $\log L$ of the boundary $\partial F_{L/2}$ to inside of $F_{L/3}$, 
which is a distance of order $L$. This is impossible on $E_1$. 
To see that the remaining terms of \eqref{eq:23} satisfy the r.h.s.\ of \eqref{eq:33} is simpler, 
and follows immediately from the \eqref{eq:6} and the fact that on $E_1$ particles move no further than $\log L$. 
In conclusion, we reach
\begin{equation}\label{eq:27}
V^2_s = \bbE_\gamma\Big[J_0 ((\eta^L_r)_{r \in [0,\kappa]} ) K^\sharp( (\eta^L_r\restr_{F_L \setminus F_{L/3}})_{r \in [\kappa, s+\kappa]} ) \Big]\,.
\end{equation}

Next, we truncate $K^\sharp$, meaning, for $N\in\bbN$ (to be chosen later), we set 
\begin{align}
K^N &\eqdef K^\sharp \ic_{ |K^\sharp| \le L^N }, \\
V_s^3 &\eqdef \bbE_\gamma\Big[J_0((\eta^L_r)_{r \in [0,\kappa]} ) K^N( (\eta^L_r\restr_{F_L \setminus F_{L/3}})_{r \in [\kappa, s+\kappa]} ) \Big]\,. \label{eq:17}
\end{align}
Since $|K^N| \le L^N$, \cref{lem:4} gives, after decreasing $C$ if necessary,
$$ \sup_{s \in \kappa \bbN \cap [0,t)} |V^3_s| \lesssim e^{-C \log(L)^2} (\bbE_\gamma[|K^N|^4])^{1/4} \le L^N e^{-C \log(L)^2} \lesssim e^{-C \log(L)^2}. $$
Thanks to the above, all that remains is to prove \eqref{eq:45} in the case of $i = 0, 1, 2$.
\medskip

\noindent\textit{Proof of case $i = 0$}
Consider the event
$$ E = \{( \eta_t\restr_{F_{L/2}} )_{t \in [0, T]} \ne ( \eta^{L}_t\restr_{F_{L/2}} )_{t \in [0, T]} \}$$
then by \cref{cor:1} we have
\begin{equation}\label{eq:49}
\bbP_\gamma( E^c ) \lesssim e^{ -C (\log L)^2 }.
\end{equation}
Using Cauchy--Schwarz, along with the fact that $J$ has all moments, we have
\begin{align}
\bbE_\gamma[J( (\eta_r)_{r \in [0, \kappa]}, 0) J((\eta_r)_{r \in [\kappa, s+\kappa]}, x) ] &= \bbE_\gamma[\ic_E J( (\eta_r)_{r \in [0, \kappa]}, 0) J((\eta_r)_{r \in [\kappa, s+\kappa]}, x) ] + O(e^{ -C (\log L)^2 }) \\
&= \bbE_\gamma[\ic_E J( (\eta^L_r)_{r \in [0, \kappa]}, 0) J((\eta^L_r)_{r \in [\kappa, s+\kappa]}, x) ] + O(e^{ -C (\log L)^2 }) \\
&= \bbE_\gamma[J( (\eta^L_r)_{r \in [0, \kappa]}, 0) J((\eta^L_r)_{r \in [\kappa, s+\kappa]}, x) ] + O(e^{ -C (\log L)^2 }).
\end{align}

\noindent\textit{Proof of case $i = 1$}
The proof is the same as the case $i = 0$, since $\bbP_\gamma(E_1^c)$ also satisfies \eqref{eq:49}. All that is required is to see that $K$ has all moments. We claim that for $p \in \bbN$,
$$ \bbE_\gamma |K|^p \lesssim_p L^{p+2}. $$ 
This is trivial for the last three terms of \eqref{eq:23}, since the terms $J_\fraki$ and $Y(\fraki)$ are bounded by $L$ in absolute value. For the first, we rewrite
$$ \sum_{\fraki \in Q_{L/2}} (L_{s+\kappa}(\fraki) - L_{\kappa}(\fraki)) = \sum_{\fraki \in Q_{L/2}} (Y_{s+\kappa}(\fraki) - Y_{\kappa}(\fraki)) - \sum_{\fraki \in Q_{L/2}} (W_{s+\kappa}(\fraki) - W_{\kappa}(\fraki)) $$
then the result follows from the fact that the Brownian motion has all moments, and again that $|Y|$ is bounded by $L$. 
\medskip

\noindent\textit{Proof of case $i = 2$}
This is the term that has the worst decay. We have by Cauchy--Schwarz
\begin{align}
|V^2_s-V^3_s| &= | \bbE_\gamma\Big[\ic_{|K^\sharp| > L^N } J_0 ((\eta^L_r)_{r \in [0,\kappa]} ) K^\sharp ( (\eta^L_r\restr_{F_L \setminus F_{L/3}})_{r \in [\kappa, s+\kappa]} ) \Big] | \\
&\lesssim \bbE_\gamma[|J_0|^4]^{1/4} \bbE_\gamma[|K^\sharp|^4]^{1/4} \bbP_\gamma( |K^\sharp| > L^N )^{1/2} \\ 
&\le \bbE_\gamma[|J_0|^4]^{1/4} \bbE_\gamma[|K^\sharp|^4]^{1/4} (\bbE_\gamma[|K^\sharp|] / L^N)^{1/2} \le L^{3 - N/2}.
\end{align}
So \eqref{eq:45} holds upon choosing, say, $N = 1000$. 
\end{proof}

\appendix

\section{Reflecting SDE}

This appendix is devoted to the rigorous study of the finite-volume dynamics~\eqref{eq:38}. 
We begin by recalling the Skorokhod reflection map, its properties and we then turn to the proof of Proposition~\ref{lem:triple}.  

\subsection{The Skorokhod reflection map}\label{sec:skorokhod-solution}

Let $\mathcal D \subset \bbR^d$ be a convex domain and for $y \in \partial \mathcal D$, 
write $\bfN(y) \subset \bbR^d$ for the set of inward-pointing unit normal vectors.  

\begin{definition}\label{def:2}
For a  path $w \in C([0, \infty), \bbR^{d})$  such that $w(0) \in \mathcal D$, 
we say that $y \in C([0, \infty), \mathcal D)$ is a solution of the Skorokhod problem under normal reflection if the following properties hold 
\begin{itemize}
\item the regulator $(a_t)_{t \ge 0}\eqdef (y_t-w_t)_{t\geq 0}$  is of finite variation with $a_0 = 0$;
\item upon defining $\ell_t \eqdef \int_0^t |\dd a_s|$, we have 
  \begin{equation}\label{eq:77}
  \int_0^\infty \ic\{y_t \in \mathring{\mathcal D}\} \dd \ell_t= 0
  \end{equation}
  where $\mathring{\mathcal D}$ denotes the interior of $\mathcal D$; 
\item upon defining $(n_t)_{t \ge 0}$ to be the Radon--Nikodym derivative of $a$ with respect to $\ell$, that is, $a_t = \int_0^t n_s \dd \ell_s$, it holds that $n$ is $\ell$-a.s.\ a unit normal at $y$, i.e.
  \begin{equation}\label{eq:78}
  \int_0^\infty \ic\{n_t \notin \bfN(y_t)\} \dd \ell_t = 0.
  \end{equation}
\end{itemize}
\end{definition}

In \cite{Tanaka79_StochasticDifferential}, it was proved that the (normal) Skorokhod problem 
has a unique solution for convex domains. Let us state the result we will need. 

\begin{theorem}\label{thm:tanaka} Let $\mathcal D$ be a convex domain of $\mathbb R^d$, 
and $w$ a path in $ C([0, \infty), \bbR^{d})$ satisfying $w(0) \in \mathcal D$.
Then, there is a continuous map $\Gamma : C([0, \infty), \bbR^{d}) \rightarrow C([0, \infty), \mathcal D)$ such that 
$y \eqdef\Gamma(w)$ is the unique solution to the Skorokhod problem on $\mathcal D$, as in \cref{def:2}.
\end{theorem}

\subsection{Proof of Proposition~\ref{lem:triple}}\label{sec:refl-sde-strong}

To prove well-posedness of the SDE in \eqref{eq:38} and complete 
the proof of Proposition~\ref{lem:triple}, we will use the deterministic map $\Gamma$ in \cref{thm:tanaka} 
with $\mathcal D = \DBar_\calH$ in~\eqref{eq:8}  and $\bbR^d = \bbR^{Q}$, i.e. 
we set $Y\eqdef \Gamma(y_0+W)$ for $y_0\in \DBar_\calH$ and $W=(W(\fraki)\colon \fraki\in Q)$ a family 
of independent Brownian motions started at $0$. 
We will begin by proving that $Y$ satisfies \eqref{eq:46} and \eqref{eq:47}, and finally show that it solves \eqref{eq:38}.

\begin{proof}[Proof of \eqref{eq:46} and~\eqref{eq:47}]
Let us begin with \eqref{eq:46}. Fix $T \ge 0$. For each $H \in \calH$, define
$$ Z_t^H := Y_t \cdot \bfn_H - a_H. $$
Now, $Z^H_t$ is a one-dimensional semi-martingale with $\langle Z^H \rangle_t = t$. 
Therefore its occupation measure is absolutely continuous with respect to the Lebesgue measure and 
thus $\int_0^t \ic\{ Z^H_s = 0 \} \dd s = 0$, and~\eqref{eq:46} follows at once. 

As for \eqref{eq:47}, this follows from~\cite[Th. 4]{Lepingle10_BoundaryBehavior}.
\end{proof}

\begin{proof}[Proof that $Y$ solves \eqref{eq:38}]
Write $Y = y_0 + W + A$ for the corresponding semi-martingale decomposition, and set $L := \int |\dd A|$.

Recall $\dd A_s = N_s \dd L_s$ in which  $N_s \in \bfN(Y_s)$ almost surely with respect to $L$. 
Due to the absence of multiple points (see \eqref{eq:47}), 
it holds that if $s \ge 0$ is such that $Y_s \in \partial H$, 
then $N_s = \bfn_H$. In particular, upon defining $\Lambda^H_t := \int_0^t  \ic\{ Y_s \in \partial H \} \dd L_s$, we have
$$ \int_0^t  \ic\{ Y_s \in \partial H \} \dd  A_s = \bfn_H \Lambda^H_t $$
Furthermore, invoking \eqref{eq:47} again, 
we have the partition $A_t = \sum_{H \in \calH} \int_0^t \ic\{ Y_s \in \partial H \} \dd A_s=\sum_{H \in \calH} \bfn_H \Lambda^H_t$.

Plugging this into $Y = y_0 + W + A$, and using the explicit expression, in which $(\bfe(\fraki) : \fraki \in Q)$ are the standard basis vectors on $\bbR^Q$, and that 
\begin{equation}
\bfn_H =
\begin{cases}
(\bfe(\frakj) - \bfe(\fraki))/\sqrt{2}, \quad &\text{if $H = H_{(\fraki, \frakj)}$ for some $(\fraki, \frakj) \in P$}, \\
-\bfe(\fraki), &\text{if $H = H^+_\fraki$ for some $\fraki \in Q$},  \\
\bfe(\fraki), &\text{if $H = H^-_\fraki$ for some $\fraki \in Q$},  
\end{cases}
\end{equation}
we obtain an equation analogous to~\eqref{eq:38} but 
in which each $L^H$ has been replaced by $\Lambda^H$.
Therefore, we are left to show that $L^H = \Lambda^H$ for each $H \in \calH$.

We only treat the case $H = H_\frakp$ for some $\frakp =(\fraki, \frakj)\in P$, the others being analogous. 
Subtracting the equation for $Y(\fraki)$ from the equation for $Y(\frakj)$ and dividing by $\sqrt{2}$, we get, for $Z := Y\cdot\bfn_H-a_H$,
$$ Z = \frac{y_0(\frakj) - y_0(\fraki) + W(\frakj) - W(\fraki)}{\sqrt{2}} + \Lambda^H + \sum_{J \in \calJ} c_{J} \Lambda^J $$
in which $\calJ \subset \calH$ is a finite set of elements, $H \notin \calJ$, and $c_{J}$ are constants whose 
value is not important for us. This implies that 
\begin{align}
L^H =\frac{1}{\sqrt{2}}\int \ic\{Z = 0\} \dd (W(\frakj) - W(\fraki))+ \int \ic\{Z = 0\} \dd\Lambda^H+ \sum_{{J} \in \calJ} c_{J} \int \ic\{Z = 0\} \dd \Lambda^J
\end{align}
By computing the quadratic variation, and invoking \eqref{eq:46}, the first summand at the r.h.s. is zero. 
Also the last is $0$ since, due to the absence of multiple points, 
it holds that $Z \ne 0$ almost surely with respect to $\Lambda^J$ whenever $J \ne H$. 
On the other hand, $Z = 0$ almost surely with respect to $\Lambda^H$, and therefore the remaining term
equals $\Lambda^H$, as required.

At last, uniqueness of \eqref{eq:38} follows from uniqueness of the Skorokhod reflection map.
\end{proof}

\subsection{Proof of \cref{thm:monotone}}\label{a:monotonicity}

\begin{proof}[Proof of \cref{thm:monotone}]
At first, notice that the continuity~\cite[Thm. 2.1]{Tanaka79_StochasticDifferential} and 
memoryless property~\cite[Lem. 4.1]{Sarantsev19_ComparisonTechniques} of the 
Skorokhod reflection map ensure that it suffices to prove the result 
for driver $w \in C([0, T], \bbR^Q)$ of the form 
$w_t = w_0 + v t e_\fraki$ for some fixed canonical basis vector $e_\fraki \in \bbR^Q$ and some velocity $v \in \bbR$.
Indeed, by~\cite[Lem. 4.1]{Sarantsev19_ComparisonTechniques}, the set of 
piecewise linear linear functions whose linear pieces are 
parallel to the coordinate axis are dense in $C([0, T], \bbR^Q)$, and the memoryless property 
says that the evolution of the Skorokhod reflection 
map after a given time $t>0$ only depends on its configuration at $t$ and the future evolution of the driver.  

Thus, fix $w$ as above and note that the solution of the Skorokhod reflection map for such $w$ is explicit: 
for any $\fraki\in Q$, $Y(\fraki)$ moves with velocity $v$; when it collides with other particles it forms a coalescent 
group $Q_t\subset Q$ (which is increasing in $t$ over subsets of $Q$) that move with velocity $v/|Q_t|$; if 
at any time $t_0 \in [0, T]$, $Y_t(\frakj) = M_\pm(\frakj)$ for some $\frakj \in Q_t$, then $Y_t(\frakj) = Y_{t_0}(\frakj)$ 
for all $t \in [t_0, T]$ and all $\frakj\in Q_{t}$. The evolution of $Y'$ follows the same description 
but if an active bead $\fraki \in Q$ does not belong $\fraki \notin Q'$, then $Y'_t(\fraki)= Y'_0(\fraki)$ for all $t \in [0, T]$. 
We write $Q'_t$ for the coalescent group of $Y'$.

If $v = 0$, the result is trivial, because there is no evolution 
and the ordering $Y'_t \ge \Pi_{Q'} Y_t$ holds at time $t = 0$ by assumption. 
For $v \ne 0$, we argue by contradiction. Let $t_0 \in [0, T]$ be the minimal time violating the constraint 
\begin{equation}\label{eq:68}
t_0 \eqdef \inf\{ t \in [0, T] : \exists \frakj \in Q', \, Y_t(\frakj) > Y'_t(\frakj) \}
\end{equation}
and let $v_0$ (resp.\ $v'_0$) be the velocity at $t = t_0$ of $Q_t$ (resp.\ $Q_t'$). 
Suppose $v > 0$, then 
\begin{align}
v_0 &= v |Q_{t_0}|^{-1} \ic\{ Y_{t_0}(\frakj) < M_+(\frakj), \, \forall \frakj \in Q_{t_0} \}\\
v'_0 &= v |Q'_{t_0}|^{-1} \ic\{ Y_{t_0}'(\frakj) < M'_+(\frakj), \, \forall \frakj \in Q'_{t_0} \}
\end{align}
For any particle $\frakj \in Q'$ such that $Y_{t_0+}(\frakj) > Y'_{t_0+}(\frakj)$, by continuity, we have 
$Y'_{t_0}(\frakj) = Y_{t_0}(\frakj)$
and thus $v_0 > v'_0$. 

First, we argue that $\frakj$ belongs to the coalescent group for both $Y$ and $Y'$. 
Indeed, $Y_{t_0+}(\frakj) > Y'_{t_0+}(\frakj)$ along with $Y'_{t_0}(\frakj) = Y_{t_0}(\frakj)$ implies that $Y(\frakj)$ 
has positive velocity, which requires $\frakj \in Q_{t_0}$. 
In addition, $\frakj \ge \fraki$, therefore $\fraki \in Q'$ by \eqref{eq:39}. 
Since we also have $Y'_{t_0} \ge \Pi_{Q'} Y_{t_0}$, 
$$ Y_{t_0}(\frakj) = Y_{t_0}(\fraki) \le Y'_{t_0}(\fraki) \le Y'_{t_0}(\frakj). $$
Therefore by $Y'_{t_0}(\frakj) = Y_{t_0}(\frakj)$, all of the above holds with equality, which implies $\frakj \in Q_{t_0}'$. 

In other words, $\fraki, \frakj$ both belong to the coalescent group, and so for $v_0 > v'_0$ 
we must have that either $v'_0 = 0$ or $|Q'_{t_0}| > |Q_{t_0}|$. 
The former is impossible by \eqref{eq:70}, and the latter is impossible upon proving that $Q'_{t_0} \subset Q_{t_0}$. 
To see the inclusion, if $\frakj' \in Q'_{t_0}$, then 
$$ Y_{t_0}(\frakj') \le Y'_{t_0}(\frakj') = Y'_{t_0}(\fraki) = Y_{t_0}(\fraki) \le Y_{t_0}(\frakj') $$
Again, all the above holds with equality and so $\frakj' \in Q_{t_0}$. This concludes the case $v > 0$. 

The case $v<0$ can be argued similarly with the inequalities reversed and $M_-, M_-'$ in place of $M_+, M_+'$ 
in the expression for $v_0, v_0'$. 
\end{proof}

\section{Computing the compressibility}
\label{computeXi}
Here, we prove \eqref{eq:compressibility} and \eqref{eq:U}.
We recall the definition \eqref{eq:Stilde} and we drop the index $t$ since we look only at $t=0$.
We recall \cite[Equation (1.25)]{Johnston23_ContinuousKasteleyn}. Let $\calL$ be an anticlockwise contour from $\gamma + \iota \sqrt{1 - \gamma^2}$ to $\gamma - \iota \sqrt{1 - \gamma^2}$ and staying left of the origin. Let $\calR$ be an anticlockwise contour from $\gamma - \iota \sqrt{1 - \gamma^2}$ to $\gamma + \iota \sqrt{1 - \gamma^2}$ and staying right of the origin. Define, for $\alpha \in \{b, o, u\}$ the kernel\footnote{We obtain \eqref{eq:71} from \cite[Equation (1.25)]{Johnston23_ContinuousKasteleyn} through rescaling the bead density from $\sqrt{1 - \gamma^2} \pi^{-1}$ to $\pi^{-1}$ and flipping the orientation of the discrete axis.}
\begin{equation}
\label{eq:71}
H_{\alpha}(\frakl, y) = 
\begin{cases}
(-1)^{\ic_{\alpha=o}} (1 - \gamma^2)^{-\ic_{\alpha = b}/2} \int_{\calL} z^{\ic_{\alpha=b}+\frakl} e^{-zy/\sqrt{1 - \gamma^2}} \,\frac{\dd z}{2\pi \iota z}, \quad y > 0 \text{ or } y = 0, \alpha \ne o\\
- (-1)^{\ic_{\alpha=o}} (1 - \gamma^2)^{-\ic_{\alpha = b}/2} \int_{\calR} z^{\ic_{\alpha=b}+\frakl} e^{-zy/\sqrt{1 - \gamma^2}} \,\frac{\dd z}{2\pi \iota z}, \quad y < 0 \text{ or } y = 0, \alpha = o
\end{cases}.
\end{equation}

By \cite[Theorem 1.10]{Johnston23_ContinuousKasteleyn}, we have
\begin{align}\label{eq:91}
V_{11}(\frakl, y) &= - H_o(\frakl, y) H_o(-\frakl, -y) \\
V_{12}(\frakl, y) &= - H_b(\frakl, y) H_o(-\frakl, -y) \\
V_{21}(\frakl, y) &= - H_o(\frakl, y) H_b(-\frakl, -y) \\
V_{22}(\frakl, y) &= \frac{1}{\pi} \delta_{(0, 0)} - H_b(\frakl, y) H_b(-\frakl, -y).
\end{align}

Define, for an $L^2(\dd \frak l\otimes \dd y)$  function $f : \bbZ \times \bbR \rightarrow \bbC$, the Fourier transform
$$  \hat f(p) :=  \sum_{\frakl \in \bbZ} \int_{\bbR} e^{\iota (p_1 \frakl + p_2 y)} f(\frakl, y) \dd y, \quad p=(p_1,p_2). $$
An explicit computation gives:
\begin{proposition}
  \label{prop:FT}
Taking Fourier transforms, we compute
\begin{align}
\VHat_{11}(p) &= \frac{1 - \gamma^2}{\pi} \frac{p_1^2}{p_1^2 + 2 p_1p_2 \gamma \sqrt{1-\gamma^2} + p_2^2 (1-\gamma^2)} + o(1)\\
\VHat_{12}(p) , \VHat_{21}(p) &= \frac{1 - \gamma^2}{\pi} \frac{p_1p_2}{p_1^2 + 2 p_1p_2 \gamma \sqrt{1-\gamma^2} + p_2^2 (1-\gamma^2)} + o(1)\\
\VHat_{22}(p) &= \frac{1 - \gamma^2}{\pi} \frac{p_2^2}{p_1^2 + 2 p_1p_2 \gamma \sqrt{1-\gamma^2} + p_2^2 (1-\gamma^2)} + o(1).
\end{align}
\end{proposition}
From this, \eqref{eq:compressibility} immediately follows. 
Finally, \eqref{eq:U} is a consequence of \cite[Theorem 1.10]{Johnston23_ContinuousKasteleyn}, that implies
\begin{eqnarray}
  \bfE_\gamma[\nabla \ch(x)]=(H_o(0,0)-\frac12,H_b(0,0)),
\end{eqnarray}
together with \eqref{eq:71}.

\section*{Statement of AI usage}

AI was used only used to generate figures and provide bibliographic references.

\section*{Acknowledgments}
G.~C. gratefully acknowledges financial support
via the UKRI FL fellowship ``Large-scale universal behaviour of Random
Interfaces and Stochastic Operators'' MR/W008246/1. The research of H.~G. and F.~T. was funded 
by the Austrian Science Fund (FWF) 10.55776/F1002. 
For open access purposes, the authors have applied a CC BY public copyright license to any author accepted manuscript version arising from this submission.

\end{document}